\documentclass[article]{amsart}
 \usepackage{amssymb,hyperref}
\usepackage{amsbsy}
\usepackage{amsmath}
\usepackage{amstext}
\usepackage{amsthm}   
\usepackage{amscd}
\usepackage{graphicx}
\usepackage{amssymb}
\usepackage{hyperref}\hypersetup{linktocpage, hidelinks}
\usepackage{mathtools}
\usepackage{tikz}
\usepackage{tikz-cd}
\usepackage[all]{xy}
\usepackage{caption} 
\usetikzlibrary{arrows}

\newcommand{\tcb}{\textcolor{blue}}

\makeatletter
\newcommand*{\rom}[1]{\expandafter\@slowromancap\romannumeral #1@}
\makeatother

\theoremstyle{plain}
\newtheorem{thm}{Theorem}[section]

\newtheorem{lem}[thm]{Lemma}
\newtheorem{prop}[thm]{Proposition}

\theoremstyle{definition}
\newtheorem{defn}[thm]{Definition}

\theoremstyle{remark}
\newtheorem{rem}[thm]{Remark}
\newtheorem{expl}[thm]{Example}

\numberwithin{equation}{section}

\newcommand\numberthis{\addtocounter{equation}{1}\tag{\theequation}}

\newcommand\GL{{\mathrm {GL}}}

\newcommand{\ba}{{\boldsymbol \alpha}}
\newcommand{\bb}{{\boldsymbol \beta}}

\def\Q{\mathbb{Q}}
\def\R{\mathbb{R}}
\def\Z{\mathbb{Z}}
\def\N{\mathbb{N}}
\def\C{\mathbb{C}}

\newcommand{\Eb}{{\mathbb E}}

\newcommand{\Nc}{{\mathcal N}}

\newcommand\Pbb{{\mathbb P}}
\newcommand\Vb{{\mathbb V}}

\newcommand{\Ac}{{\mathcal A}}

\newcommand{\Hc}{{\mathcal H}}
\newcommand{\Lc}{{\mathcal L}}

\newcommand{\Oc}{{\mathcal O}}

\newcommand\Pc{{{\mathcal P}}}

\newcommand{\al}{{\alpha}}

\newcommand{\ra}{\rightarrow}

\newcommand{\Leb}{\mathrm{Leb}}

\newcommand{\Lt}{{\widetilde \Lc}}

\newcommand{\twobytwo}[4]{\left[ \begin{matrix} #1 & #2 \\ #3 & #4
\end{matrix} \right]}

\newcommand{\hide}[1]{{}}

\newcommand{\twobytwotiny}[4]{{\left[\begin{smallmatrix} #1 & #2 \\ #3 & #4
\end{smallmatrix}\right]}}
\newcommand{\rdot}{\!\cdot\!}

\theoremstyle{plain}
\newtheorem{thmL}{Theorem}
\newtheorem{propL}[thmL]{Proposition}

\begin{document}

\title{Euclidean algorithms are Gaussian over imaginary quadratic fields}

\author{Dohyeong Kim}
\author{Jungwon Lee}
\author{Seonhee Lim}

\address{Department of Mathematical Sciences and Research Institute of Mathematics, Seoul National University, Gwanakro~1, Gwanak-ku, Seoul 08826, Republic of Korea}
\email{dohyeongkim@snu.ac.kr}
\email{slim@snu.ac.kr}

\address{Mathematics Institute, University of Warwick, Coventry CV4 7AL, UK}
\curraddr{Max Planck Institute for Mathematics, Vivatsgasse 7, 53111 Bonn, Germany}
\email{jungwon.lee@warwick.ac.uk; jungwon@mpim-bonn.mpg.de}

\date{\today. }

\begin{abstract}
We prove that the distribution of the number of steps of the Euclidean algorithm of rationals in imaginary quadratic fields with denominators bounded by $N$ is asymptotically Gaussian as $N$ goes to infinity, extending a result by
Baladi and Vall\'{e}e for the real case. 
The proof is based on the spectral analysis of the transfer operator associated to the nearest integer complex continued fraction map, which is piecewise analytic and expanding but not a full branch map. By observing a finite Markov partition with a regular CW-structure, which enables us to associate the transfer operator acting on a direct sum of spaces of $C^1$-functions, we obtain the limit Gaussian distribution as well as residual equidistribution.
\end{abstract}

\maketitle
\setcounter{tocdepth}{1}

\section{Introduction}
The Euclidean algorithm is one of the oldest algorithms which remains useful, for example in computer algebra systems and multi-precision arithmetic libraries. The algorithm can be recorded via \textit{continued fractions}:
the Euclidean algorithm for natural numbers $a,b$ with $a<b$ gives rise to the sequence $(a_i)_{i=1, \cdots, \ell}$ of positive integers satisfying
\[ b = a_1 a + r_1, \quad a = a_2 r_1 + r_2, \quad  \dots, \quad r_{\ell-2} = a_\ell r_{\ell-1} + r_\ell,
 \] 
where $0 \le r_{j+1}<r_j$ holds for all $j$, and the terms of the sequence are exactly the \textit{digits} (also called the partial quotients) of the continued fraction expansion of the rational 
\begin{equation}\label{eqn:1}
    \frac{a}{b} = \frac{1}{a_1 + \frac{1}{a_2 + \frac{1}{\ddots + \frac{1}{a_\ell}}}}=:[a_1, a_2, \cdots, a_\ell].
\end{equation}
Note that the algorithm terminates upon the condition $r_\ell=0$ and that the requirement $0\le r_{\ell-2}<r_{\ell-1}$ implies that $a_\ell \geq 2$.


For a rational $x \in [0,1]\cap \mathbb{Q}$, we let $\ell(x) :=\ell$ from \eqref{eqn:1} and call it the \emph{length of the continued fraction (CF) of $x$}.
Baladi and Vall\'ee showed that the length $\ell(x)$ of CF of the rationals with bounded denominator follows asymptotically the Gaussian law. They further showed the same asymptotic distribution for the total cost $C(x):=\sum_{j=1}^{\ell(x)} c(a_j)$ (for $x=[a_1, \cdots a_{\ell(x)}]$) of a \emph{digit cost} function $c: \N \ra \R_{\geq 0}$ with moderate growth. For the length function $C=\ell,$ the digit cost is $c\equiv 1.$
\begin{thm}[Baladi--Vallée {\cite[Theorem\,3]{bal:val}}] 
\label{thm:bv} For a digit cost $c$ with a moderate growth $c(a)=O(\log a)$,
the distribution of the \textit{total cost} $C$ on the set \[ \Omega_{\R,N}:= \left\{ \frac{a}{b} : 1 \leq a < b \leq N, (a,b)=1   \right\} \] is asymptotically Gaussian, with the speed of convergence $O(1/\sqrt{\log N})$ as $N \to \infty.$ 
\end{thm}


\subsection{Main results}\label{sec:1.1}
The aim of this article is to generalise the above result and techniques to $\C$. We want to replace $\Z$ with a discrete ring $\Oc \subset \C$ for which the field of fractions is Euclidean: it is one of the imaginary quadratic fields $K_d :=\mathbb Q (\sqrt{-d})$ for $d=1,2,3,7,11.$ For such $K_d$ and its ring of integers $\Oc_d$, one can find a strict fundamental domain $I_d' \subset \{z \in \C: |z| \leq 1 \} $ of the translation action of $\mathcal O_d$ on $\mathbb C$ (see Definition~\ref{def:2.1}). The main example is $K_1 = \Q(i), \mathcal O_1 = \Z [i]$ and 
$$I_1' = \{ z \in \C: -1/2 \leq \mathrm{Re}(z), \mathrm{Im}(z) < 1/2, \; \mathrm{or}\; z=(1-i)/2\}.$$

For $z \in \mathbb{C}$, by defining $[z]$ to be the unique element of $\Oc_d$ such that  $z -[ z] \in I_d'$, the continued fraction expansion is unique on the closure $I_d:=\overline{I_d'}$ and the map
\begin{equation}\label{def:compgauss}
T_d : I_d \to I_d, \quad T_d(z) = \frac{1}{z} - \left[\frac{1}{z} \right], \text{if $z\neq 0$ and }  T_d(0)=0,
\end{equation}
is well-defined. This map, an analogue of the Gauss map, is called Hurwitz continued fraction map or the \textit{nearest integer} complex continued fraction map. It was first introduced by A. Hurwitz \cite{hurwitz} for $d=1,3$ and has been studied by Lakein \cite{lakein} in a wider context, by Ei--Nakada--Natsui \cite{Nakada:etal2} for certain ergodic properties, and by Hensley \cite{hensley} and Nakada \textit{et al.} \cite{Nakada:etal, Nakada:pre, Nakada} for Kuzmin-type theorem. More recently, Bugeaud--Robert--Hussain \cite{bugeaud} established the metrical theory of Hurwitz continued fractions towards the complex Diophantine approximations.

Here, we present a dynamical framework for the statistical study of $K$-rational trajectories based on the transfer operator methods. 
The lengths of such trajectories have been investigated in the literature.
For example, an upper bound of the length of the continued fraction expansions for Gaussian integers was obtained in \cite{R1, R2}.
The goal of the present paper is to establish the Baladi--Vallée-type limit theorems as described in Theorem \ref{thm:bv} for complex continued fractions.

From now on, we fix $d \in \{1,2,3,7,11\}$ and suppress it from the notation, unless otherwise stated. As in the real case, the digits $\al_j$ of the nearest continued fraction expansion $z=[\al_1, \al_2, \ldots, \al_n, \ldots]$ are obtained by $\alpha_j=\left[ \frac{1}{T^{j-1}(z)} \right],$
and the expansion terminates in a finite step $\ell(z)$ if $z \in I \cap K$. For a digit cost function $c: \Oc \to \R_{\geq 0},$ define its total cost as
\begin{align}\label{eq:totalcost}
C(z):=\sum_{j=1}^{\ell(z)} c(\al_j) \mathrm{\;for\;} z = [\al_1, \cdots \al_{\ell(z)}] \in I \cap K. 
\end{align}

On the other hand, $c$ induces a function $f_c$ on $I-\{0\}$ in the following sense.
If $z \in I$ has the continued fraction expansion $z=[\alpha_1,\cdots]$, we define $f_c(z)=c(\alpha_1)$. 
In what follows, by abuse of notation, we put $c(z):=f_c(z)$ for $z \in I$. 
This should not cause confusion since the domain of $f_c$, which is $I-\{0\}$, is disjoint from $\mathcal O$.

Our first result is the following theorem on the asymptotic distribution of the \textit{cost $ C_n$ up to $n$} on the whole space $I$ (rather than just rationals $I\cap K$) which is defined by 
$$C_n(z):=\sum_{j=1}^n c(\al_j) \quad \mathrm{for} \;\; z=[\al_1, \al_2, \ldots, \al_n, \ldots] \in I.$$

\begin{thmL}[Theorem \ref{clt:cont}] \label{intro:clt1}
Let $c: \Oc \to \mathbb R_{\geq 0}$ be a digit cost function of moderate growth $c(\al) = O(\log |\al|),$ which is not cohomologous to zero. 
For any $u \in \R$, 
the distribution of $C_n$ in $I$ is asymptotically Gaussian;
\[\Pbb \left[\frac{C_n - \widehat{\mu}(c) n}{\widehat{\delta}(c) \sqrt{n}}\leq u \right]=\frac{1}{\sqrt{2\pi}}\int_{-\infty}^u e^{-\frac{t^2}{2}}dt+O\left(\frac{1}{\sqrt n}\right)  \]
as $n \to \infty,$ where $\Pbb$ denotes a probability measure  on $I$ with $C^1$-invariant density.

The expectation and variance satisfy
\begin{align*}
\Eb[C_n] &= \widehat{\mu}(c) n+\widehat{\mu}_1(c)+O(\theta^n) \\
\Vb[C_n] &= \widehat{\delta}(c) n+\widehat{\delta}_1(c)+O(\theta^n)
\end{align*}
for some $\theta<1$ and real constants $\widehat{\mu}_1(c)>0$ and $\widehat{\delta}_1(c)>0$.
\end{thmL}

The next theorem is our main result, which is the analogue of Theorem 1.1. 


\begin{thmL}[Theorem \ref{clt:discrete}]  \label{intro:clt2} 
Let $c: \Oc \to \R_{\geq 0}$ be bounded and  not cohomologous to zero.
The distribution of the total cost $C$  on 
$$\Omega_N := \left\{\frac{a}{b} \in I: |b|^2 <N \right\}$$
is asymptotically Gaussian, i.e. there exist real numbers $\mu(c), \delta(c)>0$ such that for any $u \in \R$,
\[\Pbb_N \left[\frac{C - \mu(c) \log N}{\delta(c) \sqrt{\log N}}\leq u \right]=\frac{1}{\sqrt{2\pi}}\int_{-\infty}^u e^{-\frac{t^2}{2}}dt+O\left(\frac{1}{\sqrt{\log N}}\right)  \]
as $N \to \infty$, where $\Pbb_N$ denotes the uniform probability measure on $\Omega_N$.

The expectation and variance satisfy
\begin{align*}
\Eb_N[C|\Omega_N] &= \mu(c)\log N+\mu_1(c)+O(N^{-\gamma}) \\
\Vb_N[C|\Omega_N] &= \delta(c) \log N+\delta_1(c)+O(N^{-\gamma})
\end{align*}
for some $\mu_1(c), \delta_1(c),$ and $\gamma>0$.
\end{thmL}

We present another consequence, namely the residual equidistribution mod $q$, extending a result of Lee--Sun \cite{lee:sun} for the real case. 
\begin{thmL} [Theorem \ref{thm:equid}] \label{intro:equid}
Let $c: \Oc \to \Z_{\geq 0}$ be bounded and not cohomologous to zero and let $q \in \N$. The values of $C$ modulo $q$ are equidistributed on $\Omega_N$ as $N \to \infty$, i.e., for any $a \in \Z/q\Z$, 
\[ \Pbb_N[C \equiv a \ (\mathrm{ mod } \ q) |\Omega_N]=q^{-1}+o(1) .  \]
\end{thmL}




\begin{rem}
The motivation behind our work is to extend the dynamical approach to the statistical study of modular symbols and twisted $L$-values formulated in Lee--Sun \cite{lee:sun} and Bettin--Drappeau \cite{bet:dra} to imaginary quadratic fields. However, as the period (a normalisation to make the values integral) is not known yet, we can obtain results only on the cohomology level, for instance, an alternative proof on the normal distribution and residual equidistribution of Bianchi modular symbols in hyperbolic 3-space by Constantinescu and Nordentoft \cite{con, con:nor}. 
\end{rem}

\begin{rem} \label{rmk:growth}
Note that the boundedness assumption on $c$ in Theorem \ref{intro:clt2} and \ref{intro:equid} can be relaxed to moderate growth. The boundedness of $c$ is satisfied for our major interest and it simplifies the proofs by allowing us to use the truncated Perron formula Theorem~\ref{perron} and deduce $C(z)=O(\log N)$ for $z \in \Omega_N$. The moderate growth condition is sufficient if we use the Perron formula without truncation in conjunction with the smoothing process \cite[\textsection\,4]{bal:val}, which gives $C(z)=O(\log ^2 N)$. 
\end{rem}

\subsection{Complex continued fraction maps and its transfer operator}\label{sec:1.2} The proofs of the main results above are mainly inherited from the strategy of \cite{bal:val}: our approach is based on dynamical analysis of the complex continued fraction map, an extensive use of the weighted transfer operator of the system, and a choice of an appropriate function space where the transfer operator has good spectral properties, a relation between Dirichlet series and the transfer operator, and the connection between moment generating function and the Gaussian behavior.

Specifically, $T_d$ in \eqref{def:compgauss} corresponds to $T$ in \cite[\textsection\,1]{bal:val}, the weighted transfer operator $\mathcal L_{s,w}$ of Definition~\ref{def:1.5} to \cite[(2.6)]{bal:val}, the space $C^1(\mathcal P)$ of Definition~\ref{def:functionspace} to $\mathcal C^1(\mathcal I)$ of \cite[\textsection\,2.2]{bal:val}, the expression \eqref{ex:dirichlet} to \cite[(2.17)]{bal:val}.
Theorem~\ref{thm:hwang}, which connects the moment generating function and the Gaussian behavior, is the same as \cite[Theorem 0]{bal:val}.
Despite of the similarities between the strategies, however, the novelties of the present paper include a crucial adaptation of introducing the new function space $C^1(\mathcal P)$ and a technical adaptation involving an extended two-dimensional Van Der Corput Lemma. Let us explain the new function space. 

Theorem~\ref{intro:clt1} and \ref{intro:clt2} are central limit theorems for $(I,T)$, which we call the complex Gauss dynamical system. To obtain limit theorems via thermodynamic formalism, we need a Banach space containing $C^\infty(I)$ and stable under the transfer operator of $T,$ which is, in the simplest case, of the form
\begin{align*}
    \mathcal L_{1,0} f(z)&=\sum_{z_0 \in T^{-1} (z)}1/|J_T (z_0)|  f(z_0) = \sum_{\alpha \in \Oc} \left| h_{\alpha}' (z)\right|^{2} 1_{TO_\al} (z)f(h_\al (z)).
\end{align*}
Here, $J_T$ is the Jacobian determinant of $T$ (as $\mathbb{R}^2$-valued function), $\alpha$ is the first digit of $z_0$,
$$O_\al=\{z \in I: [1/z]=\al \},$$ 
$h'_\alpha$ is the complex derivative of the inverse $h_\al$ of $T$ restricted to $O_\al$, and $1_{TO_\al}$ is the characteristic function on $TO_\al$. As the characteristic functions on $TO_\al$ appear in the transfer operator, our Banach space must contain these functions and, in particular, properly contain $C^1(I).$


Such a function space will be constructed using a theorem of Ei--Nakada--Natsui \cite{Nakada:pre}, which provides some sort of
Markovian structure on $I$ (see Proposition~\ref{part:markov}), from which we show that $I$ is a cell complex with cells in a finite partition $\Pc$ of $I$ such that any $TO_\alpha$
is a union of cells and that the space $C^1(\Pc)$ of (roughly speaking) piecewise $C^1$-functions is stable under the transfer operator. See Section~\ref{section:functionspace} for the precise definition and suppose for now that we are given $C^1(\Pc).$

\begin{rem}
We emphasize that our main theorem, Theorem~\ref{intro:clt2} is about rationals, i.e. \emph{finite trajectories} of the dynamical systems, whence one cannot ignore cells of dimension zero and one
despite of the fact that their measure is zero for the relevant measure $\mathbb P$ in the main theorems. This contrasts the common practice in ergodic theory where sets of measure zero may be ignored. 

We would like to point out that the partition $\mathcal V$ given in \cite{Nakada:pre} is a metric partition whose elements are exactly the elements of $\mathcal P$ of dimension 2, where as our partition $\mathcal P$ is a set partition, i.e. the union of elements of the partition $\mathcal P$ is the whole domain $I_d.$
We also remark that for the purposes of \cite{Nakada:pre} it was enough for the authors to consider the 2-dimensional cells of $\mathcal P$ but not those of lower dimensions, although the pictures in \cite{Nakada:pre} well visualize the cells in all dimensions.

For our purpose, it is important to work with a Banach space whose elements are honest functions and can be evaluated at all points of $I$, because the resolvent trick, the same idea as \cite[(2.17)]{bal:val}, in \textsection\,\ref{subsec:resolvent} uses the value of the iterates of the transfer operator applied to the characteristic function on $I$ at the origin.
We will construct such a Banach space by introducing suitable norms on $C^1(\Pc)$.
The norm is defined on $C^1(\Pc)$ rather than a quotient of it, we are able to insist that the Banach space consists of functions.

As our partition $\mathcal P$ is different from $\mathcal V$ from \cite{Nakada:pre}, we cannot just cite their results as the existence of some stable metric partition, but has to go through the construction of the partition and prove nice properties of the partition. In particular, our main Markovian property in Proposition \ref{part:markov} and its proof are not stated in \cite{Nakada:pre}. We explicitly prove it using the stable property of the partition $\mathcal V.$

\end{rem}

\hide{
Then for suitable positive constants $\widehat{\mu}(c)$ and $\widehat{\delta}(c)$, 
\[\Pbb \left[\frac{C_n - \widehat{\mu}(c) n}{\widehat{\delta}(c) \sqrt{n}}\leq u \right]=\frac{1}{\sqrt{2\pi}}\int_{-\infty}^u e^{-\frac{t^2}{2}}dt+O\left(\frac{1}{\sqrt n}\right)  \]
for any $n \geq 1$ and $u \in \R$.
}




With $C^1(\Pc)$ in hand, let us now formally define the weighted transfer operators. 
Regarding $h_\alpha$ as an $\mathbb R^2$-valued function, the Cauchy--Riemann equation implies that its Jacobian determinant $J_\al$ satisfies
\begin{equation}  \label{der:jac}
J_\alpha(z)
:=J_{h_\al}(z) 
= \left| h_\alpha' (z)\right|^2 =\frac{1}{|z+ \al |^4}>0. 
\end{equation} 

\hide{--------
\begin{prop}
For $z_0=x_0+y_0 i \in TO_\alpha$, we have  $J_\alpha(x_0+iy_0) = \left| h_\alpha' (z_0)\right|^2$.
\end{prop}
\begin{proof}
It is a consequence of the Cauchy-Riemann equation.
\end{proof}

It depends on two complex parameters $s$ and $w$ as the notation suggests, and also on a cost function $c$ which is suppressed in the notation. 
\begin{defn}

A cost function is a function on the set $\{\alpha \in \mathcal O \colon \text{$O_\alpha$ is non-empty}\}$ with values in the set of non-negative real numbers.
\end{defn}

Recall that $h_\alpha \colon TO_\alpha \to O_\alpha$ is an inverse branch of $T$.
Viewing $TO_\alpha$ and $O_\alpha$ as subsets of $\mathbb C$,  $h_\alpha$ is holomorphic in the interior of $TO_\alpha$, and it extends holomorphically to an open neighbourhood of $TO_\alpha$. Denote by $h_\alpha'$ its holomorphic derivative, defined on $TO_\alpha$.
Under the identification $\mathbb C \simeq \mathbb R^2$ via $z=x+iy$, we regard $h_\alpha$ as an $\mathbb R^2$-valued function on an open subset of $\mathbb R^2$. It is differentiable and we write $J_\alpha$ for its Jacobian.
It is well-known that $J_\alpha$ is determined by the holomorphic derivative of $h_\alpha$.

\begin{prop}
For $z_0=x_0+y_0 i \in TO_\alpha$, we have  $J_\alpha(x_0+iy_0) = \left| h_\alpha' (z_0)\right|^2$.
\end{prop}
\begin{proof}
It is a consequence of the Cauchy-Riemann equation.
\end{proof}

-----------}


We recall that for a digit cost function $c:\Oc \to \R_{\geq0}$, by abuse of notation, we denote the induced function again by $c$; see the paragraph below \eqref{eq:totalcost}.

\begin{defn}\label{def:1.5}
Let $s, w \in \mathbb{C}$.
The transfer operator $\Lc_{s,w}: C^1(\Pc) \rightarrow C^1(\Pc)$ of the map $T$ associated with $c$ is defined by
\begin{align*}
    \mathcal L_{s,w} f(z)&:=\sum_{z_0 \in T^{-1}(z)} g_{s,w}(z_0)  f(z_0) \\
    & =\sum_{\al \in \Oc}  (g_{s,w} \cdot f) \circ h_\al(z) \cdot 1_{TO_\al}(z),
    \numberthis\label{eq:Lsw}
\end{align*}
where $g_{s,w}(z):= \exp(w c(z))  J_{[z^{-1}]}(T(z))^s.$

\end{defn}

\begin{rem}
The finite CW-structure of $\Pc = \cup_{i=0}^2 \Pc[i]$ induces a decomposition of the function space $C^1(\Pc)=\bigoplus_{i=1}^2 C^1(\Pc[i])$ and of the operator $\Lc:=\Lc_{s,w}$ as a lower-triangular matrix 
\[ \Lc=\begin{bmatrix}
\Lc_{[2]}^{[2]} & 0 & 0 \\
\Lc_{[2]}^{[1]} & \Lc_{[1]}^{[1]} & 0 \\
\Lc_{[2]}^{[0]} & \Lc_{[1]}^{[0]} & \Lc_{[0]}^{[0]} 
\end{bmatrix} \]
where $\Lc^{[i]}_{[j]}: C^1(\Pc[i]) \rightarrow C^1(\Pc[j])$ with $0 \leq i,j \leq 2$ is the component operator.

We remark that this triangular form plays a prominent role in the proof of following Theorem \ref{intro:spectrum} and displays the technical issues in generalising the transfer operator methods for finite trajectories in higher dimensions.
\end{rem}

    \hide{-----
    The term corresponding to a given pre-image $\tilde z \in T^{-1}(z)$ is equal to $\exp\left(w c\left(\alpha\right)\right) \cdot J_{\alpha}(z)^s \cdot f(h_\alpha(z))$ with $\alpha = [\tilde z^{-1}]$, so \eqref{eq:L} is equal to the sum of $\exp\left(w c\left(\alpha\right)\right) \cdot J_{\alpha}^s \cdot \left( f\circ h \right) \cdot 1_{TO_\alpha}$ over $\alpha$'s, where $1_{TO_\alpha}$ is the characteristic function on $TO_\alpha$. 
    -------}



\hide{
    To proceed, we settle a few notations.
    For a subset $P \in \mathcal P$ of $TO_\alpha$, denote the restriction of $h_\alpha$ to $P$ by $\langle \alpha \rangle ^P_Q \colon P \to Q$ if $h_\alpha(P) \subset Q$.
    For $P,Q \in \mathcal P$, set
    \begin{align*}
    \mathcal H(P,Q) = \left\{h\colon P \to Q  : \text{ $h= \langle \alpha \rangle ^P_Q$  for some $\alpha \in \mathcal O$}\right\}
    \end{align*}
    to be the collection of restricted inverse branch maps from $P$ to $Q$.
    Then \eqref{eq:Lsw} becomes, for $z \in P$ 
    \begin{align} \label{def:op:eachp}
    \left(\mathcal L_{s,w}f\right)_P(z) = \sum_{Q \in \Pc} \sum_{\langle \alpha \rangle ^P_Q \in \mathcal H(P,Q)} g_{s,w}(z) \cdot \left( f_Q \circ \langle \alpha \rangle ^P_Q \right) (z) .
    \end{align}
    This shows that if $ \Pc$ is compatible with $T$ and \eqref{def:op:eachp} is convergent for each $P$, then the operator $\mathcal L_{s,w}$ preserves $C^1(\mathcal P)$.
    To ensure the convergence, we take the moderate growth assumption on the digit cost function $c$; see \eqref{def:modgrowth}.
}

\hide{-----

Put $i_P({z_0}) = \left(Q,\langle \alpha \rangle ^P_Q\right)$.
Temporarily, define $\mathcal H_P$ to be the disjoint union $\coprod_{Q \in \mathcal P} \mathcal H(P,Q)$ taken over the set $\mathcal P$. Then by the compatibility of $\Pc$ and $T$, $h$ is uniquely represented by a pair $(Q,\langle \alpha \rangle ^P_Q)$ where $Q \in \mathcal P $ is the unique cell with $h \in \mathcal H(P,Q)$. For $z \in P$ and $z_0 \in T^{-1}(z)$ fixed, $\alpha = [z_0^{-1}]$ is the unique element of $\Oc$ such that $h_\alpha (z) = z_0$ by definition. Let $Q \in \mathcal P$ be the unique cell with $z_0 \in Q$.
Then it follows that $h_\alpha(P) \subset Q$, i.e. $\langle \alpha \rangle ^P_Q \in \mathcal H(P,Q).$ 

Conversely, for a fixed $Q$ and $h =\langle \alpha \rangle ^P_Q \in \mathcal H(P,Q)$, there exists a unique $z_0 \in T^{-1}(z)$ such that $h_\alpha(z) = z_0$, namely
$$z_0 = \frac{1}{z+ \alpha}.$$
Thus \eqref{eq:Lsw} becomes, for $z \in P$, 
\begin{align} \label{def:op:eachp}
\left(\mathcal L_{s,w}f\right)_P(z) = \sum_{Q \in \Pc} \sum_{\langle \alpha \rangle ^P_Q \in \mathcal H(P,Q)} g_{s,w}(z) \cdot \left( f_Q \circ \langle \alpha \rangle ^P_Q \right) (z) .
\end{align}
This shows that if $ \Pc$ is compatible with $T$ and \eqref{def:op:eachp} is convergent for each $P$, then the operator $\mathcal L_{s,w}$ preserves $C^1(\mathcal P)$.
To ensure the convergence, we will assume the moderate growth condition on the cost function $c$; see Definition\,\ref{def:modgrowth}.


------}

The function space $C^1(\mathcal P)$ is a Banach space with respect to a family of norms $\{ \| \cdot \|_{(t)} \}_{t \in \R \backslash \{0\}}$ defined by
\begin{align} \label{intro:norms}
  \|f\|_{(t)} = \|f\|_0 + \frac{1}{|t|}\|f\|_1,
\end{align}
where $\| \rdot \|_0$ is essentially the sup-norm and $\| \rdot \|_1$ is a semi-norm (see \S\ref{sub:norms}).

We establish the following key spectral properties, namely the Ruelle--Perron--Frobenius Theorem and Dolgopyat-type uniform estimate.  
We note that the spectral properties stated below are almost exactly the same as \cite[Thm.\,2]{bal:val} except for the explicit and implied absolute constants in the statements.
Denote $s=\sigma+it$ and $w=u+i\tau$, with $\sigma,t,u,\tau \in \mathbb R$.

\begin{thmL}[Theorem \ref{thm:ruelle} and \ref{main:dolgopyat}] \label{intro:spectrum} Consider the operator $\Lc_{s,w}$ on $C^1(\Pc).$
\begin{enumerate}
\item For $(s,w)$ near $(1,0)$, the operator $\Lc_{s,w}$ has an eigenvalue $\lambda_{s,w}$ of maximal modulus and there are no other eigenvalues on the circle of radius $|\lambda_{s,w}|$, and $\lambda_{s,w}$ is algebraically simple.
\item Let $(s,w)$ with $(\sigma,u)$ near $(1,0)$. For $0<\xi<1/10$ and sufficiently large $|t|$, 
\[ \| (I-\Lc_{s,w})^{-1} \|_{(t)} \ll |t|^\xi . \]
Here, $f(x) \ll g(x)$ means $f(x)=O(g(x))$ and the implied constant is determined by a neighbourhood of $(1,0)$ in $\R \times \R$ on which $(\sigma, u)$ belongs.
\end{enumerate}
\end{thmL}

To prove the above theorem, new ingredients compared to Baladi--Vall\'ee \cite{bal:val} are needed due to technical difficulties that arise from higher dimensional nature of complex continued fractions. In particular, our argument relies on the analysis due to Ei--Nakada--Natsui \cite{Nakada:pre} of the natural invertible extension and the dual system of $(I,T)$ as well as a 2-dimensional version of Van der Corput Lemma. See \S\ref{sec:dolgopyat} for details.

Once the aforementioned technical difficulties have been overcome, one can establish spectral properties and use them for our purposes.
To do so, we closely follow the logical structure of \cite{bal:val}.
That is, we consider an auxiliary complex Dirichlet series which can be written in terms of the resolvent of the operator $\Lc_{s,w}$. Accordingly, Theorem~\ref{intro:spectrum} is translated into necessary analytic properties of the Dirichlet series that allow us to apply a Tauberian argument and obtain the estimates of moment generating functions, hence the limit laws for complex continued fractions.


Given the similarities between our proofs and those in \cite{bal:val}, here we highlight the differences and the associated difficulties.
A single major difference lies in the fact that $(I,T)$ is not Bernoulli: observe that there exists $\alpha$ such that $O_\alpha$ contains an open but non-dense subset of $I$.
This difference gives rise to difficulties that were not present in \cite{bal:val}.
First, we need to replace the function space since $C^1(I)$ is not invariant under the transfer operator.  
Clues to solve the problem are found in \cite{Nakada:etal2, Nakada:pre}, but for our purposes we need to analyze the cell structure that were irrelevant therein and to establish functional analytic results -- completeness of a suitable family of norms, quasi-compactness of the transfer operator and distortion properties. 
Once such preliminary ingredients have been established, a key remaining piece is the analogue of the Dolgopyat estimate in \cite{bal:val}.

Here, a new difficulty arises because $(I,T)$ is not Bernoulli. Specifically, 
one needs a explicit description of the so-called dual algorithm.
For our systems $(I,T)$, it is not at all clear whether dual systems admit an elementary description unliike \cite{bal:val}. Indeed, it is one of the main results of \cite{Nakada:pre} where the so-called natural extensions are constructed using the ergodicity of geodesic flows. Our argument for the Dolgopyat estimate relies on metric properties of the natural extensions of \cite{Nakada:pre} with a few modifications related to the presence of lower dimensional cells.

This article is organised as follows. In \S\ref{sec:dynamics}, we study expansion and distortion properties of the complex Gauss dynamical system. In \S\ref{sec:markov}, we introduce a finite partition $\Pc$ of $I$ and a finite Markov partition compatible with the countable inverse branches using the work of \cite{Nakada:pre}. In \S\ref{sec:acim}, we show quasi-compactness and a spectral gap of the transfer operator $\Lc_{s,w}$ acting on $C^1(\Pc)$.   In \S\ref{sec:nor:lasota}, we settle a priori bounds for the normalised family of operator which will be used in \S\ref{sec:dolgopyat}, where we have Dolgopyat estimate. We obtain Theorems \ref{intro:clt1} and \ref{intro:clt2} in \S\ref{sec:clt1}-\ref{sec:clt2}, and Theorem~\ref{intro:equid} in \S\ref{sec:equid}.



\vspace{.15 in}

\noindent \textit{Acknowledgements.}
We are grateful to Hans Henrik Rugh for the idea of Lemma~\ref{vandercorput}. We also thank Hitoshi Nakada, Hiromi Ei, Rie Natsui for helpful discussion and sharing the preprint with us, and Malo Jézéquel for clarifying comments.

This work is supported by (DK) Korea NRF-2020R1C1C1A01006819, Samsung Science and Technology Foundation SSTF-BA2001-01, (JL) ERC-Advanced Grant 833802 Resonances, (SL) Korea NRF RS-2025-00515082, and RS-2023-00301976. SL is an associate member of Korea Institute for Advanced Study.

\section{Complex Gauss dynamical system} \label{sec:dynamics}

From now on, we call the nearest integer continued fraction map $T$ as the CF map.
In this section, we show uniform expanding and distortion properties of the CF map, which will be crucially used later for spectral analysis.

\subsection{Metric properties of inverse branch} \label{sec:branch} Let us start with an explicit description of $I_d$ promised in Section~\ref{sec:1.1}.
For a fixed $d\in \{1,2,3,7,11\}$, let $K= \mathbb Q(\sqrt {-d})$.
Its ring $\Oc$ of integers, which is a lattice in $\mathbb C$, is of the form
\begin{align} \label{ring:integer}
    \Oc &= 
    \begin{cases}
    \Z[\sqrt{-d}]&\text{if $d \not\equiv 3 \ (\mbox{mod } 4)$,}
    \\
     \Z[\frac{1+\sqrt{-d}}{2}]&\text{if $d \equiv 3 \ (\mbox{mod } 4)$.}
    \end{cases}
\end{align}


 
A natural fundamental domain for the translation action of $\Oc$ on $\C$ would contain a connected component of the set $\C$ minus the equidistant lines with respect to two points in $\Oc$. We choose a component containing the origin and further make a choice on the boundary as follows for the strict fundamental domain $I_d'$.
\begin{defn}\label{def:2.1} For $d=1,2,$ we choose
rectangles
\begin{align*}
 I_{d} &:= \left\{ x + iy : |x| \le \frac 1 2, \  |y| \le \frac{\sqrt d}{2}\right \}, \quad \quad
    I_{d}' := I_{d} - \bigcup_{\al=1,\sqrt{-d}} I_d+\alpha,
\end{align*}
and for $d=3,7,11,$ we choose hexagons
$$
I_{d} := \left\{ x + iy : |x| \le \frac 1 2, \  \left|y \pm \frac {x}{\sqrt d}\right| \le \frac{d+1}{4\sqrt d} \right\}, \quad 
I_{d}' := I_{d} - \bigcup_{\al=1, \frac{1\pm\sqrt{-d}}{2}}I_{d}+\alpha.
$$
\end{defn}

Before defining inverse branches, let us look into the ``cylinder sets" $\Oc_\mathbf{\ba}$ of $T$, namely the sets of $z \in I:=I_d$ with fixed $n$ first digits, for some $n \in \N.$ For $n=1$, they are the sets
\begin{equation} \label{part:Oalpha}
O_\alpha =  \left \{ z \in I  \colon [1/ z ] = \alpha \right\}
\end{equation}
for $\al \in \Oc$, mentioned in Section~\ref{sec:1.2}.
The set $O_\al$ is empty for finitely many $\alpha$'s, namely those for which $I+\alpha$ is disjoint with the image of $I$ under the inversion. See Table~\ref{table:empty} for the complete list. Similarly, there are finitely many $\al$'s such that $T O_\al \neq I$, namely those for which $I+\alpha$ is not contained in the image of $I$ under the inversion. See Figure \ref{fig:d=3} for the case of $d=3$. 
Non-empty $O_\al$'s form a partition of $I$ such that $T|_{O_\al}: O_\al \ra TO_\al$ given by $z \mapsto 1/z-\al$ is bijective. 

For $n>1$, for a sequence $\ba = (\alpha_1,\cdots,\alpha_n) \in \Oc^n$, define $O_{\boldsymbol \alpha}$ to be the set of $z$ whose $j$-th digit $\al_j(z)$ equals $\al_j$ for $j \leq n$, in other words
$$O_{(\alpha_1,\cdots,\alpha_{n})}: = \left \{ z\in O_{\alpha_1} \colon T(z) \in O_{(\alpha_2,\cdots,\alpha_{n})} \right \}.$$

\begin{figure}[h] 
   \centering
   \includegraphics[width=1.5 in]{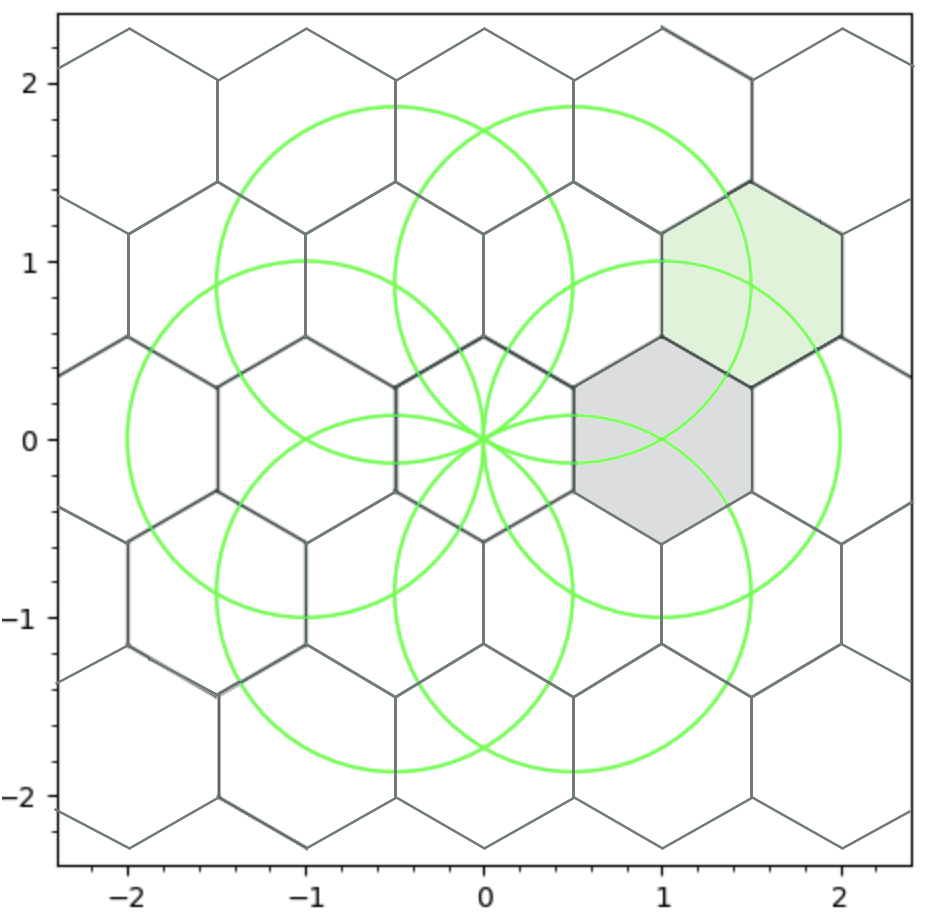} 
   \caption{For $d=3$, the domain $I$ is the hexagon in the center. The inversion $z \mapsto 1/z$ maps $I$ to the region outside the green circles. The grey hexagon $I+1$ lies in the union of the interior of circles, thus $O_1$ is empty. The green hexagon $I+ \frac{3+\sqrt{-3}}{2}$ intersects with some circles, thus $O_{\frac{3+\sqrt{-3}}{2}} \neq I.$ }
   \label{fig:d=3}
\end{figure}

\begin{table}[h]
\centering
\begin{tabular}[h]{|c|c|}
\hline
$d$ &  $\alpha$ 
\\
\hline
$1$&$\pm 1, \pm \sqrt{-1}$
\\
\hline
$2$&$\pm 1
$
\\
\hline
$3$&$\pm 1, \pm \frac{1+\sqrt{-3}}{2}, \pm\frac{1-\sqrt{-3}}{2}$ 
\\
\hline
$7$&$\pm 1$
\\
\hline
$11$&$\pm 1$
\\
\hline
\end{tabular}\caption{The list of $\alpha$'s for which $O_\alpha$ is empty.} \label{table:empty}
\end{table}

\begin{defn} For $\al$ with non-empty $O_\alpha$,
    denote the inverse of $T|_{O_\alpha}$ by $h_\alpha:TO_\alpha \to O_\alpha$, i.e. 
\begin{align*}
h_\al: z& \longmapsto \frac{1}{z+\alpha},
\end{align*} and call it an \textit{inverse branch} (of depth 1) of $T$.
\end{defn}
Note that 
the inverse branch $h_\alpha \colon TO_\alpha \to O_\alpha$ extends holomorphically and uniquely to an open neighbourhood of $TO_\alpha$, since the origin is not a limit point of any non-empty set $O_\alpha$ and $TO_\alpha$ has non-empty interior.

Since our dynamical system $(I,T)$ fails to be a full branch map, i.e. $TO_\alpha \neq I'$ for some non-empty $O_\alpha$, our analysis involves additional steps compared to Baladi--Vallée \cite{bal:val}.
For a sequence $\ba = (\alpha_1,\cdots,\alpha_n) \in \Oc^n$, we call 
the bijection $$h_{\boldsymbol\alpha} := h_{\alpha_1} \circ \cdots \circ h_{\alpha_n}:T^nO_{ \ba} \to O_{\ba}$$
an inverse branch of depth $|\ba|:=n.$
Denote by $J_{\ba}$ the Jacobian determinant of $h_\ba$.
Observe that the inverse branches are conformal and uniformly contracting as follows. 

For $\al \in \mathcal O$ with non-empty $O_\al,$ recall from \eqref{der:jac} that we have
$$|J_\al (z) | = \frac{1}{|z+\al|^4}=|h_\al(z)|^4 \leq R^4,$$
where $R$ is the radius of the ball $\{ z\in \C: |z|<R\}$ containing $I$.
By the chain rule, $|J_{\ba}| \leq R^{4|\ba|}$, and
the \textit{contraction ratio of the inverse branches} 
\begin{align} \label{contrac}
\rho:=\limsup_{n \to \infty} \ \sup_{ | \ba| =n} \sup_{z \in TO_{ \ba}}|J_{ \ba}(z)|^{1/n}
\end{align}
 is at most $R^4$. 

\begin{lem}\label{prop:2.1}
For $I=I_d$, the domain $I \subseteq \mathbb C$ is contained in an open ball centered at zero of radius $R<1$. Consequently, $\rho \leq R^4 <1$.
\end{lem}

Note that $R\le \sqrt{15/16}$ for $d=1,2,3,7,11$, with equality for $d=11.$
Using the lemma, we obtain the following distortion property of inverse branches.

\begin{prop}[Bounded distortion] \label{distortion1}
There is a uniform constant $M>0$ such that for any $n$ and $h_{\ba}= h_{\al_1} \circ \cdots \circ h_{\al_n}$, and any unit tangent vector $v$,
$$| \partial_{v} J_{\ba}(z)| \leq M |J_{\ba}(z)|$$ for all $z \in T^nO_\ba$.
Here $\partial _v$ denotes the directional derivative.
\end{prop}

\begin{proof}
Let $v=(v_1 \frac{\partial}{\partial z}, v_2 \frac{\partial}{\partial \bar{z}})$ be a unit tangent vector in the complex plane so that $v_1^2+v_2^2=1/2$. Then for $n=1$, and $z \in TO_\al,$
\begin{align*} 
 \partial_{v} J_\al (z) &= \partial_{v}|h'_\al(z)|^2 = v_1 \cdot h''_\al(z) \overline{h'_\al(z)}+ v_2 \cdot h'_\al(z) \overline{h''_\al(z)}
\end{align*}
and obtain
\begin{align*} 
\left| \frac{ \partial_{v} J_\al (z) }{ J_\al(z)} \right|&= \left| \frac{v_1 \cdot h''_\al(z) \overline{h'_\al(z)}+v_2 \cdot h'_\al(z) \overline{h''_\al(z)}}{|h_\al'(z)|^2} \right| \numberthis \label{dist:n=1} \\
&= \left| v_1  \frac{h_\al''(z)}{h_\al'(z)}+ v_2 \frac{\overline{h''_\al(z)}}{\overline{h'_\al(z)}}  \right| \\
&\leq 4(v_1^2+v_2^2) \left|   \frac{h_\al''(z)}{h_\al'(z)} \right|  \leq  \frac{2}{|z+\al|} < 2R
\end{align*}
since $1/(z+\al) = h_\al(z) \in O_\al \subset \{z\in \C: |z|<R\}$ by Lemma~\ref{prop:2.1}. 

For $n>1,$ say for $\ba = (\alpha_1,\cdots,\alpha_n)$, i.e. $h_{\ba}= h_{\al_1} \circ \cdots \circ h_{\al_n}$ and $z \in T^nO_\ba,$  let $k_{n-i}=h_{\al_{i+1}} \circ \cdots \circ h_{\al_n}$. By the chain rule of complex derivative and contraction from Lemma~\ref{prop:2.1}, inductively, we have  
\begin{align*}  
\left| \frac{ \partial_{v} J_\al (z) }{2 J_\al(z)} \right| \leq 
\left| \frac{h_\ba''(z)}{h_\ba'(z)} \right| &=  \left| \frac{(h_{\al_1}'' \circ k_{n-1}) (z)}{(h_{\al_1}' \circ k_{n-1}) (z)} \cdot k_{n-1}'(z) + \frac{k_{n-1}''(z)}{k_{n-1}'(z)} \right| \\
&\leq R \rho^{\frac{n-1}{2}} + \left| \frac{(h_{\al_2}'' \circ k_{n-2}) (z)}{(h_{\al_2}' \circ k_{n-2}) (z)} \cdot k'_{n-2}(z) + \frac{k_{n-2}''(z)}{k_{n-2}'(z)} \right| \\
&\leq R(\rho^{\frac{n-1}{2}}+ \cdots + \rho^{\frac{1}{2}}+1),
\end{align*}
which is uniformly bounded by the constant $M:=\frac{2R}{1-\rho^{1/2}}>0$. 
\end{proof}

\section{Finite Markov partition with cell structure} \label{sec:markov}

In this section, we use the work of Ei--Nakada--Natsui \cite{Nakada:pre} to endow $I$ with a cell structure, which we denote by $\mathcal P$. We show that $\mathcal P$ is compatible with $T$ in the sense of Definition \ref{def:compatibility}. Accordingly, we remark the existence of dual inverse branches through their construction of the natural extension and their metric properties, which will be crucially used later in \S\ref{L2:1}.

\subsection{Cell structures on $I$ and function spaces} \label{section:functionspace}
In this subsection, we equip $I$ with a cell structure to define a function space.
A cell structure on a space 
$X$ is a homeomorphism identifying 
$X$ with a regular CW-complex; we hence view 
$X$ as such.
For $k \geq 1,$ a
$k$-cell is the image of an open 
$k$-ball under an attaching map, which is an open subset of the 
$k$-skeleton, and is properly contained in its closure.
A $0$-cell is a vertex, which is equal to its closure.
The structure is finite if the set $\mathcal P$ of cells is finite.

We introduce a finite cell structure $\mathcal P$ on $I$, which is required to have a certain compatibility with the countable partition $\{ \mathcal O_\al\}$ defined in \eqref{part:Oalpha} (see Definition~\ref{def:compatibility} below). For $0 \le i \le 2$, let $\mathcal P[i]$ be the set of cells of real dimension $i$. Since $I \subset \mathbb{C}$, we have $\mathcal P = \bigcup_{i=0}^2 \mathcal P[i]$. For $P \in \mathcal P$, denote by $\overline P$ its closure in $I$. 


\begin{figure}[h]
   \centering
   \includegraphics[width=4.6 in]{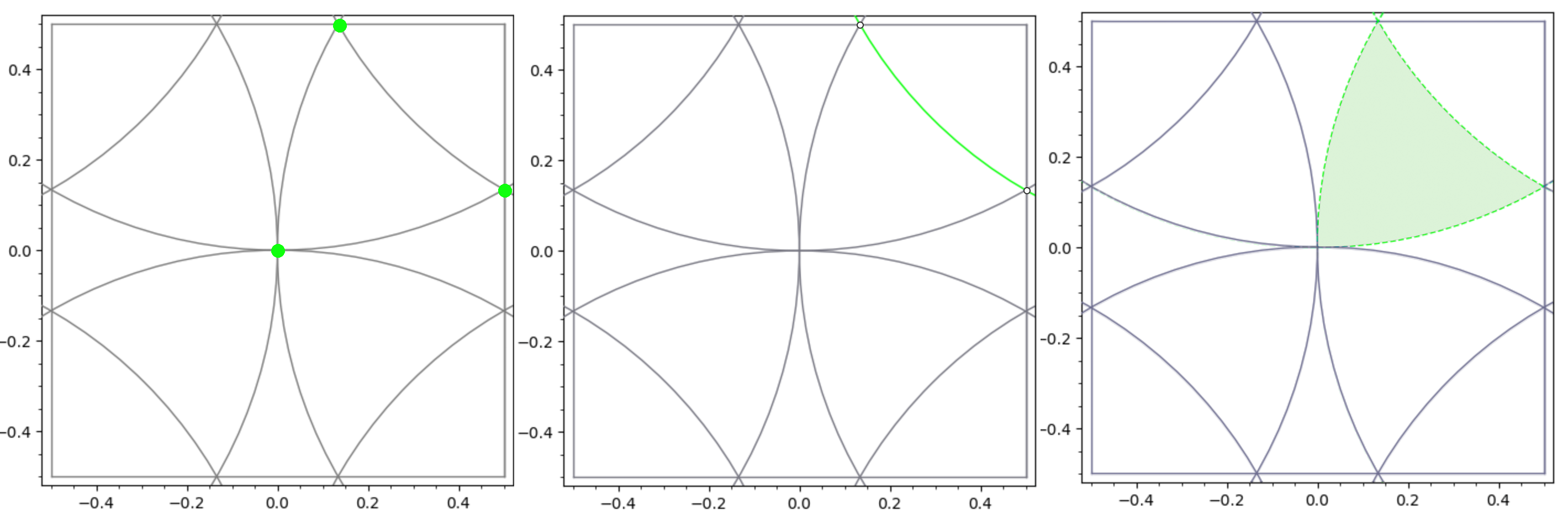} 
   \caption{ Examples of 0, 1, 2-cells in a finite partition $\Pc$ depicted in green ($d=1$).}
   \label{cells}
\end{figure}

    \begin{defn}\label{def:functionspace}
    Define $C^1(\mathcal P)$ to be the space of functions $f \colon I \to \mathbb C$ such that for every $P \in \mathcal P$, $f|_P$ extends to a continuously differentiable function on an open neighbourhood of $\overline P$.
    \end{defn}
    
\begin{rem}
    If $P\in\mathcal P[0]$, then the condition is vacuous since any function on a single point extends uniquely to a constant function on $\mathbb C$.
\end{rem}
    Denote the extension of $f|_P$ to $\overline P$ by $\operatorname{res}_P(f)$.
    By the uniqueness of such an extension, it defines a linear map $\operatorname{res}_P \colon C^1(\mathcal P) \to C^1(\overline P)$.
    They collectively define a linear map
    \begin{align*}
    \operatorname{res}_{\mathcal P} \colon C^1(\mathcal P) &\longrightarrow \bigoplus_{P \in \mathcal P} C^1(\overline P) \numberthis \label{int:fsp}
    \\
    f &\longmapsto \left(\operatorname{res}_P\left( f \right)\right)_{P \in \mathcal P}, 
    \end{align*}
    which is in fact bijective. We introduce a key definition.

    \hide{------
        \begin{prop}\label{intro:dec}
        The map $\operatorname{res}_{\mathcal P}$ is an isomorphism.
        \end{prop}
        \begin{proof}
        We show that $\operatorname{res}_{\mathcal P}$ is an isomorphism by constructing an inverse.
        For $f_P \in C^1(\overline P)$, we let $\tilde f_P \colon I \to \mathbb C$ be the function defined as
        \begin{align*}
            \tilde f_P(z) =
            \begin{cases}
            f_P(z) &\text{if $z \in P$,}
            \\
            0 &\text{if $z \not \in P$.}
            \end{cases}
        \end{align*}
        Define $j \colon 
        \bigoplus_{P \in \mathcal P} C^1(\overline P)
        \to
        C^1(\mathcal P)$
        by sending $(f_P)_{P \in \mathcal P}$ to $\sum_{P \in \mathcal P} \tilde f_P$.
        Since $\mathcal P$ is a set-theoretic partition of $I$, the restriction of $\sum_{P \in \mathcal P} \tilde f_P$ to a given $P \in \mathcal P$ agrees with $f_P$ on $P$. 
        Thus $\sum_{P \in \mathcal P} \tilde f_P$ belongs to $C^1(\mathcal P)$. 
        Once we have defined $j$, it is easyhttps://www.overleaf.com/project/655754eeaf5a9e49d9ce23e8 to verify that $\operatorname{res}_{\mathcal P} \circ j$ and $j \circ \operatorname{res}_{\mathcal P}$ are identity maps, respectively.
        \end{proof}

        Given a cell structure $\mathcal P$ on $I$, the associated function space $C^1(\mathcal P)$ is not necessarily preserved by the transfer operator.
        Therefore, it is natural to restrict our discussion to cell structures which are compatible with $T$ in a suitable sense.
        This motivates the following definition.
        
        ------}

    
    
    \begin{defn}\label{def:compatibility}
    A finite cell structure $\mathcal P$ on $I$ is said to be compatible with $T$ if the following conditions are satisfied.
    \begin{enumerate}
        \item (Markov) For each non-empty $O_\alpha$, $TO_\alpha$ is a union of cells in $\mathcal P$.
        \item For any inverse branch $h_\alpha$ and any $P \in \mathcal P$, either there is a unique element $Q \in \mathcal P$ such that $h_\alpha(P) \subset Q$ or $h_\alpha(P)$ is disjoint from $I$.
    \end{enumerate}
    \end{defn}

    Note that if $\mathcal P$ is compatible with $T$, then the characteristic function $1_{TO_\alpha}$ of $TO_{\alpha}$ belongs to $C^1(\mathcal P)$. The following proposition heavily depends on the work of Ei--Nakada--Natsui \cite{Nakada:pre}.


    
    \begin{propL}[Proposition \ref{part:markov}] \label{prop:compatibility}
    For each of the systems $I_d$ with $d=1,2,3,7,11$, there exists a finite cell structure $\Pc$ compatible with $T$.
    \end{propL}

\begin{figure}[h] \label{d=1_part}
   \centering
   \includegraphics[width=4.2 in]{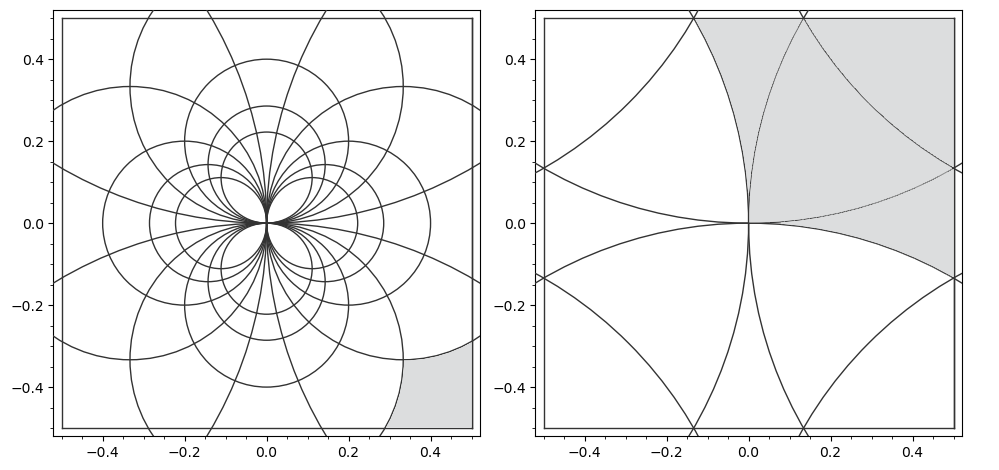} 
   \caption{Partition element $O_{1+i}$ and image $TO_{1+i}$ (as a disjoint union of cells in a finite partition $\Pc$) depicted in grey ($d=1$).}
   \label{fig:1}
\end{figure}

\hide{--------
Let $\alpha \in \Ac$. For each $z=(x,y) \in TO_\al$, we have $h_\alpha(z) \in O_\al \subset I$, thus $(a+x)^2+(b+y)^2 \leq R < 1$. On the other hand, differentiation yields
\[\left| \frac{\partial J_\al(x,y)}{\partial x} \right| = \frac{4\left | a+x \right |}{((a+x)^2+(b+y)^2)^3},  \]
and we have
\[
\max \left(
\left| \frac{\partial J_\al(x,y)}{\partial x} \right| 
,
\left |\frac{\partial J_\al(x,y)}{\partial y} \right| 
\right)
= \frac{4 \max \left( \left | a+x \right |, |b+y| \right)}{((a+x)^2+(b+y)^2)^3} \le \frac{4}{\left((a+x)^2+(b+y)^2\right)^2},
\]
establishing the first upper bound with $M(1)= 2 \sqrt 2$.

To proceed by induction, assume that $M(1), M(2), \cdots,  M(n-1)>0$ with desired properties have been found. Let $ \alpha = (\alpha_1,\alpha_2,\cdots,\alpha_n)$ and put $ \beta = (\alpha_2,\cdots,\alpha_n)$. Then, applying the chain rule to $h_{\alpha_1} \circ h_{ \beta}$, we obtain
\[
J_{ \alpha}(x,y) 
= J_{\alpha_1}(h_{ \beta}(x,y)) \cdot J_{ \beta}(x,y) 
\]
whose partial derivative is
\[
\frac{\partial J_{ \alpha}}{\partial x}(x,y) 
=
\frac{\partial h_{ \beta}}{\partial x}(x,y) 
\cdot
\frac{\partial J_{\alpha_1}}{\partial x}(h_{ \beta}(x,y))
\cdot 
J_{ \beta}(x,y) 
+
J_{\alpha_1}(h_{ \beta}(x,y))
\cdot 
\frac{\partial J_{ \beta}}{\partial x}(x,y).
\]
The condition $h_{\alpha_1}(P_1) \subset Q_1$ implies that $\left |\frac{\partial h_{ \beta}}{\partial x}(x,y) 
\right| \le 1$, and the induction hypothesis ensures that
\begin{align*}
    \left|\frac{\partial J_{\alpha_1}}{\partial x}(h_{ \beta}(x,y))\right|
    & \le 
    M(1) \left| J_{\alpha_1}(h_{ \beta}(x,y)) \right|
    \\
    \left | \frac{\partial J_{ \beta}}{\partial x}(x,y)\right |
    & \le
    M(n-1) \left| J_{ \beta}(x,y)\right|
\end{align*}
which we use to obtain
\begin{align*}
    \left | 
    \frac{\partial J_{ \alpha}}{\partial x}(x,y) 
    \right |
    \le
    \left(M(1)+M(n-1)\right) \left| J_{ \alpha}(x,y) \right| .
\end{align*}
We established the desired bound with $M(n) = M(1) + M(n-1)$.
----------}

\subsection{Cell structure from lines and circles}
In this subsection, we introduce a procedure which yields a cell structure that is compatible with $T$. 
Here, we focus on the procedure and its property will be verified in the following subsection.

The finite cell structure $\mathcal P$ for $I$ will be constructed from a finite set $Z$ of lines and circles.
We begin by noting that $I=I_d \subset \C$ is a bounded convex closed subset whose boundary $\partial I$ is a piecewise linear closed curve embedded in $\mathbb C$.
Suppose that we are given a finite set $Z$ of lines and circles in $\C$ that contains all lines that are extensions of a line segment in $\partial I$.
This condition means that $Z$ contains four (for $d=1,2$) or six (for $d=3,7,11$) lines obtained by elongating the sides of $\partial I$.

Given such a finite set $Z$, define the \textit{local dimension} function
$r \colon I \to \{0,1,2\}$
by 
\begin{align} \label{eq:ftn-r}
    r(x) &= 
    \begin{cases}
    2 &\text{if $x$ belongs to no line or circle in $Z$,}
    \\
    1 &\text{if $x$ belongs to exactly one line or circle in $Z$,}
     \\
   0 &\text{if $x$ belongs to two or more lines or circles in $Z$.}
    \end{cases}
\end{align}
\begin{prop}\label{prop:cell-from-lines}
    The union of the collection $\mathcal P[i]$ of connected components of $r^{-1}(i)$ with $i=0,1,2$ defines a cell structure on $I$.
\end{prop}
\begin{proof}
    Let $\ell_1,\cdots,\ell_k$ be the sides of $I$, i.e. the line segments of $\partial I$, which extend to distinct lines $\tilde l_1, \cdots \tilde l_k$, respectively. By assumption, $Z$ contains the lines $\tilde \ell_i$'s.
    We use induction on $|Z|$. The smallest possible $Z$ is the set of lines $\tilde \ell_1,\cdots,\tilde \ell_k$. Since $I$ is convex, $\mathcal P[2]$ is the interior of $I$. A member of $\mathcal P[1]$ is the set $\ell_i$ minus its endpoints for some $i$. Finally $\mathcal P[0]$ consists of intersections of $\ell_i$'s. The collection $\Pc = \cup \Pc[i]$ is clearly a cell structure of $I$.
    
    Now suppose that we have a cell structure $\mathcal P_Z$ arising from $Z$ and let $z$ be an additional line or circle.
    We claim that $Z\cup\{z\}$ refines $\mathcal P_Z$ and yields a cell structure.
    Define a function
    $$
    s \colon z \to \{0,1\}
    $$
    by the rules $s(x)=1$ if $x$ belongs to no line $w \in Z$ and $s(x)=0$ otherwise.
    Since we only have lines and circles in $Z$, $s^{-1}(0)$ is finite and $s^{-1}(1)$ has a finite number of bounded connected components, say $m$.
    Each bounded component intersecting $I$ divides a member of $\mathcal P_Z[2]$ into two distinct connected components. Proceeding by induction on $m$, one sees that $\mathcal P_{Z \cup \{z\}}$ is a cell structure on $I$.
\end{proof}

Now we apply Proposition~\ref{prop:cell-from-lines} in order to specify a cell structure on $I$.
For example, when $d=1$, the set $Z$ of lines or circles will be those in Figure~\ref{cells}. In general, we want to find the set $Z$ of lines and circles that is stable under $T$. 
In \cite[\textsection\,4]{Nakada:pre}, the authors define a sequence of subsets
$$
\mathcal W_0 :=\partial I, \mathcal W_1 : =T(\partial I), \mathcal W_2, \mathcal W_3, \cdots  \subset I,
$$
where 
$\mathcal W_n$ is defined recursively as the union of sets of the form\footnote{In \cite{Nakada:pre}, p.3894 does not have ``$\cap I$" but it is most likely a typo.}
\begin{align}\label{eq:Wi}
\left(\left(w - b\right) \backslash \mathcal W_0 \right)\cap I
\end{align}
with
\begin{itemize}
    \item $w$ being a line or a circle which extends a line segment or an arc in $\mathcal W_{n-1}^{-1}$,
    \item $b \in \mathcal O$ satisfies $A \cap \left(I+b\right) \cap I^{-1}\not = \emptyset$.
\end{itemize}
Here, for a set $S$, denote $S^{-1} : = \{ 1/z: z \in S\}.$

One of the main results in \cite{Nakada:pre} is a case-by-case analysis of CF map to construct the so-called finite range structure (see Theorem\,1 of \cite{Nakada:pre}). A sufficient condition for the existence of a finite range structure is provided in Theorem\,2 of \cite{Nakada:pre}, and is verified by means of explicit calculations, which is quite extensive when $d\not = 1,3$. This sufficient condition is reproduced below as Theorem\,\ref{thm:nakada}.

\begin{thm}[Ei--Nakada--Natsui \cite{Nakada:pre}] \label{thm:nakada}
For the complex Gauss system $(I,T)$, there exists $n_0=n_0(d) \geq 1$ such that 
\begin{align}\label{eq:stable}
\mathcal W_{n_0+1} \subseteq \bigcup_{j=1}^{n_0} \mathcal W_j.
\end{align}
and the set
$
\mathcal W:= \bigcup_{n \ge 0} \mathcal W_n
$
is a finite union of line segments and arcs.
\end{thm}

\begin{defn} Let $Z(\mathcal W)$ be the set of all lines and circles extending line segments or arcs of $\mathcal W$. 
    Define $\mathcal P_{\operatorname{ENN}}$ to be the cell structure on $I$ induced by $Z(\mathcal W)$.
\end{defn}
For the equations of members of $Z(\mathcal W)$,
see \cite[\textsection\,4]{Nakada:pre}.
From now on, we denote $\mathcal P=\mathcal P_{\operatorname{ENN}}$ as our discussion will be restricted to the cell structure $\mathcal P_{\operatorname{ENN}}$.

\subsection{Compatibility of $\mathcal P$ and $T$
}

Here is a consequence of Theorem\,\ref{thm:nakada}.
\begin{lem}\label{lem:r-monotone}
    If $x,y \in I$ satisfies $T(x)=y$, then $r(x) \ge r(y)$.
\end{lem}
\begin{proof}
Suppose that $w \in Z(\mathcal W)$ passes through $x$. Unfolding the conditions defining $\mathcal W_n$'s, if a small neighborhood -- a line segment or an arc --  of $x$ belongs to $w$, it will be a subset of $\mathcal W$. Hence, it generates another $w' \in Z(\mathcal W)$ which passes through $y$. To verify $r(x) \le r(y)$, we consider three cases.
If $r(y)=2$, then there are no $w' \in Z(\mathcal W)$ which passes through $y$, by definition. The above observation shows that no $w \in Z(\mathcal W)$ passes through $x$, showing that $r(x)=2$.

If $r(y)=1$, then there is exactly one $w' \in Z(\mathcal W)$ passing through $y$. We need to show that there are at most one $w \in Z(\mathcal W)$ passing through $x$. 
Since $w'$ is the unique member of $Z(\mathcal W)$ passes through $y$, there are two possibilities. If $y \in \partial I$, then there is a unique $P \in \mathcal P[2]$ such that $y \in \overline{\mathcal P}$. If $y \not \in \partial I$, then there are exactly two distinct $P,Q \in \mathcal P[2]$ whose closure contains $y$. In the former case, there is exactly one $w$ passing through $y$. In the latter, there are exactly one such $w$ when $P$ and $Q$ belong to different $O_\alpha$'s, or no such $w$ if $P$ and $Q$ are contained in a single $O_\alpha$ for some $\alpha \in \mathcal A$.
If $r(y)=0$, there is nothing to prove.
\end{proof}

\begin{prop} \label{part:markov} The cell structure $\mathcal P$ is compatible with $T$, i.e.
\begin{enumerate}
\item For each non-empty $O_\alpha$, $TO_\alpha$ is a disjoint union of cells in $\mathcal P$.
\item For each inverse branch $h_\alpha$ and $P \in \mathcal P$, either there is a unique member $Q \in \mathcal P$ such that $h_\alpha(P) \subset Q$ or $h_\alpha(P)$ is disjoint from $I$.
\end{enumerate}

\end{prop}

\begin{proof}
(1) This follows immediately by definition of $\mathcal W_1$. 

(2) 
It is enough to show that  if $h_\al(P) \cap I \neq  \varnothing$, then $h_\al (P)\cap w =  \varnothing$ for all $w \in Z(\mathcal W)$.
If $h_\al(P) \cap w \neq  \varnothing$ for some $w \in Z(\mathcal W)$, then $\mathcal O_\al \cap w \neq  \varnothing$ since $\mathcal O_\al$ contains $h_\al (P)$. It follows that $\mathcal O_\alpha ^{-1} \cap w^{-1} \neq  \varnothing$, which in turn implies that $(I+\al) \cap w^{-1} \neq  \varnothing.$
Thus $\al$ is one of the elements that are used in the inductive process of constructing $Z(\mathcal W)$. It follows that $w = h_\al (w')$ for some $w' \in Z(\mathcal W)$, since $Z(\mathcal W)$ is stable under the inductive process \eqref{eq:stable}.
In other words, $h_\al (P) \cap h_\al (w') \neq  \varnothing$, i.e. $(P + \al)^{-1} \cap (w'+\al)^{-1} \neq  \varnothing.$ 
Thus we conclude that $P \cap w' \neq  \varnothing$, which contradicts the construction of the partition $\mathcal P$. See Figure \ref{al=1+2i}. 
\end{proof}
\begin{figure}[h] 
   \centering
   \includegraphics[width=4.5 in]{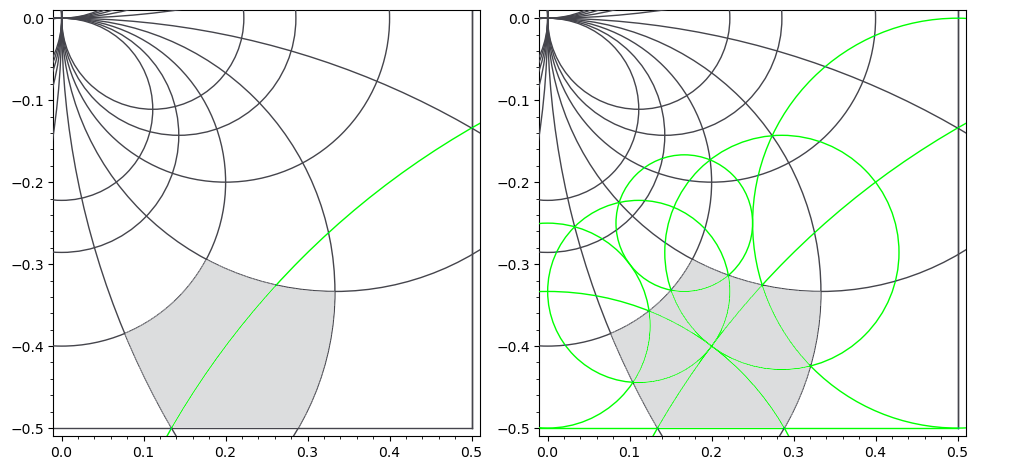} 
   \caption{Boundary of cells in $\Pc$ induced by a circle in $W_0 \cup W_1$ intersecting $I+(1+2i)$, and all images of $h_{1+2i}(P)$ inside $O_{1+2i}$ depicted in grey $(d=1)$. }
   \label{al=1+2i}
\end{figure}

\hide{----

\begin{proof}[Alternative proof]
(1) This follows from the construction of $\mathcal W_1$ from \S\ref{sec:nakada}. 
\par
(2) It suffices to show that $r$ is constant on $h_\alpha(P)$. Indeed, if $r$ on $h_\alpha(P)$ takes the constant value $i \in \{0,1,2\}$, then from the connectedness of $P$ we see that there will be a unique connected component, say $Q$ of $r^{-1}(i)$, such that $h_\alpha(P) \subset Q$. 

We proceed to show the constancy of $r$ on $h_\alpha(P)$. By definition of $\mathcal P$, $r$ is constant on $P$, say $r(x)=j$ for all $x \in P$. We consider each of the cases $j=0,1,2$. If $j=0$, there is nothing to prove because $h_\alpha(P)$ is a point. If $j=2$, then Lemma\,\ref{lem:r-monotone} implies that $r(x)=2$ for all $x \in h_\alpha(P)$. 

The last case $j=1$ remains. In that case, the only way it may fail to be constant is $r(h_\alpha(P))=\{1,2\}$, given the constraint from Lemma\,\ref{lem:r-monotone}. We claim that this is impossible. With $j=1$, $P$ is the boundary of some two-dimensional cell $P'$ which is contained in $TO_\alpha$. We first argue that there can be at most one such $P'$. Indeed, if $P',P'' \subset TO_\alpha$ have $P$ as boundary, then $h_\alpha(P)$ is contained in the boundaries of $h_\alpha(P')$ and $h_\alpha(P'')$, which are two-dimensional cells which share $h(P)$ as a boundary segment. It shows that $r\equiv 1$ on $h(P)$. Now we consider the case when there is exactly one $P'$. This happens only if $P$ is contained in $\partial I$. If $r(h_\alpha(P))=\{1,2\}$, then there is some $x \in h_\alpha(P)$ with $r(x)=2$. Since $r^{-1}(2)$ is discrete, $r(x')=1$ if $x' \in h_\alpha(P)$ near $x$ while $x \not = x'$. So there is $w \in Z(\mathcal W)$ such that $w$ passes through $x$ but does not extend $h(P)$. Then $Tw$ must cut through $P'$, which is a contradiction. 
\end{proof}
----}

\hide{----

For part (2), notice that there are a finite or countable number of $\al$'s such that $O_\al$ intersects with several elements of $\Pc$. For instance, for $d=1$, we have $\al=1+2i, 2+i, 2+2i$ (for which $O_\al$ lies in 4th quadrant and intersects two disjoint 2-cells and 1-cell in $\Pc$) and its rotations. When $d=3$, we have infinitely many such $\al$'s; e.g., we have $\al=\frac{m+(m-1)\sqrt{-3}}{2}$ $(m \geq 2)$ such that $O_\al$ lies in 4th quadrant and intersects with two disjoint 2-cells and 1-cell. In general, for other cases of $d$, it may intersect with more than three cells.

Since $O_\al=h_\al(TO_\al)$, the above can be understood as $I+\al$ intersecting with a finite number of circles $w$ in $W:=W_0 \cup \cdots \cup W_{n_0(d)}$. Under the inversion map, these circles map to circles that induce the boundaries of cells in $\Pc$ that intersect with $O_\al$ in $I$. Further, from the construction of $W_j$, circles $w-\al$ belong to $W$, i.e., each $w$ is a boundary of $P+\al$ for some $P \in \Pc$. Hence all the images of $P+\al$ under the inversion induce the circles intersecting with $O_\al$, where some of them coincide with the boundaries of cells, and preserving the angles at the intersections (0-cells) due to conformality of the linear fractional transformation $h_\al$ being homography. See Figure \ref{al=1+2i}.

-----}
\hide{------

Ei--Nakada--Natsui \cite[Theorem 2]{Nakada:pre} showed that $W_n$ stabilises. 



Initialize $W_1$ as above. Denote by $w^{-1}$ the  image of $w \in W_1$ under the inversion map $z \mapsto 1/z$. Each of them may intersect with the interior of $I+\al$ for some (multiple) $\al$. We denote by $\Ac(w) \subset \Ac$ the set of such $\al$'s. Define

\begin{equation} \tag{\textasteriskcentered} \label{algorithm}
W_2:= \bigcup_{w \in W_1} \bigcup_{\al \in \Ac(w)} (w^{-1}-\al) .
\end{equation}
If $W_2 \neq W_1$, then return to Step \eqref{algorithm} and one can obtain $W_n$ inductively. Ei--Nakada--Natsui \cite[Theorem 2]{nakada:pre} show that this algorithm stabilises in a finite step.

------}



\hide{
\begin{expl}[$d=1$]
The set $W_0$ contains the lines $\ell_1^\pm: \mathrm{Re}(z)=\pm 1/2$ and $\ell_2^\pm:\mathrm{Im}(z)=\pm 1/2$. Then the step \eqref{algorithm} induces the circles $C_1^\pm:|z \pm 1|=1$, $C_2^\pm: |z \pm i|=1$, $C_3^\pm:|z \pm (1+i)|=1$, and $C_4^\pm:|z \pm (1-i)|=1$ in $W_1$.
Then we see, under the inversion map, $\ell_1^\pm$ maps to $C_1^\mp$, $\ell_2^\pm$ maps to $C_2^\mp$, and $C_3^\pm$ maps to $C_4^\mp$. Hence, $W_2=W_1$ and $n_0=1$. 
\end{expl}
}

\subsection{Natural extension and dual inverse branch}
In this subsection, we record some of the results of \cite{Nakada:etal2} that will be used later.

Based on Theorem \ref{thm:nakada}, Ei-Nakada-Natsui constructed the natural extension map $\widehat{T}$ of $T$ and found a subset of $ I \times \C$ on which $\widehat{T}$ is bijective (modulo a null set) as follows.

For $z \in I$, write $\frac{P_n}{Q_n}=[ \al_1, \ldots, \al_n]$ for the $n$-th convergent of $z$. Remark that we have $\frac{Q_{n-1}}{Q_n}=[\al_n, \ldots, \al_1]$. Put
\[ V_z^*:=\overline{\left\{-\frac{Q_n(z_n)}{Q_{n-1}(z_n)} : z_n \in I, T^n(z_n)=z, n \geq 1   \right\}} \cup \{ \infty \} \]
and define $\widehat{T}$ on $\widehat{I}=\{(z,w): z \in I, w \in V_z^* \}$ by $\widehat{T}(0,w)=(0, \infty)$ and 
\[ \widehat{T}:(z,w) \longmapsto \left(T(z), \frac{1}{w}-\al_1(z) \right)  \ \ (z \neq 0). \]
 
For each $P \in \Pc$ and $z \in P$, $(V_z^*)^{-1}$ lies in the closed unit disk \cite[\S5]{Nakada:pre}. Hence, we have a bounded fractal domain $I^*=\cup_{P \in \Pc} P^*$ which is contained in the closed unit disk, where
\[ P^*:= \bigcup_{z \in P}(V_z^*)^{-1}= \bigcup_{z \in P} \overline{\left\{\frac{Q_{n-1}(z_n)}{Q_{n}(z_n)} : z_n \in I, T^n(z_n)=z, n \geq 1   \right\}} . \]



\hide{------

\begin{thm} \label{dual:corresp}
There is a piecewise invertible map $T^* \colon I^* \to I^*$ whose inverse branches in all depths are in one-to-one correspondence under $h_{\alpha^*}\mapsto h_{\alpha}$. Moreover, $(I^*,T^*)$ is locally expanding.
\end{thm}

--------}

In view of continued fraction expansion, if the sequence of digits $\ba=(\al_1, \cdots, \al_n)$ is an expansion for $z \in I$, then the backward sequence $\ba^*=(\al_n, \cdots, \al_1)$ is also an admissible expansion for some $w \in I^*$. We denote by $h_{\ba^*}$ the corresponding inverse branch and call this the dual inverse branch. 

Denote by $\mathrm{Leb}$ the Lebesgue measure on $\R^2$. We then have: 

\begin{prop} \label{dual:metric} 

For $d \in \{1,2,3,7,11 \}$, there is a positive constant $R=R_d < 1$ such that
for $h_{\ba^*} \in \Hc^{*n}$ and $P \in \mathcal P,$
\begin{enumerate}
\item $\mathrm{Diam}(h_{\ba^*}(P^*)) \leq R^{2(n-1)} |1-R|^{-1}.$

\item $\mathrm{Leb}(h_{\ba^*}(P^*)) \leq \left(R^{2(n-1)} |1-R|^{-1}\right)^2$.
\end{enumerate}
\end{prop}

\begin{proof}
Let $C_d=1/R_d$, where $R_d<1$ is the minimal radius of the ball containing $I_d$ centered at the origin. Since $Q^*_{n-1}/Q^*_n  \in I_d,$ it follows that $C_d |Q^*_{n-1}|<|Q^*_n|$. For instance, we have $C_1 = \sqrt{2}.$

Following Ei--Ito--Nakada--Natsui \cite{Nakada:etal} (which covers the case $d=1$), we have
\begin{align*}
\left|w - \frac{P^*_n}{Q^*_n}\right| & \leq \frac{1}{|Q^*_n|^2 \left| 1 + T^n_d(w) \frac{Q^*_{n-1}}{Q^*_n}\right|}\\
 & \leq \frac{1}{|Q^*_n|^2| |1-1/C_d|} \leq R_d^{2(n-1)} |1-R_d|^{-1}.
\end{align*}
Then (2) follows immediately from (1). 
\end{proof}

\begin{prop} \label{distortion2} 

There exist $L_1, L_2>0$ such that for any $n \geq 1$ and $h_\ba \in \Hc^n$, all $z_1, z_2 \in I$,
$$ L_1 \leq \left | \frac{h_\ba'(z_1)}{h_{\ba}' (z_2)} \right|  \leq L_2 .$$
The same property holds for the dual inverse branch $h_{\ba^*} \in \Hc^{*n}$.
\end{prop}

\begin{proof}

Notice that $h_\ba$ with an admissible $\ba=(\al_1, \cdots, \al_n)$ corresponds to $\GL_2(\Oc)$ matrices with determinant $\pm 1$, 
\[ \twobytwo{0}{1}{1}{\al_1} \cdots \twobytwo{0}{1}{1}{\al_n}=\twobytwo{P_{n-1}}{P_n}{Q_{n-1}}{Q_n} .  \]
Thus we have $h_\ba(z)=\frac{P_{n-1}z +P_n}{Q_{n-1}z +Q_n}$ and $h_{\ba^*}(z^*)=\frac{P_{n-1}z^* +Q_{n-1}}{P_n z^* +Q_n}$, in turn we obtain the expression
\begin{equation} \label{distor}
\left| \frac{h_\ba'(z_1)}{h_{\ba}' (z_2)} \right| =\left| \frac{\frac{Q_{n-1}}{Q_n} z_2+1}{\frac{Q_{n-1}}{Q_n} z_1+1} \right|^2 \ \mbox{and } 
\left| \frac{h_{\ba^*}'(z_1^*)}{h_{\ba^*}' (z_2^*)} \right| =\left| \frac{\frac{P_n}{Q_n} z_2^*+1}{\frac{P_n}{Q_n} z_1^*+1} \right|^2 . 
\end{equation}

Recall the triangle inequality that $||z_1|-|z_2| | \leq |z_1+z_2| \leq |z_1|+|z_2|$ for any $z_1, z_2 \in \C$. Since $\ba$ and $\ba^*$ are admissible, we have $|\frac{P_n}{Q_n}|<R_d$, $|\frac{Q_{n-1}}{Q_n}|\leq 1$. Further for $j \in \{1,2\}$, $|z_j|<R_d$ and $|z_j^*|\leq 1$ as $I^*$ is a domain bounded by the unit circle.  Hence \eqref{distor} yields the final bounds, e.g. by taking $L_2=\frac{4}{|R_d-1|^2}$ and $L_1=\frac{1}{L_2}$. 
\end{proof}

\begin{rem} \label{rem:distortion}
The same argument yields
$$
L_1 \le \left| 
    \frac
        {h_\ba'(z_1)}
        {h_{\ba^*}' (z_2^*)}
\right| 
\le
L_2
$$
for all $z_1 \in I$ and $z_2^* \in I^*$.
\end{rem}

\section{Spectral gap of the transfer operators on piecewise $C^1$-space} \label{sec:acim}

In this section, we show that the transfer operator $\Lc_{s,w}$ acting on $C^1(\Pc)$ has a spectral gap with the unique simple dominant eigenvalue $\lambda_{s,w}$ which is the spectral radius $r(\Lc_{s,w})$ of $\Lc_{s,w}$ when $(s,w)$ is close to $(1,0)$.
The proofs are modified from those in \cite{bal:val} to deal with additional complexities which arise from the presence of cells in multiple dimensions.

More specifically, we show that the contributions from the cells in dimension less than two are negligible as long as one's interest is limited to the peripheral spectrum.

For $\ba = (\al_1, \cdots, \al_n) \in \Oc^n$ and $P \in \mathcal P$, by Proposition~\ref{part:markov}, $$h_\ba(P) =h_{\al_1} \circ \cdots h_{\al_n}(P)  \subset Q$$ for a unique $Q \in \mathcal P$ if $h_\ba(P)$ intersects $I.$
\begin{defn}\label{def:4.1} If $h_\ba(P) \subset Q$ as above, denote the restriction of $h_\ba$ to $P$ by $$\langle \ba \rangle ^P_Q \colon P \to Q$$ and call it \emph{the inverse branch of depth $n$ from $P$ to $Q$}. For $n=1,$ we denote $\ba = (\al)$ simply by $\al$ and $h_\ba$ simply by $h_\al$.
\end{defn}

For $P,Q \in \mathcal P$, denote the set of inverse branches from $P$ to $Q$ by
\begin{align*}
\mathcal H(P,Q) = \left\{h\colon P \to Q  : \text{ $h= \langle \alpha \rangle ^P_Q$  for some $\alpha \in \mathcal O$}\right\}.
\end{align*}


By Proposition ~\ref{prop:compatibility},  for $z \in P$, \eqref{eq:Lsw} can be rewritten as
\begin{align} \label{def:op:eachp}
\left(\mathcal L_{s,w}f\right)_P(z) = \sum_{Q \in \Pc} \sum_{\langle \alpha \rangle ^P_Q \in \mathcal H(P,Q)} (g_{s,w} \cdot f_Q ) \circ \langle \alpha \rangle ^P_Q (z) .
\end{align}
This shows that if $ \Pc$ is compatible with $T$, then the operator $\mathcal L_{s,w}$ preserves $C^1(\mathcal P)$.
We assume that the digit cost $c:\Oc \to \R_{\geq 0}$ is \emph{of moderate growth,} which means that $c(\al)=O(\log |\al|)$ for $\al \in \Oc$. For such $c$, there exists a neighborhood $U$ of $(1,0)$ in $\mathbb R^2$ such that for any $(s,w)$ with real parts $(\sigma, u) \in U$, the series
\begin{align}\label{def:modgrowth}
\sum_{\langle \alpha \rangle^P_Q \in \mathcal H(P,Q)} \exp\left (w c( \alpha )\right ) |J_\alpha|^s
\end{align}
converges for all $P,Q \in \mathcal P$. Therefore there exists $A_U>0$, depending only on $U$, such that the absolute value of (\ref{def:modgrowth}) is bounded by $A_U$.


\subsection{Function space: Norms on $C^1(\mathcal P)$} \label{sub:norms}

We show that Proposition \ref{part:markov} allows us to consider the space of piecewise continuously differentiable functions, on which the $\Lc_{s,w}$ acts properly.


\hide{\tcb{[In the definition, we should be careful about zero cells. Differentiability is not appropriate. "in $I$" added at the end. Check if it is correct.]}}


\begin{prop}\label{prop:dec}
The map $\operatorname{res}_{\mathcal P}$ is a bijection.
\end{prop}
\begin{proof}
We show that $\operatorname{res}_{\mathcal P}$ is an isomorphism by constructing an inverse.
For each $f_P \in C^1(\overline P)$ let $\tilde f_P \colon I \to \mathbb C$ be the function defined as
\begin{align*}
    \tilde f_P(z) =
    \begin{cases}
    f_P(z) &\text{if $z \in P$,}
    \\
    0 &\text{if $z \not \in P$.}
    \end{cases}
\end{align*}
Define $j \colon 
\bigoplus_{P \in \mathcal P} C^1(\overline P)
\to
C^1(\mathcal P)$
by sending $(f_P)_{P \in \mathcal P}$ to $\sum_{P \in \mathcal P} \tilde f_P$.
Since $\mathcal P$ is a set-theoretic partition of $I$, the restriction of $\sum_{P \in \mathcal P} \tilde f_P$ to a given $P \in \mathcal P$ agrees with $f_P$ on $P$. 
Thus, $\sum_{P \in \mathcal P} \tilde f_P$ belongs to $C^1(\mathcal P)$. 
Once we have defined $j$, it is easy to verify that $\operatorname{res}_{\mathcal P} \circ j$ and $j \circ \operatorname{res}_{\mathcal P}$ are identity maps, respectively.
\end{proof}

Recall that $\mathcal P[i] \subset \mathcal P$ be the set of open $i$-cells for $i=0,1,2$. Consider the following norms and semi-norms on $C^1(\mathcal P[i])$. For $P \in \mathcal P$ and $f_P \in C^1( \overline P)$, define
\[ 
\|f_P\|_0 := \sup_{z \in P}  \left| f_P(z) \right|.
\]
For a positive-dimensional cell $P$, define
\[
\|f_P\|_1 =  \sup_{z \in P} \sup_{v \in T^1(z)} \left|\partial_{v}f_P(z) \right|,
\]
where $T^1(z)$ is the set of all unit tangent vectors $v$ with directional derivative $\partial_{v}$ at $z$.
When the dimension of $P$ is zero, we adopt the convention that $\|f_P\|_1=0$.

For $t \neq 0$, put
\begin{align}\label{eq:norm}
\|f_P \|_{(t)}= \|f_P\|_0+ \frac{1}{|t|} \|f_P\|_1.
\end{align}
By abuse of notation, we equip $C^1(\Pc)$ with following norms. For $f = (f_P)_P$ and $k=0,1$, set
\begin{align} 
 \|f\|_{k} &= \sup_{P \in \mathcal P} \|f_P\|_k \label{norm1}
\\
 \|f\|_{(t)} &= \|f\|_0 + \frac{1}{|t|} \|f\|_1. \label{norm2}
\end{align}
The norm $\| \cdot \|_{(t)}$ is equivalent to $\| \cdot \|_{(1)}$ for any $t \not = 0$.

\begin{prop}
For any $t \not = 0$, $\left(C^1(\mathcal P), \|\cdot\|_{(t)}\right)$ is a Banach space.
\end{prop}
\begin{proof}
It suffices to show that for each $P \in \mathcal P$, $C^1\left( \overline P \right)$ is a Banach space with respect to the norm \eqref{eq:norm}, which is trivial for $P \in \mathcal P[0]$. For $P \in \mathcal P[i]$ with $i>0$, it is an elementary property of Sobolev spaces; see \cite[Prop.\,9.1]{brezis} and \cite[\textsection\,9.1\,Rmk.\,2]{brezis}.
\end{proof}


Decompose $\mathcal L_{s,w}$ into the sum of component operators
\begin{equation} \label{def:op:comp}
\mathcal L^{[i]}_{[j],(s,w)} \colon C^1(\mathcal P[j]) \to C^1(\mathcal P[i])
\end{equation}
with $0 \le i,j \le 2$. In particular, $\mathcal L^{[i]}_{[j],(s,w)}=0$ whenever $j<i$. 
We first look into the real parameter family $\Lc_{\sigma, u}$ and obtain the boundedness.

\begin{prop} \label{op:bounded}
For $(\sigma, u) \in U$, we have $\Lc_{\sigma,u}(C^1(\Pc)) \subset C^1(\Pc)$ and the operator norm $\| \Lc_{\sigma,u} \|_{(1)} \leq \widehat{A}_U$ with $\widehat{A}_U >0$.
\end{prop}

\begin{proof}
This is a straightforward calculation using \eqref{def:modgrowth}, similar to Proposition \ref{op:lasota} below, by taking $\widehat{A}_U= |\Pc|A_U (1+|\sigma|+R^4)$.
\end{proof}

\hide{-------
    \begin{proof}
    It is sufficient to check for a positive dimensional $P$. Let $v=(v_1 \frac{\partial}{\partial z}, v_2 \frac{\partial}{\partial \bar{z}})$ be a unit tangent vector with $v_1^2+v_2^2=1/2$. Recall that for any $\langle \alpha \rangle^P_Q \in \mathcal H(P,Q)$ and $z \in P$, 
    \begin{align*}
      \nabla_{v} J_\alpha (z) &= \nabla_{v}|(h'_\al)(z)|^2 = v_1 \cdot (h''_\al)(z) (\bar{h'}_\al)(z).
    \end{align*}
    
    Then we have
    \begin{align*}
      |\nabla_{v}(\Lc_{\sigma,u} f)_P(z)| &\leq \left| \sum_Q \sum_{\langle \alpha \rangle^P_Q} e^{u c(\al)} \nabla_v |J_\al(z)|^\sigma \cdot f_Q \circ \langle \alpha \rangle^P_Q(z) \right| \\
      &+ \left| \sum_Q \sum_{\langle \alpha \rangle^P_Q} e^{u c(\al)} |J_\al(z)|^\sigma \cdot  \nabla_v (f_Q \circ \langle \alpha \rangle^P_Q)(z) \right| \\
      & \leq \left( \sum_Q \sum_{\langle \alpha \rangle^P_Q} e^{u c(\al)} |\sigma| |J_\al(z)|^\sigma \frac{|h''_\alpha(z)|}{|h'_\al(z)|} \cdot |f_Q \circ \langle \alpha \rangle^P_Q(z)| \right)  \\
      & +\left( \sum_Q \sum_{\langle \alpha \rangle^P_Q} e^{u c(\al)} |J_\al(z)|^\sigma \cdot 2 |J_\al(z)| \cdot |\nabla_v f_Q \circ \langle \alpha \rangle^P_Q(z)| \right) .
    \end{align*}
    The first term is then bounded by $A_K M|\sigma| \|f \|_0$ where $A_K$ from \eqref{def:modgrowth} and $M$ from Proposition \ref{distortion1}, and the second term is bounded by $A_K \rho \|f \|_1$. By taking $\widetilde{A}_K$ to be $\max\{ A_K M |\sigma|, A_K \rho \}$, we obtain the boundedness of the operator. 
    \end{proof}
-------}

\subsection{Sufficient conditions for quasi-compactness} \label{sub:quasicomp}

The following is a sufficient criterion for the quasi-compactness of the bounded linear operators on a Banach space due to Hennion \cite[Theorem XIV.3]{Hennion}:

\begin{thm}[Hennion] \label{hennion}
Let $(B, \| \cdot \|)$ be a Banach space. Let $\| \cdot \|'$ be a continuous semi-norm on $B$ and $\Lc$ a bounded linear operator on $B$ such that
\begin{enumerate}
\item The set $\left\{\Lc(f) \colon f \in B, \| f \|\le 1 \right\}$ is pre-compact in $(B, \|\rdot \|')$.
\item For $f \in B$, $\|\Lc f \|' \ll \|f \|'$.
\item There exist $n \geq 1$, and real positive numbers $r$ and $C$ such that for $f \in B$,
\begin{align}\label{eq:ly}
\|\Lc^n f \| \le  r^n \|f\|+C\|f\|' \mbox{ and } r < r(\Lc) .
\end{align}
\end{enumerate}
Then $\Lc$ is quasi-compact, i.e., there is $r_e<r(\Lc)$ such that the part of its spectrum outside the disc of radius $r_e$ is discrete. 
\end{thm}

We remark that the two-norm estimate in (3) is so-called Lasota--Yorke (or Doeblin--Fortet, Ionescu--Tulcea and Marinescu) inequality. In this subsection, we verify the conditions of Hennion's criterion to obtain the quasi-compactness of the operator $\Lc_{\sigma,u}$ on $C^1(\Pc)$.

We immediately have (2) with $\| \Lc_{\sigma,u} \|_0 \leq |\Pc| A_U$. Further, we observe the following compact inclusion, which implies (1) that $\|\cdot \|_{(1)}$ is pre-compact in $\|\cdot \|_0$.

\begin{lem} \label{embedding}
The embedding $(C^1(\mathcal P), \|\cdot\|_{(1)}) \ra (C^1(\mathcal P), \|\cdot\|_{0})$ is a compact operator.
\end{lem}
\begin{proof}
It suffices to show that $(C^1(\mathcal P[i]), \|\rdot\|_{(1)}) \ra (C^1(\mathcal P[i]), \|\rdot\|_0)$ is compact for each $i=0,1,2$. When $i=0$, it follows from the Bolzano--Weierstrass theorem. For $i=1,2$, it follows from the Arzel\`a--Ascoli theorem. 
\end{proof}


Now let us prove the key Lasota--Yorke estimate \eqref{eq:ly}  in (3). This will be also useful for the later purpose. 
For $\ba = (\al_1, \cdots, \al_n) \in \Oc^n$ and $P \in \mathcal P$, recall from Definition~\ref{def:4.1} that the inverse branch of depth $n$ from $P$ to $Q$, which we denote by $\langle \alpha \rangle ^P_Q \colon P \to Q$, is the restriction of $h_\ba = h_{\al_1} \circ \cdots h_{\al_n}$ to $P$. Denote by 
$\Hc^n(P,Q)$ the set of all inverse branches of depth $n$ from $P$ to $Q$,
\[ 
\mathcal H^\star (P,Q) := \bigcup_{n \ge 1} \mathcal H^n(P,Q)
 \ \mbox{ and } \
\mathcal H^\star := \bigcup_{P,Q \in \mathcal P} \mathcal H^\star(P,Q) .
\]
 Note that $\langle \ba \rangle_Q^P$ extends uniquely to a conformal map on $\mathbb C \cup \{\infty\}$.

\begin{prop} \label{op:lasota}
Let $(\sigma,u) \in U$. For $f \in C^1(\Pc)$ and $n \geq 1$, we have 
\[ 
\| \Lc^n_{\sigma,u} f \|_{(1)} \leq C_U ( |\sigma| \| f \|_0 + \rho^{n} \| f\|_{(1)})  
\]
for some $C_U>0$, depending only on $U$, where $\rho<1$ is the contraction ratio. 
\end{prop}

\begin{proof}

It suffices to check for a positive dimensional $P$. Let $v=(v_1 \frac{\partial}{\partial z}, v_2 \frac{\partial}{\partial \bar{z}})$ be a unit tangent vector with $v_1^2+v_2^2=1/2$. Recall that for any $\langle \alpha \rangle^P_Q \in \mathcal H(P,Q)$ and $z \in P$, 
\begin{align*}
  \partial_{v} J_\alpha (z) &= \partial_{v}|(h'_\al)(z)|^2 = v_1 (h''_\al)(z) (\overline{h'}_\al)(z)+v_2 (h'_\al)(z) (\overline{h''}_\al)(z).
\end{align*}

Recall the notation that for $\langle \ba \rangle \in \Hc^n(P,Q)$, $\langle \ba \rangle= \langle \al_n \rangle_Q^{R_{n-1}} \circ \cdots \circ \langle \al_1 \rangle_{R_1}^{P}$ for some $R_1, \cdots , R_{n-1} \in \Pc$. 
We put $c(\ba) := \sum_{j=1}^n c(\al_j)$. 
For $n \geq 1$, we have
\[ (\Lc^n_{\sigma,u} f)_P(z)= \sum_{Q} \sum_{\langle \ba \rangle \in \Hc^n(P,Q)} 
e^{u c(\ba) } 
|J_\ba(z)|^\sigma \cdot f_Q \circ \langle \ba \rangle (z). \]
Thus we have
\begin{align*}
  |\partial_{v}(\Lc_{\sigma,u}^n f)_P(z)| &\leq \left| \sum_Q \sum_{\langle \ba \rangle} 
  e^{u c(\ba)}
  \partial_v |J_\ba(z)|^\sigma \cdot f_Q \circ \langle \ba \rangle(z) \right| \\
  &+ \left| \sum_Q \sum_{\langle \ba \rangle} 
  e^{u c(\ba)} 
  |J_\ba(z)|^\sigma \cdot \partial_v (f_Q \circ \langle \ba \rangle)(z) \right| \\
  & \ll \left( \sum_Q \sum_{\langle \ba \rangle} 
  e^{u c(\ba)} 
  |\sigma| |J_\ba(z)|^\sigma \left| \frac{ \partial_{v} J_\ba (z) }{ J_\ba(z)} \right| \cdot |f_Q \circ \langle \ba \rangle(z)| \right)  \\
  &+\left( \sum_Q \sum_{\langle \ba \rangle} 
  e^{u c(\ba)} 
  |J_\ba(z)|^\sigma \cdot 2 |J_\ba(z)| \cdot |\partial_v f_Q \circ \langle \ba \rangle(z)| \right) .
\end{align*}
The first term is then bounded by $\widetilde{A}_U M|\sigma| \| f \|_0$ and the second term is bounded by $\widetilde{A}_U \rho^{n} \|f \|_1$ (for a suitable $\widetilde{A}_U>0$ due to moderate growth  \eqref{def:modgrowth}), where $\rho$ from Proposition \ref{prop:2.1} and $M$ from Proposition \ref{distortion1}. By taking supremum and maximum on both sides, we obtain the inequality for some $C_U>0$.
\end{proof}

\hide{-----

For the operator $\Lc_{\sigma,u}$ on $(C^1(\Pc),\|\cdot\|_{(1)})$ and $(C^1(\Pc),\|\cdot\|_0)$, we are ready to verify the conditions in Theorem \ref{hennion} using the inequalities settled in \S\ref{sub:norms}-\ref{sub:quasicompact}.

\begin{lem} \label{quasi}
The operator $\Lc_{\sigma,u}$ on $C^1(\Pc)$ is quasi-compact. 
\end{lem}

\begin{proof}
We apply Theorem \ref{hennion} to a Banach space $(C^1(\Pc),\|\rdot\|_{(1)})$ with (semi-)norm $\|\rdot\|_0$. The conditions (1),(2) are implied by Lemma \ref{embedding} and direct computation, respectively. For condition (3), we obtain the inequality
\begin{equation} \label{eq:lasota}
\| \Lc^n_{\sigma,u} f \|_{(1)} \leq \widetilde{C}_K ((1+ |\sigma|) \| f \|_0 + \rho^n \| f\|_{(1)})  
\end{equation}
for some $\widetilde{C}_K>0$ instantly from Proposition \ref{op:lasota}. 
\end{proof}

--------}

\subsection{Ruelle--Perron--Frobenius Theorem} \label{subsec:pf}

In this subsection, we conclude the quasi-compactness by \S\ref{sub:quasicomp}, and in turn obtain the following Ruelle--Perron--Frobenius theorem, i.e. spectral gap for $\Lc_{\sigma,u}$ on $C^1(\Pc)$.

\begin{thm}\label{thm:ruelle}
There exists a small neighbourhood $U$ of $(\sigma, u) = (1,0)$ such that for any $(\sigma,u) \in U$, the operator $\mathcal L_{\sigma,u}$ on $C^1(\Pc)$ is quasi-compact. It has a real eigenvalue $\lambda_{\sigma,u}$ with the following properties:
\begin{enumerate}
\item The eigenvalue $\lambda_{\sigma,u}>0$ is unique and simple.  If $\lambda$ is an eigenvalue other than $\lambda_{\sigma,u}$, then $|\lambda|< \lambda_{\sigma,u}$.
\item A corresponding eigenfunction $\psi_{\sigma,u}=(\psi_{\sigma,u,2},\psi_{\sigma,u,1},\psi_{\sigma,u,0})$ for $\lambda_{\sigma,u}$ is positive. That is, $\psi_{\sigma,u,j}>0$ for all $j=0,1,2$.
\item There exists a unique linear functional $\nu_{\sigma,u}=(\nu_{\sigma,u,2},\nu_{\sigma,u,1},\nu_{\sigma,u,0} )$ and the dual operator satisfies $\Lc_{\sigma,u}^* \nu_{\sigma,u}=\lambda_{\sigma,u} \nu_{\sigma,u}$.

\item In particular, $\lambda_{1,0}=1$ and  $\nu_{1,0,2}$  is the 2-dimensional Lebesgue measure.
\end{enumerate}
\end{thm}

\begin{proof}
First we prove the quasi-compactness using Theorem~\ref{hennion}.
The required estimate \eqref{eq:ly} for some $n$ would follow from Proposition~\ref{op:lasota} if $\rho<r(\mathcal L_{\sigma,u})$ for any $(\sigma,u) \in U$. 
Since $r(\mathcal L_{\sigma,u})=r(\mathcal L_{\sigma,u}^*)$, where $\mathcal L_{\sigma,u}^*$ is the dual operator, it suffices to prove $\rho<r(\mathcal L_{\sigma,u}^*)$.
Indeed, observe that the change of variable formula implies
\begin{equation} \label{change:jac}
\int_I \Lc_{1,0} f(x,y) dxdy=\int_I f(x,y) dxdy  
\end{equation} for any $f \in C^1(\Pc).$  Thus, the linear functional $\alpha : (f_2, f_1, f_0) \mapsto \sum_{P \in \mathcal{P}[2]}\int_P f_2 dxdy$ is an eigenfunctional, i.e. an element of $(C^1(\mathcal P))^*$ with eigenvalue 1 for $\mathcal L_{1,0}^*$. So we conclude $1 \leq r(\Lc_{1,0}^*)$. 
By the analyticity of $r(\mathcal L_{\sigma,u}^*)$ in $(\sigma,u)$, if $U$ is a sufficiently small neighbourhood of $(1,0)$, we have $\rho< R^4 < r(\mathcal L_{\sigma,u}^*)$ for any $(\sigma,u) \in U.$



To proceed, we state and prove some $L^1$-estimates. In view of Proposition~\ref{prop:dec}, we have a decomposition 
\[ C^1(\Pc)= C^1(\Pc[2]) \oplus C^1(\Pc[1]) \oplus C^1(\Pc[0])   \]
and accordingly the operator 
$\Lc:=\Lc_{\sigma,u}$
can be written as 
\begin{equation} \label{eqn:matrixform} \Lc f= \begin{bmatrix}
\Lc_{[2]}^{[2]} & 0 & 0 \\
\Lc_{[2]}^{[1]} & \Lc_{[1]}^{[1]} & 0 \\
\Lc_{[2]}^{[0]} & \Lc_{[1]}^{[0]} & \Lc_{[0]}^{[0]} 
\end{bmatrix} \begin{bmatrix}
f_2  \\
f_1  \\
f_0  
\end{bmatrix}  \end{equation}
with $\Lc_{[j]}^{[i]}: C^1(\Pc[j]) \ra C^1(\Pc[i])$ from \eqref{def:op:comp}.
Equip each $C^1(P)$ for $P \in \mathcal P[i]$ with the $L^1$-norm, by which we mean the $L^1$-norm with respect to the Lebesgue measure, $L^1$-norm with respect to the length element, and the counting measure, respectively for $i=2,1,0$. Define the $L^1$-norm on $C^1(\mathcal P[i])$ to be the sum of $L^1$-norms on its direct summands $C^1(P)$.

We claim that, for $(\sigma,u)=(1,0)$,
\begin{align}
    \label{eq:lii-l1}
    \| \Lc_{[i]}^{[i]}\|_{L^1} &\le R^{4-2i}\text{ for $i=2,1,0$.}
\end{align}
The case $i=2$ is immediate since for $f \in C^1(\mathcal P[2])$, the change of variable formula with the triangle inequality implies $\|\mathcal L_{[2]}^{[2]} f\|_{L^1}\le \|f\|_{L^1}.$ To obtain the cases $i=1,0$ we use similar arguments. First consider the case $i=1$. By definition of $\Lc_{[1]}^{[1]}$, for $f \in L^1(\Pc[1])$ and $P \in \Pc[1]$, we have
\[
\| ( \Lc_{[1]}^{[1]} f )_P \|_{L^1}
    = \sum_{Q \in \mathcal P[1] }
    \int_{P}
    \left|
   \sum_{\langle \alpha \rangle \in \Hc(P,Q)} \left|z+\alpha\right|^{-4} \cdot f_Q \circ \langle \alpha \rangle (z)
    \right|
    d \ell_P
\]
where $d \ell_P$ is the length element of the curve $P$. 
Applying the change of variable formula to the right hand side, we obtain
\begin{align*}
\| ( \Lc_{[1]}^{[1]} f )_P \|_{L^1} =     \sum_{Q \in \mathcal P[1]} \int_{h_\al(P)} |z|^2 |f_Q(z)| d\ell_Q .
\end{align*}
Since $h_\alpha(P)$'s are disjoint and $|z| \le R$ for $z \in I$, we conclude $\|\mathcal L_{[1]}^{[1]} f\|_{L^1}\le R^2 \|f\|_{L^1}$.

Now consider the case $i=0$. For $P \in \Pc[0]$, we have
\[
\| ( \Lc_{[0]}^{[0]} f )_P \|_{L^1}
    = \sum_{Q \in \mathcal P[0]} 
    \left|
   \sum_{\langle \alpha \rangle \in \Hc(P,Q)} \left|z+\alpha\right|^{-4} \cdot f_Q \circ \langle \alpha \rangle (z)
    \right|
\] 
where $L^1(\Pc[0])$-norm is given by the integral with respect to a counting measure. Again by the disjointness of $h_\alpha(P)$, we conclude $\| \Lc_{[0]}^{[0]}  f\|_{L^1} \leq R^4 \| f \|_{L^1}$.

Since $C^1(P)$ is a subspace of $L^1(P)$, \eqref{eq:lii-l1} yields, for $(\sigma,u)=(1,0)$, $r(\mathcal L_{[i]}^{[i]}) \le \| \Lc_{[i]}^{[i]} \|_{L^1} \le R^{4-2i}$ for $i=2,1,0$. 
\hide{\tcb{Seems that we are not using the density, but just that $C^1(P)$ is a subset of $L^1(P)$, right?}}
With $1 \leq r((\Lc_{[2]}^{[2]})^*)=r(\Lc_{[2]}^{[2]})$, it follows that 
\begin{align}\label{eq:rineq}
r(\mathcal L_{[2]}^{[2]}) > r(\mathcal L_{[i]}^{[i]})
\end{align}
for $i=0,1,$ thus for all $(\sigma,u) \in U$, 
we have $r(\mathcal L_{\sigma, u}) = r((\mathcal L^{[2]}_{[2],(\sigma, u)})$.

Now to prove (1), observe first that the assertion (1) and (2) for $\mathcal L_{[2]}^{[2]}$ when $(\sigma,u)\in U$ follows by adapting the proof of {\cite[Theorem 1.5.(4)]{Baladi}} with Proposition \ref{op:lasota}. Thus the spectral radius $r(\mathcal L^{[2]}_{[2],(\sigma, u)})$ is a positive simple eigenvalue $\lambda_{\sigma, u}$ with a positive eigenfunction
$\psi_{\sigma, u, 2} \in C^1(\Pc[2]),$ i.e., 
$\mathcal L^{[2]}_{[2],(\sigma, u)} \psi_{\sigma, u, 2} = \lambda_{\sigma, u} \psi_{\sigma, u, 2}.$


Observe that $(f_2,f_1,f_0) \mapsto f_2$ induces a map from the $\lambda_{\sigma,u}$-eigenspace of $\mathcal L$ to that of $\mathcal L_{[2]}^{[2]}$ by the equation \eqref{eqn:matrixform}. 
We claim that~\eqref{eq:rineq} implies that this is an isomorphism. Indeed, if $f_2=0,$ then $\Lc f_1 = \Lc_{[1]}^{[1]} f_1$, thus by \eqref{eq:rineq}, $(0, f_1, f_0)$ cannot be an eigenfunction for $\Lc$. If $f_2 \neq 0$ is a $\lambda_{\sigma,u}$-eigenfunction for $\mathcal L_{[2]}^{[2]}$ then there is a unique way to complete it as a triple $(f_2,f_1,f_0)$ which is an eigenfunction of $\mathcal L$. Concretely, $f_1$ and $f_0$ are determined by $f_2$ via the formulae
\begin{align}\label{eq:f1}
f_1=\lambda_{\sigma,u}^{-1}(1- \lambda_{\sigma,u}^{-1} \mathcal L_{[1]}^{[1]})^{-1}(\mathcal L_{[2]}^{[1]} f_2)
\end{align}
and
\begin{align}\label{eq:f0}
f_0=\lambda_{\sigma,u}^{-1}(1- \lambda_{\sigma,u}^{-1} \mathcal L_{[0]}^{[0]})^{-1}(\mathcal L_{[2]}^{[0]} f_2+\mathcal L^{[0]}_{[1]} f_1)
\end{align}
where the existence of $(1- \lambda_{\sigma,u}^{-1}\mathcal L_{[i]}^{[i]})^{-1}$ for $i=0,1$ follows from~\eqref{eq:rineq}.

Now we prove~(2). 
From the referred proofs \cite{Baladi,hensley} for the first step in the preceding paragraph, we know that there is a $\lambda_{\sigma,u}$-eigenfunction $\psi_{\sigma,u,2}$ which is positive.
The positivity of $\psi_{\sigma,u,2}$ together with the formulae~\eqref{eq:f1} and~\eqref{eq:f0} implies $\psi_{\sigma,u,1}>0$ and $\psi_{\sigma,u,0}>0$ in order. So $\psi_{\sigma,u,2}=(\psi_{\sigma,u,2},\psi_{\sigma,u,1},\psi_{\sigma,u,0})$ is the positive eigenfunction for $\mathcal L$, as desired.

We prove (3). This is nothing but an equivalent form of (1) in terms of the dual of a Banach space. We remark that for a bounded linear operator $\Lc$ on a Banach space, the notion of dual $\Lc^*$ is well-defined and $\lambda \in \mathrm{Sp}(\Lc)$ if and only if $\lambda \in \mathrm{Sp}(\Lc^*)$. The operator $\mathcal L^*$ is upper-triangular and its $\lambda_{\sigma,u}$-eigenspace is identified with that for $(\mathcal L_{[2]}^{[2]})^*$. 


To prove (4), it suffices to show $r(\mathcal L_{[2]}^{[2]})=1$ when $(\sigma,u)=(1,0)$, because we had proved that $r(\mathcal L_{[2]}^{[2]})=r(\mathcal L)$. 
By \eqref{change:jac}, we have $r((\mathcal L_{[2]}^{[2]})^*)\ge 1$ when $(\sigma,u)=(1,0)$.
On the other hand, \eqref{eq:lii-l1} implies $r((\mathcal L_{[2]}^{[2]})^*)\le 1$. We conclude that $\lambda_{1,0}=1$. The assertion about the density function follows from the proof of~(3).
\end{proof}

\begin{rem} \label{pf:kuzmin}
Theorem \ref{thm:ruelle}.(4) can be viewed as an alternative proof of the main result of Ei--Nakada--Natsui \cite{Nakada:pre} based on a thermodynamic formalism. However, their proof based on the construction of an invertible extension yields an integral expression for the density function $\psi_{1,0,2}$;
\begin{equation} \label{nakada:density}
\psi_{1,0,2}(z)= \int_{P^*} \frac{1}{|z-w|^4} d \mathrm{Leb}(w) 
\end{equation}
for $z \in P$, where $P \in \Pc[2]$. See also Hensley \cite[Thm.5.5]{hensley} for the case $d=1$.
\end{rem}

\hide{------

\begin{prop}
Suppose that $\lambda_{1,0}=1$ is a simple eigenvalue of $\mathcal L_{1,0}$ and that any other eigenvalue $\lambda$ of $\mathcal L_{1,0}$ satisfies $|\lambda| < 1$ and that there is an eigenfunction $\psi_{1,0} = \left(\psi_{1,0,0},\psi_{1,0,1},\psi_{1,0,2}\right)$ with $\psi_{1,0,i}>0$ for $i=0,1,2$. 
\end{prop}

Put $\Lc = \mathcal L_{1,0}$ and for $i \le j,$ let
\begin{align}
    \mathbb H_{i,j} \colon \mathcal B_j \to \mathcal B_i
\end{align}

Let $\psi = \psi_0 + \psi_1 + \psi_2$ be a non-zero eigenvector for $\mathbb H$ with eigenvalue $\lambda$. \tcb{Positivity of each $\psi_i$ for $i=0,1$.}
\par
\begin{prop}\label{prop:PF}
The spectrum of $\mathbb H^*$ contains $\lambda$ and its eigenspace is one-dimensional. 
\end{prop}
\begin{proof}
\end{proof}

------}


We state some consequences of the assertion of Theorem~\ref{thm:ruelle}.(1). We refer the reader to
Kato \cite[{\S}VII.4.6, {\S}IV.3.6]{Kato}. First, there is a decomposition
$$
\mathcal L_{s,w} = \lambda_{s,w} \mathcal P_{s,w} + \mathcal N_{s,w}
$$
where $\mathcal P_{s,w}$ is a projection onto the $\lambda_{s,w}$-eigenspace and $\mathcal N_{s,w}$ satisfies 
both $r(\Nc_{s,w})<|\lambda_{s,w}|$ and $\mathcal P_{s,w}\mathcal N_{s,w}=\mathcal N_{s,w}\mathcal P_{s,w}=0$. Moreover, $\lambda_{s,w}$, $\mathcal P_{s,w}$, and $\mathcal N_{s,w}$ vary analytically in $(s,w)$.

In particular, for a given $\varepsilon>0$, for any $(s,w)$ in a sufficiently small neighborhood $K$ of $(1,0)$, we have $r(\Nc_{s,w})<|\lambda_{s,w}|-\varepsilon$. 
This yields
\begin{equation} \label{sp:decomp}
\Lc_{s,w}^n=\lambda_{s,w}^n \Pc_{s,w}+\Nc_{s,w}^n
\end{equation}
where $r(|\lambda_{s,w}|^{-n}\Nc_{s,w}^n)$ converges to zero as $n$ tends to infinity.

For later use, we state the following.

\begin{lem} \label{comp:pressure}
The function $(s,w) \mapsto \lambda_{s,w}$ satisfies: 
\begin{enumerate}
\item We have $\frac{\partial \lambda_{s,0}}{\partial s}  \big|_{s=1} < 0$, whence there is a complex neighborhood $W$ of 0 and unique analytic function $s_0:W \ra \C$ such that for all $w \in W$, 
\[ \lambda_{s_0(w),w}=1.   \]
In particular, $s_0(0)=1$.
\item We have $\frac{d^2}{dw^2} \lambda_{1+s_0'(w)w, w} \big|_{w=0} \neq 0$ if and only if $c$ is not of the form $g-g \circ T$ for some $g \in C^1(\Pc)$.
\end{enumerate}
\end{lem}

\begin{proof}

(1) Recall Theorem \ref{thm:ruelle} and \eqref{sp:decomp} that we have a spectral gap given by the identity $\Lc_{s,w} \psi_{s,w}=\lambda_{s,w} \psi_{s,w}$ and corresponding eigenmeasure $\nu_{s,w}$. We can assume that $\nu_{s,w}$ is normalised, i.e. $\int_I \psi_{s,w} d \nu_{s,w}=1$. Observe that 
\begin{align*}
\Lc_{s,w} \psi_{s,w} &= \sum_{Q \in \Pc} \sum_{\langle \alpha \rangle \in \Hc(P,Q)} e^{w c(\al)} |J_\al|^s \cdot (\psi_{s,w})_Q \circ \langle \alpha \rangle  \\
&= \Lc_{1,0} (e^{w c} |J_T|^{1-s} \cdot \psi_{s,w}) = \lambda_{s,w} \psi_{s,w} \numberthis \label{press:id1}
\end{align*}
where we regard $c$ as a function on $I$ given by $c (z):=c(\al)$ if $z \in O_\al$. Differentiating \eqref{press:id1} with respect to $s$ and integrating with respect to $\nu_{1,0}$ yields the identity: 
\[\frac{\partial \lambda_{s,0}}{\partial s} \bigg|_{s=1} = - \int_I \log |J_T| \psi_{1,0} d \nu_{1,0}.  \]
From the right-hand-side, we see that it is negative from the positivity of $|J_T|$ and $\psi_{1,0}$. Then the existence of $s_0$ is obtained by implicit function theorem.

(2) This is a standard argument (convexity of the pressure) using a spectral gap as detailed in e.g., Parry--Pollicott \cite[Proposition 4.9--4.12]{par:pol}, Broise \cite[Proposition 6.1]{broise}, or Morris \cite[Proposition 3.3]{morris}. Here, we briefly recall the main ideas.

Set $L(w):=\lambda_{1+s_0'(w)w,w}$ and $\Psi(w):=\psi_{1+s_0'(w)w,w}$. Notice that $L(0)=1$ and $L'(0)=0$ by the mean value theorem. Similarly as \eqref{press:id1}, we have for any $n \geq 1$, 
\[
\Lc_{1+s_0'(w)w,w}^n \Psi(w)=\Lc_{1,0}^n (e^{w \sum_{j=1}^n (c \circ T^{j-1})} |J_T|^{1-s} \cdot \Psi(w)) = L(w)^n \Psi(w).
\]
Differentiating this twice, setting $w=0$, and integrating gives 
\begin{equation} \label{press:id2} 
L''(0)= \lim_{n \ra \infty} \frac{1}{n} \int_I (\sum_{j=1}^n c \circ T^{j-1} )^2 \Psi(0) d\nu_{1,0}  
\end{equation}
with the use of some limiting argument for $\Psi'(0)$. 
Further, one can observe that the right hand side of \eqref{press:id2} equals to $\int_I \widetilde{c}^2 \Psi(0) d\nu_{1,0}$, where $\widetilde{c}:=c+ g \circ T - g$ for some $g \in C^1(\Pc)$. Hence $L''(0)=0$ if and only if $\widetilde{c}=0$, which yields the final form of the statement. 
\end{proof}

\hide{---------

For any $n \geq 1$, we have 
\begin{equation} \label{press:id2}
\Lc_{s,w}^n \psi_{s,w}=\Lc_{1,0}^n (e^{w \sum_{j=1}^n (c \circ T^j)} |J_T|^{1-s} \cdot \psi_{s,w}) = \lambda_{s,w}^n \psi_{s,w}.
\end{equation}
Then differentiating \eqref{press:id1} with respect to $s$, at $(s,w)=(1,0)$ gives 
\[ \frac{\partial \lambda_{s,0}}{\partial s} \bigg|_{s=1} \psi_{s,0}+ \lambda_{1,0} \frac{\partial \psi_{s,0}}{\partial s} \bigg|_{s=1}=\Lc_{1,0} \left(-\log |J_T| \psi_{1,0} +\frac{\partial \psi_{s,0}}{\partial s} \bigg|_{s=1} \right)  .\]
Integrating with respect to $\mu_{1,0}$ yields 
\[\frac{\partial \lambda_{s,w}}{\partial s}(1,0) = - \int_I \log |J_T| \psi_{1,0} d \mu_{1,0} <0  \]
by the positivity of $|J_T|$ and $\psi_{1,0}$. 

Differentiating \eqref{press:id2} with respect to $w$, at $(s,w)=(1,0)$ gives 
\[  n \lambda_{1,0}^{n-1} \frac{\partial \lambda_{1,w}}{\partial w} \bigg|_{w=0} \psi_{1,0}+ \lambda_{1,0}^n \frac{\partial \psi_{1,w}}{\partial w} \bigg|_{w=0}=\Lc_{1,0}^n \left( e^{w \sum_{j=1}^n (c \circ T^{j-1})} \left( \sum_{j=1}^n (c \circ T^{j-1}) \psi_{1,0} +\frac{\partial \psi_{1,w}}{\partial w} \bigg|_{w=0} \right) \right)  .\]
Setting $n=1$ and $w=0$, and integrating gives 
\[\frac{\partial \lambda_{s,w}}{\partial w}(1,0)= \int_I c \psi_{1,0} d\mu_{1,0}.    \]

Now differentiating \eqref{press:id2} twice with respect to $w$ at $(s,w)=(1,0)$ and setting $n=1$, we obtain 
\begin{align*}
\frac{\partial^2 \lambda_{1,w}}{\partial w^2} & \bigg|_{w=0}  \psi_{1,0} + 2 \frac{\partial \lambda_{1,w}}{\partial w}  \frac{\partial \psi_{1,w}}{\partial w} \bigg|_{w=0} +  \frac{\partial^2 \psi_{1,w}}{\partial w^2} \bigg|_{w=0}\\
&=\Lc_{1,0} \left(  \left( c \circ T \right)^2 \psi_{1,0} + 2 (c \circ T) \frac{\partial \psi_{1,w}}{\partial w} \bigg|_{w=0} +\frac{\partial^2 \psi_{1,w}}{\partial w^2} \bigg|_{w=0} \right)  .
\end{align*}
By integrating, we have 
\begin{equation} \label{press:id3}
\frac{\partial^2 \lambda_{s,w}}{\partial w^2}(1,0)= \int_I \left( (c \circ T)^2 \psi_{1,0}+2  \frac{\partial \psi_{1,w}}{\partial w} \bigg|_{w=0}  \left( c \circ T- \int_I  (c \circ T) \psi_{1,0} d\mu_{1,0}  \right)  \right) d\mu_{1,0} . 
\end{equation}
This shows that if $c$ is a positive constant, then the second term vanishes in the integral, hence \eqref{press:id3} is equal to $c^2>0$. If $c$ is non-constant, the identity yields that the first term of \eqref{press:id3} is non-zero if $c$ is not of the form $g-g \circ h$ for some $g \in C^1(\Pc)$.

--------------}

\section{A priori bounds for the normalised family} \label{sec:nor:lasota}

In this section, we establish some a priori bounds, which will be crucially used for Dolgopyat--Baladi--Vallée estimate in the section~\ref{sec:dolgopyat}.

For each $P \in \Pc$, normalise $\Lc_{s,w}$ by setting
\begin{equation} \label{def:normalop}
(\Lt_{s,w} f)_P= \frac{(\Lc_{s,w} (\psi_{\sigma,u} \rdot f ))_P}{\lambda_{\sigma,u} (\psi_{\sigma,u})_P} 
\end{equation}
where $\lambda_{\sigma,u}$ and $\psi_{\sigma,u}$ are from Theorem \ref{thm:ruelle}, and $(\psi_{\sigma,u})_P$ denotes the restriction of $\psi_{\sigma,u}$ to $P$. It follows that $\Lt_{\sigma,u} \bf 1=1$ and $\Lt_{\sigma,u}^*$ fixes the probability measure $\mu_{\sigma,u}:=\psi_{\sigma,u} \nu_{\sigma,u}$.


\subsection{Lasota--Yorke inequality} \label{sub:nor:lasota}

We begin with the Lasota--Yorke estimate and integral representation of the projection operator for the normalised family.

\begin{lem} \label{est:normalised}
For $(s,w)$ with $(\sigma,u) \in U$, we have for $f \in C^1(\Pc)$ and some constant $\widetilde{C}_K>0$
\begin{enumerate}
\item $\| \Lt_{s,w}^n f \|_{(1)} \leq \widetilde{C}_U (|s| \| f \|_0+ \rho^{n} \| f \|_{(1)})$.
\item $\| \Lt_{1,0}^n f \|_0=  \int_I f d\mu_{1,0}+O(r_{1,0}^n \|f \|_{(1)})$.
\end{enumerate}
Here $r_{s,w}$ denotes the spectral radius of $\frac{1}{\lambda_{\sigma,u}} \Lc_{s,w}-\Pc_{s,w}$. 
\end{lem}

\begin{proof}
To prove (1), it is enough to show that for each $P$
$$
\| (\Lt_{s,w}^n f)_P\|_{(1)} \leq \widetilde{C}_U (|s| \| f \|_0+ \rho^{n} \| f \|_{(1)}).
$$
If $P\in\mathcal P[0]$, then the left hand side involves no derivatives and the inequality holds for all sufficiently large $\tilde C_U$. 
Assume that $P$ is positive dimensional. 
Recall that $c(\ba)=\sum_{j=1}^n c(\al_j)$.
Divide $|\partial_v (\Lt_{s,w}^n f)_P|$ into three terms (\rom{1}), (\rom{2}) and (\rom{3}): 
\[ \lambda_{\sigma,u}^{-n} \cdot \frac{\partial_v (\psi_{\sigma,u})_P}{(\psi_{\sigma,u})_P^2} \sum_{Q \in \Pc } \sum_{\langle \ba \rangle } 
e^{wc(\ba)}
|J_{\langle \ba \rangle}|^s \cdot (\psi_{\sigma,u} \rdot f)_Q \circ \langle \ba \rangle
\quad (\textrm{\rom{1}}), \]
\[ \frac{\lambda_{\sigma,u}^{-n}}{(\psi_{\sigma,u})_P} \sum_{Q \in \Pc } \sum_{\langle \ba \rangle} 
e^{w c(\ba)}
|s||J_{\langle \ba \rangle}|^{s-1}  |\partial_v J_{\langle \ba \rangle}| \cdot (\psi_{\sigma,u} \rdot f)_Q \circ \langle \ba \rangle \quad (\textrm{\rom{2}}) \]
and 
\[ \frac{\lambda_{\sigma,u}^{-n}}{(\psi_{\sigma,u})_P} \sum_{Q \in \Pc } \sum_{\langle \ba \rangle} 
e^{w c(\ba)}
|J_{\langle \ba \rangle}|^{s} \left( f \rdot \partial_v \psi_{\sigma,u}+ \psi_{\sigma,u}\rdot \partial_v f  \right)_Q \circ \langle \ba \rangle .    \quad (\textrm{\rom{3}}) \]
Here, the inner sum is taken over $\langle \ba \rangle \in \Hc^n(P,Q)$.

The term (\rom{1}) is equal to $\left|\frac{\partial_v (\psi_{\sigma,u})_P}{(\psi_{\sigma,u})_P} (\Lt_{s,w}^n f)_P \right|$, whence bounded by $A_{U} \| \Lt_{\sigma,u}^n |f| \|_0$ for some $A_{U}=\sup_U \| \psi_{\sigma,u} \|_1 \|\psi_{\sigma,u}^{-1} \|_0$, which depends only on $U$ by perturbation theory. This is bounded by $A_{U} \|f \|_0$. The term (\rom{2}) is bounded by $M |s| \| f\|_0$, where $M$ is the distortion constant in Proposition \ref{distortion1}.
The term (\rom{3}) is bounded by $A_{U}\rho^{n} \|f \|_0+ \rho^{n} \|f\|_1$, up to constant. Taking a suitable $\widetilde{C}_U>0$, we obtain (1).

To prove (2), assume that eigenfunction and measure are normalised, i.e., $\int_I  \psi_{\sigma,u} \nu_{\sigma,u}=1$. For $f \in \mathcal C^1(\Pc)$, we have for any $n \geq 1$
\[ \mathcal L_{\sigma,u}^n f = \lambda_{\sigma,u} \rdot \psi_{\sigma,u} c(f)+ \Nc^n_{\sigma,u} f \]
by the spectral decomposition \eqref{sp:decomp}. It follows that
\[ \lambda_{\sigma,u}^{-n} \mathcal L_{\sigma,u}^n f = \psi_{\sigma,u} c(f)  + \lambda_{\sigma,u}^{-n}  \Nc_{\sigma,u}^n f,
\] 
which yields the identity $c(f)=\int_I f d\nu_{\sigma,u}$ by integrating against $\nu_{\sigma,u}$ and taking the limit as $n$ tends to infinity.
Due to the normalisation \eqref{def:normalop}, we have
\begin{align*} 
\Lt_{\sigma,u}^n f &=\lambda_{\sigma,u}^n \psi_{\sigma,u}^{-1} \Lc_{\sigma,u}^n  (\psi_{\sigma,u} \rdot f)  \\
&=\lambda_{\sigma,u}^n \int_I f d\mu_{\sigma,u} +O(r_{\sigma,u}^n \| \psi_{\sigma,u}^{-1} \|_{(1)}  \| \psi_{\sigma,u} \rdot f \|_{(1)} )
\end{align*}
with $r_{\sigma,u}<1$, which gives (2).
\end{proof}

\hide{-----

Integrate against $\mu$ and obtain
\begin{align}
\int \lambda^{-n} \mathcal L_{s,w} ^n f \mu=\int f  \left((\lambda^{-1} \mathcal L_{s,w} )^* \right)^n  \mu     = c(f) + o(1).
\end{align}
Using $(\lambda^{-1}\mathbb H)^*\mu=\mu$ and $\mu = \psi_2\mu_L$, 
\begin{align}
\int f\mu = c(f)  + o(1).
\end{align}
As $n$ tends to infinity, we obtain
\begin{align}
c(f) = \int f_2 \psi_2 \mu_L,
\end{align}
where $f_2$ is the component of $f$ in $\mathcal B_2.$

------}

\subsection{Key relation of $(\sigma,u)$ and $(1,0)$} \label{sub:lebesgue}

We aim to relate $\Lt_{\sigma,u}$ to $\Lt_{1,0}$ in a suitable way, in order to utilise the properties of $\mu_{1,0}$ 
proved in Lemma \ref{est:normalised}.

\begin{lem} \label{perturb:id}
For $(s,w)$ with $(\sigma,u) \in U$, there are constants $B_U>0$ and $A_{\sigma, u}>0$ such that 
\[ \| \Lt_{\sigma,u}^n f \|_0^2 \leq B_U A_{\sigma,u}^n \| \Lt_{1,0}^n(|f|^2) \|_0.       \]
\end{lem}

\begin{proof}
For $P \in \Pc$, we have
\begin{align*}
|(\Lt_{\sigma,u}^n f)_P|^2 &\leq \frac{\lambda_{\sigma,u}^{-2n}}{(\psi_{\sigma,u})_P^{2}} \left( \sum_Q \sum_{\langle \ba \rangle} 
e^{u c(\ba)}
|J_{\ba}|^\sigma  |(\psi_{\sigma,u} \rdot f)_Q| \circ \langle \ba \rangle   \right)^2  \\
&\leq \frac{\lambda_{\sigma,u}^{-2n}}{(\psi_{\sigma,u})_P^{2}} \left( 
\sum_{Q, \langle \ba \rangle} 
e^{2 u c(\ba)}
|J_{\ba}|^{2\sigma-1}   \right) \left(  
\sum_{Q,\langle \ba \rangle}  
|J_{\ba}| |(\psi_{\sigma,u} \rdot f)_Q|^2 \circ \langle \ba \rangle   \right)
\end{align*}
by the Cauchy--Schwartz inequality. 

The second factor is equal to $(\Lc_{1,0}^n |(\psi_{\sigma,u} \rdot f)|^2)_P$, while the rest satisfies
\begin{align*}
\frac{\lambda_{\sigma,u}^{-2n}}{(\psi_{\sigma,u})_P^{2}} \left( 
\sum_{Q, \langle \ba \rangle} 
e^{2 u c(\ba)}
|J_{\ba}|^{2\sigma-1}   \right)
&=
\lambda_{2\sigma-1, 2u}^n (\psi_{2\sigma-1, 2u})_P ( \Lt_{2\sigma-1, 2u}^n \psi_{2\sigma-1, 2u}^{-1})_P 
\\
&\leq 
\sup_U \lambda_{2\sigma-1, 2u}^n \| \psi_{2\sigma-1, 2u} \|_0 \| \psi_{2\sigma-1, 2u}^{-1} \|_0 
\end{align*}
where the first equality follows from  normalisation~\eqref{def:normalop}. By setting $A_{\sigma, u}= \frac{\lambda_{2\sigma-1, 2u}}{\lambda_{\sigma,u}^2}$ and taking the supremum over $P$, we obtain the desired inequality.
\end{proof}

\hide{------

To this end, first note that for $\al \in \Oc$ with non-empty $O_\al$, we have $T^2(O_\al)=I$. This leads to:

\begin{prop} \label{mixing}
The complex Gauss map $(I,T)$ is topologically mixing.
\end{prop}

\begin{proof}

\end{proof}

-------------}


\hide{-------

\begin{thm}[Hennion] 
Let $E$ and $F$ be Banach spaces with a compact inclusion $F \subseteq E$ and $\Lc$ be a bounded operator on $F$ such that for all $n \geq 1$ and $\varphi \in F$
\begin{equation*} 
\| \Lc^n \varphi \|_F \leq  \rho^n \|\varphi\|_F+ C_n \|\varphi\|_E
\end{equation*}
for suitable real constants $C_n$ and $\rho$ with $\rho< \rho(\Lc)$, where $\rho(\Lc)$ denotes the spectral radius. Then $\Lc$ is quasi-compact. 
\end{thm}

This inequality is called the Lasota--Yorke estimate. Theorem \ref{hennion} suggests that to bound the essential spectral radius of transfer operator, it is crucial to take a Banach space endowed with two norms satisfying the compact embedding assumption and Lasota--Yorke inequality. Consider the norms (\ref{norm1}) and (\ref{norm2}) on $C^1(I;\mathcal P)$, then we have:

-------}

\hide{
\begin{prop}
For real $(\sigma,\nu) \in \Sigma_0 \times W_0$, $\mathcal L_{\sigma,\nu}$ has a unique eigenvalue $\lambda(\sigma,\nu)$ of maximal modulus, which is real and simple, the dominant eigenvalue. The associated eigenfunction $f_{\sigma,\nu}$ is strictly positive. 
\end{prop}
\begin{proof}
\end{proof}
}

\hide{
For a subset $Y \subset I$, define
$$
\mathcal H^n_Y(P,Q) = \{\langle \alpha \rangle \in \mathcal H^n(P,Q) \colon \mathrm{Im}\langle \alpha \rangle \subset Y\}.
$$
Consider a family $Y_r$ of open balls in $\mathbb C$, where $r>0$ is a positive real parameter tending to zero. 
}

\section{Dolgopyat--Baladi--Vall\'ee estimate} \label{sec:dolgopyat}

In this section, we show the Dolgopyat-type uniform polynomial decay of transfer operator with respect to the $(t)$-norm. The main steps of the proof are parallel to Baladi--Vall\'ee \cite[\S3]{bal:val} and include Local Uniform Non-Integrability (Local UNI) property for the complex Gauss system that is modified with respect to the finite Markov partition, a version of Van der Corput lemma in dimension 2, and the spectral properties we settled in \S\ref{sec:acim}-\ref{sec:nor:lasota}.
Despite of the parallelism, we note that the proofs are longer and different in details due to the presence of cells in multiple dimensions as well as inverse branches between them.

\subsection{Main estimate and reduction to $L^2$-norm} \label{sub:maindolgopyat}

Our goal is to prove the following polynomial decay property for a family of transfer operators which we call Dolgopyat--Baladi--Vallée estimate.

As before, let $U$ be a neighbourhood of $(1,0)$ in Definition~\ref{def:modgrowth}.
\begin{thm} \label{main:dolgopyat}

There exist $\widetilde{C}, \widetilde{\gamma}>0$ such that for $(s,w)$ with $(\sigma,u) \in U$, and for any $n = [\widetilde{C} \log |t|]$ with $|t| \geq 1/\rho^2$, we have
\[ \| \Lt_{s,w}^n \|_{(t)} \ll \frac{1}{|t|^{\widetilde{\gamma}}}. \]
Here, the implied constant depends only on the given neighbourhood $U$.

For $0<\xi<1/10$, we have
\begin{align} \label{6.1}
 \|(I-\Lc_{s,w})^{-1} \|_{(t)} \ll |t|^\xi   .  \end{align}
\end{thm}

As in \cite{Dolgopyat}, the proof of  Theorem \ref{main:dolgopyat} can be reduced to the following $L^2$-norm estimate through the key relation in \S\ref{sub:lebesgue}.

\begin{prop} \label{main:L2} 
There exist $\widetilde{B}, \widetilde{\beta}>0$ such that for $(s,w)$ with $(\sigma,u) \in U$, and for any $n_0 = [\widetilde{B} \log |t|]$ with $|t| \geq 1/\rho^2$, we have 
\begin{align}\label{eqn:6.2}
    \int_I | \Lt_{s,w}^{n_0} (f)|^2 d \mu_{1,0}  \ll \frac{\| f\|_{(t)}^2}{|t|^{\widetilde{\beta}}} .
    \end{align}
Here, the implied constant depends only on the given neighbourhood $U$.
\end{prop}

Dolgopyat's estimate was first established for symbolic coding for Anosov flows, and Baladi--Vallée \cite{bal:val, bal:val2} adapted the argument to countable Markov shifts such as continued fractions. Avila--Gou\"ezel--Yoccoz \cite{Avila:G:Y} generalised Baladi--Vallée \cite{bal:val2} to the arbitrary dimension. We first prove Proposition \ref{main:L2} and  prove Theorem \ref{main:dolgopyat} at the end of this section.

We observe that 
\[  \int_I |\Lt_{s,w}^n f|^2 d\mu_{1,0}= \lambda_{\sigma,u}^{-2n}\sum_{P \in \Pc[2]} 
\int_P  \left(\psi_{\sigma,u}^{-2}\right)_P |(\Lc_{s,w}^n (\psi_{\sigma,u} \rdot f))_P|^2 dxdy  \]
since $\mu_{1,0}=\psi_{1,0}\nu_{1,0}$ by definition and $\nu_{1,0}$ is equal to the 2-dimensional Lebesgue measure by Thm.\,\ref{thm:ruelle}~(4). Put
\begin{align*}
    I_P := \int_P
\left(\psi_{\sigma,u}^{-2}\right)_P 
|(\Lc_{s,w}^n (\psi_{\sigma,u} \rdot f))_P|^2 dxdy 
\end{align*}

and expand it as
\begin{align}\label{L2:expansion}
    I_P = 
    \sum_{Q \in \Pc[2]} \sum_{\langle\ba\rangle,\langle \bb \rangle} 
        \int _P e^{w c(\ba)+\bar{w} c(\bb)} e^{it \phi_{\ba, \bb}}R^\sigma_{\ba,\bb} dxdy
\end{align}
where we let
\begin{align*}
    R_{\ba,\bb}^\sigma
    &:= \left(\psi_{\sigma,u}^{-2}\right)_P |J_{\ba}|^\sigma |J_{\bb}|^\sigma  \cdot (\psi_{\sigma,u} \rdot f)_Q \circ \langle \ba \rangle \cdot (\psi_{\sigma,u} \rdot \bar{f})_Q \circ \langle \bb \rangle
    \\
    \phi_{\ba, \bb}&:=\log|J_{\ba}|-\log|J_{\bb}|
\end{align*}
in order to simplify the notation. The inner sum in \eqref{L2:expansion} is taken over $\mathcal H^n(P,Q)^2$.

\hide{
    Each term in the summand on the right hand side is written as
    \begin{align*} 
    \int_P
    \left(\psi_{\sigma,u}^{-2}\right)_P 
    |(\Lc_{s,w}^n (\psi_{\sigma,u} \rdot f))_P|^2 dxdy &= \sum_{Q \in \Pc[2]} \sum_{(\ba , \bb )} \int_P e^{[w \sum_{j=1}^n c(\al_j)+\bar{w} \sum_{j=1}^n c(\beta_j)]} e^{it \log \frac{|J_{\ba}|}{|J_{\bb}|}} \\ 
    & \cdot \left(\psi_{\sigma,u}^{-2}\right)_P  |J_{\ba}|^\sigma |J_{\bb}|^\sigma (\psi_{\sigma,u} \rdot f)_Q \circ \langle \ba \rangle \cdot (\psi_{\sigma,u} \rdot \bar{f})_Q \circ \langle \bb \rangle dxdy \numberthis \label{L2:expansion} 
    \end{align*}
    where the sum is taken over admissible branches $(\langle \ba \rangle, \langle \bb \rangle) \in \Hc^n(P,Q) \times \Hc^n(P,Q)$.

    Set
    \[ \phi_{\ba, \bb}(x,y):=\log|J_{\ba}(x,y)|-\log|J_{\bb}(x,y)|   \]
    and 
    \[ R_{\ba,\bb}^\sigma:= \left(\psi_{\sigma,u}^{-2}\right)_P |J_{\ba}|^\sigma |J_{\bb}|^\sigma  \cdot (\psi_{\sigma,u} \rdot f)_Q \circ \langle \ba \rangle \cdot (\psi_{\sigma,u} \rdot \bar{f})_Q \circ \langle \bb \rangle . \]
}

To bound (\ref{L2:expansion}), we decompose it into two parts with respect to the following distance $\Delta$ on the set of inverse branches. For $\langle \ba \rangle, \langle \bb \rangle \in \Hc^n(P,Q)$, define the distance 
\begin{equation*}
\Delta(\ba,\bb) := \inf_{(x,y) \in P} 
\left| \left(\partial_z\phi_{\ba,\bb}(x,y),\partial_{\bar z} \phi_{\ba,\bb}(x,y)\right)\right|_2 
\end{equation*}
where $\partial_z$ and $\partial_{\bar z}$ respectively denote the derivative in $z=x+iy$ and $\bar z=x-iy$. Here $|\cdot|_2$ denotes the 2-norm of a vector.

Given $\varepsilon>0$, decompose $I_P$ as
$    I_P:
    = I_{P,1}+I_{P,2}
    $
where we define
\begin{align*}
    I_{P,1}& :=\sum_{Q \in \Pc[2]} \sum_{\Delta(\ba,\bb) \leq \varepsilon} \int_P 
    e^{wc(\ba)+\bar{w}c(\bb)}
    e^{it \phi_{\ba, \bb}}
    R_{\ba,\bb}^\sigma dxdy
\end{align*}
and
\begin{align*}
    I_{P,2}& :=\sum_{Q \in \Pc[2]} \sum_{\Delta(\ba,\bb) > \varepsilon} \int_P
    e^{wc(\ba)+\bar{w}c(\bb)} 
    e^{it \phi_{\ba, \bb}}
    R_{\ba,\bb}^\sigma dxdy.
\end{align*}

In the following subsections, we estimate $I_{P,1}$ by showing local UNI property, and $I_{P,2}$ by showing a 2-dimensional version of Van der Corput Lemma. Accordingly, we complete the proof of Theorem \ref{main:L2} and obtain the main estimate~\eqref{eqn:6.2}.

\subsection{Local Uniform Non-Integrability: Bounding $I_{P,1}$} \label{L2:1}

In order to bound $I_{P,1}$, we need technical Lebesgue measure properties of the complex Gauss system $(I,T)$. This is an analogue of Baladi--Vall\'ee \cite[\S3.2]{bal:val}, which is formulated algebraically as an adaptation of UNI condition of foliations in Dolgopyat \cite{Dolgopyat}. Since $T$ is not a full branch map, we modify the condition locally with respect to the finite Markov partition as follows.

\begin{prop}[Local UNI] \label{uni}
Let $P,Q \in \Pc[2]$ and $\langle \ba \rangle \in \Hc^n(P,Q)$. Then,
\begin{enumerate}
\item For any sufficiently small $a>0$, we have
\begin{align}\label{eq:finalbound}
\mathrm{Leb} \left( \bigcup_{\langle \bb \rangle \in \Hc^n(P,Q) \atop \Delta(\ba,\bb) \leq \rho^{an/2}} h_\bb (P) \right) \ll \rho^{an} .  
\end{align} 
\item There is a uniform constant $C>0$ such that for any direction $v$ and $w$, and for any $\langle \bb \rangle \in \Hc^n(P,Q),$
\[ \sup_{P \in \Pc} \sup_{(x,y) \in P} | \partial_w(\partial_v \phi_{\ba,\bb}(x,y))|_2 \leq C.\]
\end{enumerate}
\end{prop}

Before the proof, we first make the following observation. Recall from Proposition \ref{distortion2} that for $\langle \ba \rangle \in \Hc^n(P,Q)$, the linear fractional transformation $h_\ba$ corresponds to $\twobytwotiny{A_\ba}{B_\ba}{C_\ba}{D_\ba} \in \GL_2(\Oc)$, where the matrix is given by the identity
\begin{equation}\label{eq:gl2}
\twobytwo{A_\ba}{B_\ba}{C_\ba}{D_\ba}  =\twobytwo{0}{1}{1}{\al_1} \twobytwo{0}{1}{1}{\al_2} \cdots \twobytwo{0}{1}{1}{\al_n} 
\end{equation}
with determinant $\pm 1$ and $\ba=(\al_1, \cdots, \al_n) \in \Oc^n$. We have $\twobytwotiny{A_\ba}{C_\ba}{B_\ba}{D_\ba}$ for the corresponding dual branch $h_{\ba^*}$.

Recall that $ |J_\ba(x,y)|=|h_\ba'(z)|^2$. Proposition \ref{distortion1} allows us to see that for a fixed $\langle \ba \rangle \in \Hc^n(P,Q)$ and $\langle \bb \rangle$ of the same depth satisfying $\Delta(\ba,\bb)\leq \varepsilon$, we have 
\begin{align*}
\varepsilon &\geq \inf_{(x,y) \in P} \left| \left(\partial_z\phi_{\ba,\bb}(x,y),\partial_{\bar z} \phi_{\ba,\bb}(x,y)\right)\right|_2 
 \\
 &=\inf_{z \in P} \left|  \left( \frac{h_\ba''(z)}{h_\ba'(z)}-\frac{h_{\bb}''(z)}{h_{\bb}'(z)}, \frac{\overline{h_\ba''}(z)}{\overline{h_\ba'}(z)}-\frac{\overline{h_{\bb}''}(z)}{\overline{h_{\bb}'}(z)} \right) \right|  \\
 &=\inf_{z \in P} \left|  \frac{2\sqrt{2} (C_{\ba} D_{\bb}- C_{\bb} D_{\ba})}{(C_{\ba} z+D_{\ba})(C_{\bb} z+D_{\bb})} \right| .
\end{align*}

Observe that $|h_\ba'(z)|=\frac{1}{|C_\ba z+D_\ba|^2}$. Then we obtain 
\begin{align*}
 |(C_\ba z+D_\ba)^{-1}(C_\bb z+D_\bb)^{-1}| &= |h_\ba'(z)|^{1/2} |h_\bb'(z)|^{1/2}  \\
& \geq \frac{1}{L_1^{1/2}}  |h_\ba'(0)|^{1/2} |h_\bb'(0)|^{1/2}
\end{align*} 
by Proposition \ref{distortion2} (where $L_2=1/L_1$).
It follows that 
\[ \varepsilon \geq \frac{2\sqrt{2}}{L_1^{1/2}} \left| \frac{C_\ba}{D_\ba}-\frac{C_\bb}{D_\bb}  \right| = \frac{2\sqrt{2}}{L_1^{1/2}}|h_{\ba^*}(0) - h_{\bb^*}(0)| . \]

\hide{----

Remark that the expression is simply equal to $2L^{-1/2} |h_{\al^*}(0)-h_{\beta^*}(0)|_2$.
Note that for any $P \in \mathcal P$, we have a sequence of points in $P$ with arbitrary large last partial quotient, thus we have a sequence of points in $P^*$ with arbitrary large first partial quotient. It follows that any $P^*$ contains the point $(0,0)$.

\begin{lem}\label{lem:contraction}
There is a constant $L>0$ such that for each $P \in \mathcal P$ such that $P \subset TO_\al$,  there exist subsets  $Q_1, \cdots, Q_i  \in \mathcal P$ of $TO_\al$  that contains $0$ in its closure and 
$$\left |\frac{h_\alpha(P\cup Q_1 \cup \cdots \cup Q_i )}{h_\al'(0)} \right| < L.$$
$$\left |\frac{h_\alpha^*(P\cup Q_1 \cup \cdots \cup Q_i )}{h_\al^*(0)} \right| < L.$$
\end{lem}

------}

\begin{proof}[Proof of Proposition \ref{uni}]

(1) By the above observation, if the distance $\Delta(\ba, \bb) \leq \varepsilon$ then $|h_{\ba^*}(0)-h_{\bb^*}(0)| \leq 2\sqrt{2L_1} \varepsilon$. 
Recall Proposition \ref{dual:metric} that for $h_{\bb^*} \in \Hc^{*n}$ and $P \in \mathcal P$, 
$$\mathrm{Diam}(h_{\bb^*}(P^*)) \leq R_d^{2(n-1)} |1-R_d|^{-1}  \ll \rho^{a n/2}$$
with any sufficiently small $0<a<1$. Thus if we take $\varepsilon \leq \rho^{an/2}$, then
$$\mathrm{Diam}\left( \bigcup_{\langle \bb \rangle \in \Hc^n(P,Q) \atop \Delta(\ba,\bb) \leq \rho^{an}}  h_{\bb^*}(P^*)\right) \ll \rho^{an/2},$$
which implies that
\begin{equation} \label{leb:dual}
\mathrm{Leb}\left( \bigcup_{\langle \bb \rangle \in \Hc^n(P,Q) \atop \Delta(\ba,\bb) \leq \rho^{an}}  h_{\bb^*}(P^*)\right)  \ll \rho^{an}.
\end{equation}

Note that 
for any $h_\ba \in \Hc^n$ and $h_{\ba^*} \in \Hc^{*n}$, 
\[ \mathrm{Leb}(h_\ba(P))= \int_P |J_{\ba}(x,y)| dxdy \leq \sup_{z \in I} |h_\ba'(z)|^2  \]
and 
\[\mathrm{Leb}(h_{\ba^*}(P^*))=\int_{P^*} |J_{\ba^*}(x,y)| dxdy \geq \inf_{z^* \in I^*} |h_{\ba^*}'(z^*)|^2 . \]
By Remark \ref{rem:distortion}, we obtain $\sup_I |h_\ba'|^2 \leq L_2^2 \cdot \inf_{I^*} |h_{\ba^*}'|^2$, hence 
\[ \mathrm{Leb}(h_\ba(P)) \leq L_2^2 \cdot \mathrm{Leb}(h_{\ba^*}(P^*)). \]
Since the cells $h_{\bb^*}(P^*)$ are disjoint in the union \eqref{leb:dual}, we obtain \eqref{eq:finalbound}.

\hide{-----


Now if $\beta \neq \beta' \in \Hc^n(P,Q),$ the sets $h_{\beta^*}(P^*)$ and $h_{\beta'^*}(P^*)$ are disjoint. Indeed, if $h_{\beta^*}(z^*)=h_{\beta'^*}(w^*)$ for some $z,w \in P,$ then $z^* = T^{*n} h_{\beta^*}(z^*)=T^{*n} h_{\beta'^*}(w^*)=w^*,$ and thus $\beta = \beta'$.
It follows that
\[ \mathrm{Leb} \left( \bigcup_{\langle \beta \rangle \in \Hc^n(P,Q) \atop \Delta(\al,\beta) < \rho^{an}} \langle \beta \rangle(P) \right) \leq (LL^*) \mathrm{Leb} \left( \bigcup_{\langle \beta \rangle \in \Hc^n(P,Q) \atop \Delta(\al,\beta) < \rho^{an}} h_{\beta^*} (P^*) \right)
 \ll \rho^{an} . \]

---------}

(2) Observe that we have
\begin{align*}
\partial_w (\partial_v \phi_{\ba,\bb}) &= w_1 v_1 \left( \frac{h_\ba''' h_\ba' - h_\ba''^2}{h_\ba'^2}-\frac{h_\bb''' h_\bb' - h_\bb''^2}{h_\bb'^2} \right) \\
& \ \ \ \ \ + w_2 v_2 \left( \frac{\overline{h_\ba'''} \overline{h_\ba'}-\overline{h_\ba''^2}}{\overline{h_\ba'^2}}-\frac{\overline{h_\bb'''} \overline{h_\bb'}-\overline{h_\bb''^2}}{\overline{h_\bb'^2}} \right) . 
\end{align*}
Thus, to bound $|\partial_w (\partial_v \phi_{\ba,\bb})|_2$, it suffices to show that the right hand side of  
\[
\left| \frac{h_\ba''' h_\ba' - h_\ba''^2}{h_\ba'^2}  \right| = \left| \frac{h_\ba'''}{h_\ba'}- \frac{h_\ba''^2}{h_\ba'^2}  \right| \leq  \left| \frac{h_\ba'''}{h_\ba'} \right|+ \left| \frac{h_\ba''^2}{h_\ba'^2}  \right|  
\]
has a uniform upper bound on $P$. Recall from Proposition \ref{distortion1} that the second term is bounded by $M^2$. For the first term, if $|\al|=1$, we have $\left| \frac{h_\al'''(z)}{h_\al'(z)} \right| = \frac{6}{|z+\al|^2}$, which is uniformly bounded since $|z+\al|>1$. Hence for any $|\ba|=n \geq 1$, we obtain a constant $N>0$ such that $\left| \frac{h_\ba'''(z)}{h_\ba'(z)} \right| \leq N$ in the same way as in Proposition \ref{distortion1}. 
\end{proof}

Finally, we observe the following non-trivial consequence of bounded distortion, which plays a crucial role in the proof of Proposition \ref{thm:I_1}.

\begin{lem} \label{key:distortion}
For $(\sigma,u) \in U$, there are uniform constants $C_{U}^1>0$ and $C_{U}^2>0$ such that 
\begin{enumerate}
\item For any $\langle \ba \rangle \in \Hc^n(P,Q)$, we have 
\[ C_{U}^1 \frac{\| J_\ba \|_0^\sigma}{\lambda_{\sigma,u}^n} \leq \mu_{\sigma,u}(h_\ba (P)) \leq C_U^2 \frac{\| J_\ba \|_0^\sigma}{ \lambda_{\sigma,u}^n} .  \]
\item For any $\mathcal{E} \subseteq \Hc^n(P,Q)$ and $J=\bigcup_{\langle \ba \rangle \in \mathcal{E}} h_\ba (P)$, we have 
\[ \mu_{\sigma,u}(J) \ll A_{\sigma,u}^n  \Leb(J)^{1/2} ,  \]
where $A_{\sigma,u}$ is as in Lemma \ref{perturb:id}.
\end{enumerate}
\end{lem}

\begin{proof}
(1) Recall from \eqref{def:normalop} that 
\[ \sum_{P\in \Pc} \int_P \Lt_{\sigma, u}^n f d \mu_{\sigma,u}=  \sum_{P\in \Pc} \int_P f d\mu_{\sigma,u} \]
holds for all $f \in L^1(\Pc)$. Taking $f=\chi_{h_\ba(P)}$ gives the identity 
\begin{align*} 
\mu_{\sigma,u}(h_\ba(P)) &= \frac{e^{uc(\ba)}}{\lambda_{\sigma,u}^n} \int_{h_\al(P)} \psi_{\sigma,u}^{-1} |J_\ba|^\sigma \rdot \psi_{\sigma,u} \circ \langle \ba \rangle d \mu_{\sigma,u} .
\end{align*}
Thus by bounded distortion from Proposition \ref{distortion2} yields the bound (1).

(2) Recall that $\mu_{\sigma,u}=\psi_{\sigma,u} \nu_{\sigma,u}$ where $\mu_{1,0}$ is equivalent to Lebesgue, we observe
\begin{align*} 
\mu_{\sigma,u}(J) &\leq \sum_{\langle \ba \rangle \in \mathcal{E}} \mu_{\sigma,u}(h_\ba(P)) \\
&\ll \sum_{\langle \ba \rangle \in \mathcal{E}} \frac{e^{u c(\ba)}}{\lambda_{\sigma,u}^n} \cdot \Leb(h_\ba(P))^\sigma \\
&\ll \lambda_{\sigma,u}^{-n} \left( \sum_{\langle \ba \rangle \in \mathcal{E}} e^{2u c(\ba)} \cdot \Leb(h_\ba(P))^{2\sigma-1} \right)^{1/2} \left( \sum_{\langle \ba \rangle \in \mathcal{E}} \Leb(h_\ba(P)) \right)^{1/2}
\end{align*}
by Cauchy--Schwarz inequality. Then by Lemma \ref{perturb:id}, the first factor is bounded by $\lambda_{2\sigma-1,2u}^n$ (up to a uniform constant). Since all the cells $h_\ba(P)$ are disjoint, we obtain the statement.
\end{proof}

\hide{-----
again using bounded distortion, the pairwise ratios of three quantities $\nu_{\sigma,u}(\langle \alpha \rangle(P))$, $\lambda_{\sigma,u}^{-n} e^{u c(\al)} \rdot \Leb(\langle \alpha \rangle(P))^\sigma$ and 
\begin{align*} 
\frac{e^{u c(\al)}}{\lambda_{\sigma,u}^n} \int_P |J_\al(x,y)|^\sigma dxdy 
\end{align*}
admit the uniform upper and lower bounds when $(\sigma,u)$ varies in $K$. Hence for $\mathcal{E} \subseteq \Hc^n(P,Q)$ and $J$, we have 
-----------}

Now we are ready to present:

\begin{prop} \label{thm:I_1}
For any $0<a<1$ and $n \geq 1$, the integral $I_{P,1}$ of \eqref{L2:expansion} restricted to pairs $(\langle \ba \rangle, \langle \bb \rangle)$ of depth $n$ for which $\Delta(\ba,\bb) \leq \rho^{an/2}$ satisfies 
\[ |I_{P,1}| \ll \rho^{an/2} \| f \|_0^2 . \]

\end{prop}

\begin{proof}
Notice that for some $M_U > 0$, we have 
\begin{align*} 
|I_{P,1}| &\leq  M_U \frac{\| f \|_0^2}{\lambda_{\sigma,u}^{2n}} \sum_{Q \in \Pc[2]} \sum_{\Delta(\ba, \bb)\leq \varepsilon} 
e^{w c(\ba)+\bar{w}c(\bb)}  
\int_P 
 |J_{\ba}|^\sigma |J_{\bb}|^\sigma dxdy .    
 \end{align*}
 
 Observe that 
 \begin{align*} 
  \int_P  |J_{\ba}|^\sigma |J_{\bb}|^\sigma dxdy &\leq \sup_{I} |J_\ba|^\sigma \sup_{I} |J_\bb|^\sigma \\
 &\leq  (L_2^2 \cdot \inf_{P} |J_\ba|^\sigma) (L_2^2 \cdot \inf_{P} |J_\bb|^\sigma) \\
 &\leq \left( \int_P |J_\ba|^\sigma dxdy \right)\left( \int_P |J_\bb|^\sigma dxdy \right)    
 \end{align*}
by Proposition \ref{distortion2} and the mean value theorem for integrals in dim 2.

Then by Lemma \ref{key:distortion}, up to a positive constant (depending only on $U$), we have
\begin{align*} 
|I_{P,1}| &\ll \| f \|_0^2  \sum_{\Delta(\ba, \bb)\leq \varepsilon} \mu_{\sigma,u}(h_\ba(P)) \mu_{\sigma,u}(h_\bb(P)) \\ &\ll \| f \|_0^2 \sum_{\ba}\mu_{\sigma,u}(h_\ba(P)) \left( \sum_{\Delta(\ba,\bb) \leq \varepsilon} \mu_{\sigma,u}(h_\bb(P))  \right)  \\
&\ll \| f \|_0^2 A_{\sigma,u}^n \mathrm{Leb}(h_\ba(P))^{1/2} \mathrm{Leb}(\cup_{\Delta(\ba,\bb) \leq \varepsilon} h_\bb(P))^{1/2} .
\end{align*}
Finally, UNI property from Proposition \ref{uni}.(1) completes the proof by taking $\varepsilon$ in the scale $\rho^{an/2}$.
\end{proof}

\hide{------
&\leq M_K \frac{\| f \|_0^2}{\lambda_{\sigma,u}^{2n}} \sum_{Q \in \Pc[2]} \sum_{\Delta(\ba, \bb)\leq \varepsilon} e^{[(*)]} \sup_{P^+} |J_\ba|^\sigma \sup_{P^+} |J_\bb|^\sigma \\
 &\leq M_K \frac{\| f \|_0^2}{\lambda_{\sigma,u}^{2n}} \sum_{Q \in \Pc[2]} \sum_{\Delta(\ba, \bb)\leq \varepsilon} e^{[(*)]} (L_2^2 \cdot \inf_{P^-} |J_\ba|^\sigma) (L_2^2 \cdot \inf_{P^-} |J_\bb|^\sigma) \\
 &\leq M_K \frac{\| f \|_0^2}{\lambda_{\sigma,u}^{2n}} \cdot L_2^4 \sum_{Q \in \Pc[2]} \sum_{\Delta(\ba, \bb)\leq \varepsilon} e^{[(*)]} \left( \int_P |J_\ba|^\sigma d(x,y) \right)\left( \int_P |J_\bb|^\sigma d(x,y) \right)
 
 --------}

\subsection{Van der Corput in dimension two: Bounding $I_{P,2}$} \label{L2:2}

Now it remains to bound the sum $I_{P,2}$ of \eqref{L2:expansion}. 
The strategy is to bound each term of $I_{P,2}$ by taking advantage of the oscillation in the integrand.
We begin by having a form of Van der Corput lemma in dimension two. 


Let $\Omega \subset \R^2$ be a domain having a piecewise smooth boundary. 
For $\phi \in C^2(\Omega)$, set $M_0(\phi):=\sup_\Omega |\phi|$ and $M_1(\phi):=\sup_\Omega |\nabla \phi|_2$ where $|\cdot |_2$ denotes the 2-norm.
Also we set $M_2(\phi)=\sup_{D^2}\sup_{\Omega}  |D^2 \phi|$ where the outer supremum is taken over $D^2 \in \{\partial_x^2 , \partial _x \partial _y, \partial _y^2\}$. Put $m_1(\phi)=\inf_\Omega |\nabla \phi |_2$. Finally, write $\textrm{Vol}_2(\Omega)$ for the area of $\Omega$ and $\mathrm{Vol}_1(\partial \Omega)$ for its circumference. 


\begin{lem}
\label{vandercorput}
Suppose $\phi \in C^2(\Omega)$ and $\rho \in C^1(\Omega)$.
For $\lambda \in \R$, define the integral 
\[ I(\lambda)= \int \int_\Omega e^{i\lambda \phi(x,y)} \rho(x,y) dxdy .\]
Then we have a bound: 
\begin{align}\label{vdcineq}
|\lambda I(\lambda)| \leq \frac{M_0(\rho)}{m_1(\phi)} \mathrm{Vol}_1(\partial \Omega)+ \left( \frac{M_1(\rho)}{m_1(\phi)}+\frac{5 M_0(\rho)M_2(\phi)}{m_1(\phi)^2}   \right) \mathrm{Vol}_2(\Omega).
\end{align}

\end{lem}

\begin{proof}
Let $\omega = dx \wedge dy$ be the standard volume form on $\mathbb R^2$.
Put 
$$
\alpha = e^{i\lambda \phi} \frac{\rho}{|\nabla \phi|_2^2} \iota_{\nabla \phi} \omega
$$
where $\iota_{\nabla \phi}$ denotes the contraction by $\nabla \phi$.
Differentiating, we obtain
\begin{align}
    d\alpha = i \lambda e^{i \lambda \phi}  \rho \omega  + e^{i\lambda \phi} d \left(\frac{\rho}{|\nabla \phi|_2^2} \iota_{\nabla \phi} \omega\right)
\end{align}
by using $d\phi \wedge i_{\nabla\phi} \omega = \omega$.
The second term can be rewritten using
$$
d \left(\frac{\rho}{|\nabla \phi|_2^2} \iota_{\nabla \phi} \omega\right) = \nabla \cdot \left( \frac{\rho}{|\nabla \phi|_2^2} \nabla\phi \right)\omega
$$
which holds because for any $f$ we have an identity $d(f\iota_{\nabla\phi}\omega) = \nabla \cdot ( f \nabla \phi ) \omega$.
By Green's theorem, we have $\int_{\Omega} d\alpha = \int_{\partial \Omega} \alpha$, which yields
$$
i \lambda \int_{\Omega} e^{i \lambda \phi} \rho \omega = \int_{\partial \Omega} \alpha 
- \int_{\Omega} \nabla \cdot \left( \frac{\rho}{|\nabla \phi|_2^2} \nabla\phi \right)\omega.
$$
The first integral is bounded by $m_1(\phi)^{-2}M_0(\rho) \operatorname{Vol}_1\left(\partial \Omega\right)$.
To bound the second integral, we use
$$
\nabla \cdot \left( \frac{\rho}{|\nabla \phi|_2^2} \nabla\phi\right)
=\frac{(\nabla\rho) \cdot ( \nabla \phi)}{|\nabla \phi|_2^2}  + 
\frac{\rho\nabla \phi}{|\nabla \phi|_2^2}  
+(\rho\nabla \phi) \cdot (\nabla |\nabla \phi|_2^{-2})
$$
whose first and second summands have absolute values bounded by $M_1(\rho)m_1(\phi)^{-1}$ and $M_0(\rho)M_2(\rho)m_1(\phi)^{-2}$, respectively.
For the last summand, a direct computation shows
$$
|(\rho\nabla \phi) \cdot (\nabla |\nabla \phi|_2^{-2})|= \frac{M_0(\rho)}{|\nabla\phi|^4}(\nabla \phi)\cdot(\nabla|\nabla\phi|_2^2)  \le \frac{4 M_0(\rho)M_2(\phi)}{m_1(\phi)^2}.
$$
Summing up, we obtain \eqref{vdcineq}.
\end{proof}
\hide{----
\begin{lem}
Let $J \subseteq \R^2$ with $J=J_1 \times J_2$. Suppose $\psi$ is a real-valued function such that 
\[\inf_{i=1,2} \left| \frac{\partial}{\partial x_i} \psi(x_1,x_2) \right| \geq \Delta \ \mbox{and } \sup_{i=1,2} \left| \frac{\partial^2}{\partial {x_i}^2} \psi(x_1,x_2) \right| \leq Q \]
for some uniform constants $\Delta, Q >0$ and all $(x_1,x_2) \in J$. Then for all $t \neq 0 \in \R$ and smooth function $\phi$, we have 
\[ \left| \int_J e^{it \psi(x,y)} \phi(x,y) d(x,y)   \right| \leq \frac{\|\phi \|_1 |J_1|}{|t|} \left(\frac{2}{\Delta}+\frac{|J_2|}{\Delta}+ \frac{Q |J_2|}{\Delta^2}\right) .  \]

\end{lem}

\begin{proof}
The argument is basically due to the use of Fubini theorem and integration parts for partial derivatives. By Fubini, we have
\[ I(t):=\int_J e^{it \psi(x,y)} \phi(x,y) d(x,y) = \int_{J_1} \left( \int_{J_2} \frac{ \phi(x,y) }{it \partial_y \psi(x,y)} \rdot \partial_y e^{it \psi(x,y)}dy \right) dx  \]
where $\partial_y$ is short for $\partial/ \partial y$.
Then integration by parts gives
\begin{align*}
I(t)&= \int_{J_1} \left( \frac{ \phi(x,y) }{it \partial_y \psi(x,y)} \rdot e^{it \psi(x,y)} \Big\vert_{\partial J_2} - \int_{J_2} \partial_y \left( \frac{ \phi(x,y) }{it \partial_y \psi(x,y)}  \right) \rdot e^{it \psi(x,y)} dy  \right) dx \\
& \leq \int_{J_1}  \frac{2 \phi(x,b) }{it \partial_y \psi(x,b)} \rdot e^{it \psi(x,b)} dx  \\
&\ \ \ \ \ \ \ - \int_{J_1} \int_{J_2} \frac{1}{it} \left( \frac{ \partial_y \phi(x,y) \partial_y \psi(x,y)- \phi(x,y) \partial^2_y \psi(x,y) }{ (\partial_y \psi(x,y))^2}  \right) \rdot e^{it \psi(x,y)} dy  dx 
\end{align*}
where $b$ is one boundary point of the interval $J_2$. Thus we have
\begin{align*}
|I(t)|& \leq \int_{J_1}  \frac{2 \| \phi \|_0 }{|t| \Delta} dx + \int_{J_1} \int_{J_2} \frac{1}{|t|} \left( \frac{\| \partial_y \phi \|_0}{\Delta}+ \frac{\|\phi\|_0 Q}{\Delta^2}\right) dy dx \\
&\leq \frac{\|\phi\|_1 |J_1|}{|t|} \left(\frac{2}{\Delta}+ \frac{|J_2|}{\Delta}+\frac{Q |J_2|}{\Delta^2}  \right).
\end{align*}
\end{proof}

------}


\begin{prop}\label{prop:boundingI2}
For all $a$ with $0 < a < \frac 1 4$, there is $n_0$ such that the integral $I_{P,2}$ of \eqref{L2:expansion} for the depth $n=n_0$ with $\Delta(\ba,\bb) \geq \rho^{an_0}$ and for any $|t| \geq 1/\rho^2$ satisfies 
\[ |I_{P,2}| \ll \rho^{(1-4a)\frac{n_0}{2}} \| f\|_{(t)}^2  . 
\]
\end{prop}

\begin{proof}

Recall that 
\begin{equation} \label{dolg:I2} 
I_{P,2}= \lambda_{\sigma, u}^{-2n} \sum_{Q \in \Pc[2]} \sum_{\Delta(\ba,\bb)\geq \varepsilon} 
e^{w c(\ba)+\bar{w} c(\bb)} 
\int_P e^{it \phi_{\ba,\bb}(x,y)} R_{\ba,\bb}^\sigma(x,y) dxdy \end{equation}
and by Lasota--Yorke arguments used in Lemma \ref{est:normalised}, we obtain
\[ \| R_{\ba,\bb}^\sigma \|_{(1)} \ll \|J_\ba \|_0^\sigma \|J_\bb \|_0^\sigma \| f \|_{(t)}^2 (1+\rho^{n_0/2}|t|) .  \]

Since $P$ is a bounded domain with piecewise smooth boundary, by applying Lemma \ref{vandercorput} to the oscillatory integral for each $P$ in \eqref{dolg:I2}, we have 
\[ |I_{P,2}| \leq M_U \| f\|_{(t)}^2  \frac{(1+\rho^{n_0/2}|t|)}{|t|} \left( \frac{\mathrm{Vol}_1(\partial P)+\mathrm{Vol}_2(P)}{\varepsilon/\sqrt{2}} +\frac{C}{(\varepsilon/\sqrt{2})^2} \mathrm{Vol}_2(P)  \right)    \]
for some $M_U>0$, where $C$ is the UNI constant  from Proposition \ref{uni}.(2).  
Here we used the identity $\sqrt {2} |\nabla\phi_{\ba,\bb}|_2=|(\partial_z\phi_{\ba,\bb},\partial_{\bar z} \phi_{\ba,\bb})|_2$.

It remains to take $\varepsilon=\rho^{a n_0}$ and $n_0$ in a suitable scale. 
Setting $n_0:=[ m \log|t| ]$ with $m$ small enough to have $
{(1+\rho^{n_0/2}|t|)} 
\left(\frac{\mathrm{Vol}_1(\partial P)+\mathrm{Vol}_2(P)}{|t| \rho^{an_0}} +\frac{C}{|t| \rho^{2a n_0}} \mathrm{Vol}_2(P)\right)$ decaying polynomially in $|t|$, we conclude the proof.
\end{proof}


\begin{proof}[End of Theorem \ref{main:dolgopyat}] \label{rem:xi}

We conclude by showing that Proposition~\ref{main:L2} implies Theorem \ref{main:dolgopyat}. 

For the first assertion, set $n_0=n_0(t)=[\widetilde{B} \log |t|] \geq 1$. For $n=n(t) = [\widetilde{C} \log |t|]$, we have 
\begin{align*}
\| \Lt_{s,w}^n f \|_0^2 &\leq \| \Lt_{\sigma, u}^{n-n_0} (| \Lt_{s,w}^{n_0} f  |) \|_0^2 \\
&\leq B_U A_{\sigma,u}^{n-n_0} \| \Lt_{1,0}^{n-n_0} (| \Lt_{s,w}^{n_0} f  |^2)   \|_0 
\end{align*}
by Lemma \ref{perturb:id}. Recall from Lemma \ref{est:normalised}.(2) that there is a gap in the spectrum of $\Lt_{1,0}$, which yields 
\begin{align*} 
\| \Lt_{s,w}^n f \|_0^2 &\leq \widetilde{B}_U A_{\sigma,u}^{n-n_0} \left( \int_I | \Lt_{s,w}^{n_0} f |^2 d\mu_{1,0} + r_{1,0}^{n-n_0} |t| \| f \|_{(t)}^2  \right) \\
& \leq \widetilde{B}_U A_{\sigma,u}^{n-n_0} \left( \frac{1}{|t|^{\widetilde{\beta}}}+ r_{1,0}^{n-n_0} |t|  \right) \| f \|_{(t)}^2 \numberthis \label{L2:sup}
\end{align*}
by Proposition \ref{main:L2}. Choose $\widetilde C>0$ large enough so that $r_{1,0}^{n-n_0}|t| <|t|^{-\widetilde \beta}$ and a sufficiently small neighbourhood $U$ so that 
$A_{\sigma,u}^{n-n_0}< |t|^{\widetilde{\beta}/2}.$
Then \eqref{L2:sup} becomes 
\begin{equation} \label{L2:sup:fin}
\| \Lt_{s,w}^n f \|_0 \ll   \frac{\| f\|_{(t)}}{|t|^{\widetilde{\beta}/4}} . 
\end{equation}

By using Lemma \ref{est:normalised}.(1) twice and \eqref{L2:sup:fin}, for $n \geq 2n_0$, we obtain
\begin{equation} \label{L2:c1:fin}
\| \Lt_{s,w}^n f \|_{(t)} \ll   \frac{\| f\|_{(t)}}{|t|^{\widetilde{\gamma}}} 
\end{equation} 
for some $\widetilde{\gamma}>0$, which implies the first bound for normalised family in Theorem \ref{main:dolgopyat}. Returning to the operator $\Lc_{s,w}$, we obtain the final bound with a suitable choice of implicit constants.

For the second statement, choose $n_0$ from Proposition~\ref{prop:boundingI2} and some $a$ with $\frac{1}{5}<a<\frac{1}{4}$, then we can take a real neighborhood of $(\sigma,u)$ small enough to ensure $A_{\sigma,u}\rho^{a/2} \leq \rho^{(1-4a)/2}$ since $a/2>(1-4a)/2>0$. Together with Proposition~\ref{thm:I_1}, this gives 
\[ I_P \ll \rho^{(1-4a)\frac{n_0}{2}} \|f\|_{(t)}^2. \] 

Accordingly by writing any integer $n=kn_0+r$ with $r<n_0$, using \eqref{L2:sup:fin} and \eqref{L2:c1:fin}, we have $\| \Lt_{s,w}^n \|_{(t)} \ll |t|^\xi$ for some $\xi$ with $0<\xi<(1-4a-\varepsilon)/2$. Thus $\xi$ can be any value between $0$ and $1/10$. 
\end{proof}

\hide{----
\begin{rem} \label{rem:xi}
More detailed computations in Baladi--Vallée \cite[\S3.3]{bal:val} show that the constant $\xi$ in Theorem \ref{main:dolgopyat} can be taken between $0$ and $\frac 1 9$  by choosing $a$ in Proposition~\ref{thm:I_1} and Proposition~\ref{prop:boundingI2} with $\frac 2 9 < a < \frac 1 4 $.
\end{rem}
-----}


\section{Gaussian \rom{1}} \label{sec:clt1}

In this section, we observe the central limit theorem for continuous trajectories of $(I,T)$. For $z \in I \cap (\C \backslash K)$, recall that we defined
\[ 
C_n(z)= \sum_{j=1}^n c(\al_j) 
\]
where $z=[0; \al_1, \al_2, \ldots ]$ with $\al_j=[ \frac{1}{T^{j-1}(z)}]$. We show that $C_n$, where $z$ is distributed with law $\mu_{1,0}$ from Theorem \ref{thm:ruelle}, follows the asymptotic normal distribution as $n$ goes to infinity.

First we state the following criterion due to Hwang, used in Baladi--Vall\'ee \cite[Theorem 0]{bal:val}. This says that the Quasi-power estimate of the moment generating function implies the Gaussian behavior.

\begin{thm}[Hwang's Quasi-Power Theorem]  \label{thm:hwang} 

Assume that the moment generating functions for a sequence of functions $X_N$ on probability space $(\Xi_N, \Pbb_N)$ are analytic in a neighbourhood $W$ of zero, and 
\[ \Eb_N[\exp(w X_N)\,|\,\Xi_N]= \exp(\beta_N U(w)+ V(w))(1+O(\kappa_N^{-1})) \] with $\beta_N, \kappa_N \rightarrow \infty$ as $N \rightarrow \infty$, $U(w), V(w)$ analytic on $W$, and $U''(0) \neq 0$.

\begin{enumerate}
\item The distribution of $X_N$  is asymptotically Gaussian with the speed of convergence $O(\kappa_N^{-1}+\beta_N^{-1/2})$, i.e.
\[\Pbb_N\left[\frac{X_N-\beta_N U'(0)}{\sqrt{\beta_N}}\leq u\,\bigg|\,\Xi_N\right]=\frac{1}{\sqrt{2\pi}}\int_{-\infty}^u e^{-\frac{t^2}{2}}dt+O\left(\frac{1}{\kappa_N}+\frac{1}{\beta_N^{1/2}}\right) \]
where the implicit constant is independent of $u$.

\item The expectation and variance of $X_N$ satisfy
 \begin{align*}
       \Eb[X_N\,|\,\Xi_N] &= \beta_N U'(0)+V'(0)+O(\kappa_N^{-1}), \\ \Vb[X_N\,|\,\Xi_N]  &= \beta_N U''(0)+V''(0)+O(\kappa_N^{-1}) .
\end{align*}    
\end{enumerate}
\hide{-------
 \begin{align*}
       \Eb[X_N\,|\,\Xi_N] &= \beta_N U'(0)+V'(0)+O(\kappa_N^{-1}), \\ \Vb[X_N\,|\,\Xi_N]  &= \beta_N U''(0)+V''(0)+O(\kappa_N^{-1}) .
\end{align*}    
Furthermore, the distribution of $X_N$ on $\Xi_N$ is asymptotically Gaussian with speed of convergence $O(\kappa_N^{-1}+\beta_N^{-1/2})$, i.e.,
\begin{align*}
\Pbb_N\left[\frac{X_N-\beta_N U'(0)}{\sqrt{\beta_N}}\leq u\,\bigg|\,\Xi_N\right]=\frac{1}{\sqrt{2\pi}}\int_{-\infty}^u \exp\left(-\frac{t^2}{2}\right)dt+O\left(\frac{1}{\kappa_N+\beta_N^{1/2}}\right)
\end{align*}
---------}
\end{thm}

Recall the moment generating function of a random variable $C_n$ on the probability space $(I, \mu_{1,0})$: Let $\psi=\psi_{1,0}$ and $\mu=\mu_{1,0}$. Then we have
\begin{align*} 
\Eb[\exp(w C_n)] &= \int_I \exp(w C_n(x,y)) \rdot \psi(x,y) d\mu(x,y) \\
&=\sum_{\langle \ba \rangle \in \Hc^n} 
e^{w c(\ba)}  
\sum_{P \in \Pc}  \int_{h_\ba (P)} \psi(x,y) dxdy \numberthis \label{def:mgf}
\end{align*}



where $\langle \ba \rangle=\langle \alpha_n \rangle_Q^{R_{n-1}} \circ \cdots \circ \langle \alpha_1 \rangle^P_{R_{1}}$ for some $P, R_1, \cdots, R_{n-1}, Q$ in the set of all admissible length $n$-sequences of inverse branch, which is given by
 \[ \Hc^n= \bigcup_{P,Q \in \Pc} \Hc^n(P,Q).    \]

We further observe that (\ref{def:mgf}) can be written in terms of the weighted transfer operator. By the change of variable $(x,y)=h_\ba(X,Y)$, we obtain
\begin{align*} 
\Eb[\exp(w C_n)] &= \sum_{\langle \ba \rangle \in \Hc^n} 
e^{w c(\ba)} 
\sum_{P \in \Pc}  \int_{P} |J_\ba(X,Y)| \cdot \psi \circ h_\ba(X,Y) dXdY \\ 
&= \int_I  \Lc_{1,w}^n \psi(X,Y) dXdY \numberthis \label{exp:mgf} .
\end{align*}
Then by \eqref{sp:decomp}, $\Lc_{1,w}^n$ splits as $\lambda_{1,w}^n \Pc_{1,w}+\Nc_{1,w}^n$ and \eqref{exp:mgf} becomes 
\begin{align*} 
\Eb[\exp(w C_n)] &= \left( \lambda_{1,w}^n \int_I  \Pc_{1,w} \psi(X,Y) dXdY \right) (1+O(\theta^n))  \numberthis \label{exp:mgf2} .
\end{align*}
where the error term is uniform with $\theta<1$ satisfying $r(\Nc_{1,w}) \leq \theta |\lambda_{1,w}|$.

Hence by applying Theorem \ref{thm:hwang}, we conclude the following limit Gaussian distribution result for the complex Gauss system $(I,T)$.

\begin{thm} \label{clt:cont}
Let $c$ be the digit cost with moderate growth assumption that is not of the form $g - g \circ T$ for some $g \in C^1(\Pc)$. Then there exist positive constants $\widehat{\mu}(c)$ and $\widehat{\delta}(c)$ such that for any $n \geq 1$ and $u \in \R$, 
\begin{enumerate}
\item the distribution of $C_n$ is asymptotically Gaussian,
\[\Pbb \left[\frac{C_n - \widehat{\mu}(c) n}{\widehat{\delta}(c) \sqrt{n}}\leq u \right]=\frac{1}{\sqrt{2\pi}}\int_{-\infty}^u e^{-\frac{t^2}{2}}dt+O\left(\frac{1}{\sqrt n}\right)  . \]

\item the expectation and variance satisfy
\begin{align*}
\Eb[C_n] &= \widehat{\mu}(c) n+\widehat{\mu}_1(c)+O(\theta^n) \\
\Vb[C_n] &= \widehat{\delta}(c) n+\widehat{\delta}_1(c)+O(\theta^n)
\end{align*}
for some constants $\widehat{\mu}_1(c)$ and $\widehat{\delta}_1(c)$, where $\theta<1$ is as given in \eqref{exp:mgf2}.
\end{enumerate}
\end{thm}

\begin{proof}
From the expression \eqref{exp:mgf2}, the function $U$ is given by $w \mapsto \log \lambda_{1,w}$ and $V$ is given by $w \mapsto \log (\int_I \Pc_{1,w} \psi)$ with $\beta_n=n$ and $\kappa_n=\theta^{-n}$. 
Take $\widehat \mu(c) = U'(0)$, $\widehat \delta(c) = U''(0)$, $\widehat \mu_1(c)=V'(0)$, and $\widehat \delta_1(c) = V''(0)$.
We have $U''(0)\neq 0$ by Lemma \ref{comp:pressure}, in turn conclude the proof by Theorem \ref{thm:hwang}.\end{proof}

\section{Gaussian \rom{2}} \label{sec:clt2}

In this section, we obtain the central limit theorem for $K$-rational trajectories of $(I,T)$. 

Let us first introduce a height function. 
Any $z \in K^\times$ can be written in the reduced form as $z= \al / \beta$ with relatively prime $\al, \beta \in \Oc$. Define $\mathrm{ht} \colon K \to \mathbb Z_{\ge 0}$ by 
\begin{equation} \label{def:ht}
\mathrm{ht}: z \longmapsto  \max  \{|\al|,|\beta|\},    
\end{equation}
where $| \cdot |$ denotes the usual absolute value on $\C$. The height is well-defined since $\Oc^\times$ consists of roots of unity. By convention, write $\mathrm{ht}(0)=0$.

Let $N \geq 1$ be a positive integer. Set 
\[ \Sigma_N:= \{ z \in I \cap K: \mathrm{ht}(z)^2 =N      \}              \]
and 
\[ \Omega_N:=\cup_{n \leq N} \Sigma_n= \{ z \in I \cap K: \mathrm{ht}(z)^2 \leq N  \}  .  \]

Recall that the total cost is defined by \[ C(z)=\sum_{j=1}^{\ell(z)} c(\al_j) \] for $z=[0;\al_1, \ldots, \al_{\ell(z)}] \in I \cap K$. From now on, we impose a technical assumption that $c$ is bounded. See Remark \ref{rmk:growth}. 

Now $C$ can be viewed as a random variable on $\Sigma_N$ and $\Omega_N$ with the uniform probability $\Pbb_N$. Studying its distribution on the set $\Sigma_N$ of $K$-rational points with the fixed height is extremely difficult in general, and there is no single result as far as the literature shows. Instead, we observe the asymptotic Gaussian distribution of $C$ on the averaging space $\Omega_N$ by adapting the established framework (cf. Baladi--Vallée \cite{bal:val}, Lee--Sun \cite{lee:sun}, Bettin--Drappeau \cite{bet:dra}), along with spectral properties settled in \S\ref{sec:acim}-\ref{sec:dolgopyat} as follows.

\subsection{Resolvent as a Dirichlet series}\label{subsec:resolvent}
In this subsection, we recall our previous results to express the resolvent as Dirichlet series.
We also establish its analytic properties, which will be used later. 

The proofs parallel \cite{bal:val}, differing only in certain absolute constants; we include them for completeness.

Let $ \mathbf{1} \in C^1(\mathcal P)$ be the characteristic function on $I$. We obtain an expression for $\mathcal L_{s,w}^n \mathbf 1(0)$ as a Dirichlet series.  

Let $O \in \mathcal P[0]$ be the zero-dimensional cell consisting of the origin. Then
\begin{align} \label{ev:at0}
\mathcal L_{s,w}^n \mathbf 1(0)
= \sum_{Q \in \mathcal P}  \sum_{\langle \ba \rangle \in \mathcal H^n(O,Q)} 
\exp\left(wc(\ba)\right)
|J_\ba(0)|^{s}.
\end{align}
To proceed, we make the following observation.
\begin{lem}
Let $\langle \ba \rangle \in \mathcal H^n(O,Q)$. If $z=h_\ba(0)$, then $|J_\ba(0)|=\mathrm{ht}(z)^{-4}$. 
\end{lem}
\begin{proof}
Recall that $h_\ba$ corresponds to $\twobytwotiny{A}{B}{C}{D}=\twobytwotiny{0}{1}{1}{\al_1} \cdots \twobytwotiny{0}{1}{1}{\al_n} \in \GL_2(\Oc)$. Then a simple calculation shows $|J_\ba(0)|=|h_\ba'(0)|^2=|D|^{-4}=\mathrm{ht}(z)^{-4}$. 
\end{proof}

Set $\Omega_N^{(n)}= \left\{ z \in \Omega_N \colon T^n(z)=0\right\}$, i.e. elements whose length of continued fraction expansion is given by $n$.  
Then \eqref{ev:at0} becomes
\begin{align*}
\mathcal L_{s,w}^n \mathbf 1 (0)
=
\lim_{N\to \infty}
\sum_{z \in \Omega_N^{(n)}} \exp\left(w C(z)\right){\mathrm{ht}(z)}^{-4s}.
\end{align*}
Summing over $n$, we obtain
\begin{align*}
\sum_{n=0}^\infty \mathcal L_{s,w}^n \mathbf 1(0)
=\lim_{N\to \infty} \sum_{z \in \Omega_N} \exp\left(w C(z)\right){\mathrm{ht}(z)}^{-4s}.
\end{align*}

Recall that $\Omega_N=\bigcup_{n \le N} \Sigma_n$. By putting
$$
d_n(w) = \sum_{z \in \Sigma_n} \exp(w C(z)) ,
$$
we have the expression for the resolvent of the operator as a Dirichlet series 
\begin{align} \label{ex:dirichlet}
L(2s,w):=  \sum_{n=1}^\infty \frac{d_n(w)}{n^{2s}}= (I-\Lc_{s,w})^{-1} \mathbf 1 (0).
\end{align}

In the next proposition, we deduce the crucial analytic properties of Dirichlet series as a direct consequence of spectral properties of $\Lc_{s,w}$.
Recall from Lemma~\ref{comp:pressure} that there is an analytic map $s_0:W \ra \C$ such that for all $w \in W$, we have $\lambda_{s_0(w),w}=1.$ 
Recall that $t$ denotes the imaginary part of $s$.
\begin{prop} \label{dirichlet:pr}
For any $0<\xi<\frac{1}{10}$, we can find $0<\al_0, \al_1 \leq \frac{1}{2}$ with the following properties: 

For any $\widehat \al_0$ with $0<\widehat \al_0<\al_0$ and $w \in W$,
\begin{enumerate}
\item $\Re s_0(w)>1-(\al_0-\widehat \al_0)$.
\item $L(2s,w)$ has a unique simple pole at $s=s_0(w)$ in the strip $|\Re s-1 | \leq \al_0,$.
\item $|L(2s,w)| \ll |t|^\xi$ for sufficiently large $|t|$ in the strip $|\Re s-1 | \leq \al_0$.
\item $| L(2s,w)| \ll \max (1,|t|^\xi)$ on the vertical line $\Re s=1 \pm \al_0$.
\end{enumerate}
Furthermore, for all $\tau \in \R$ with $0<|\tau|<\pi$,
\begin{enumerate}
\item[(5)] $L(2s,i \tau)$ is analytic in the strip $|\Re s -1 | \leq \al_1$.
\item[(6)] $|L(2s,i \tau)| \ll |t|^\xi$ for sufficiently large $|t|$ in the strip $|\Re s-1 | \leq \al_1$.
\item[(7)] $| L(2s,i \tau)| \ll \max (1,|t|^\xi)$ on the vertical line $\Re s=1 \pm \al_1$.
\end{enumerate}
\end{prop}

\hide{------
\begin{prop} \label{dirichlet:pr2}
For some $\xi>0$, we can find $0<\al_1 \leq \frac{1}{2}$ such that for all $\tau \in \R$ with $0<|\tau|<\pi$, we have
\begin{enumerate}
\item $L(2s,i \tau)$ is analytic. 
\item $| L(2s,i \tau)| \ll |t|^\xi$ for sufficiently large $|t|$.
\end{enumerate}
in the vertical strip $|\Re(s)-1 | \leq \al_1$.
\end{prop}
-------}

\begin{proof}

This is an immediate consequence of Theorem \ref{thm:ruelle} and \eqref{6.1} of Theorem~\ref{main:dolgopyat}, through the identity \eqref{ex:dirichlet} as in Baladi--Vallée \cite[Lemma 8,9]{bal:val}. Each vertical line $\Re(s)=\sigma$ splits into three parts: Near the real axis, spectral gap for $(s,w)$ close to $(1,0)$ gives (1), the location of simple pole at $s=s_0(w)$. 
For the domain with $|t| \geq 1/\rho^2$, Dolgopyat estimate yields the uniform bound.

To finish, it remains to argue (3) that there are no other poles in the compact region $|t|<1/\rho^2$, which comes from the fact that $1 \not\in \mathrm{Sp}(\Lc_{1+it, i\tau})$ if $(t,\tau) \neq (0,0)$. This is shown following the lines in Baladi--Vallée \cite[Lemma 7]{bal:val}. 
\end{proof}

\subsection{Quasi-power estimate: applying Tauberian theorem}

In this subsection, we carry out the Tauberian argument, following \cite{bal:val}, including details to pin down a few absolute constants depending on 
$(I,T).$

We remark that the coefficients $d_n(w)$ of the Dirichlet series $L(2s,w)$ in \eqref{ex:dirichlet} determines the moment generating function of $C$ on $\Omega_N$. 
That is, we have 
\[ \Eb_N[\exp(w C)| \Omega_N]= \frac{1}{|\Omega_N|} \sum_{n \leq N}  d_n (w).    \]

Thus, we obtain the explicit estimate of the moment generating function by studying the average of the coefficients $d_n(w)$. This can be done by applying a Tauberian argument. We will use the following version of truncated Perron's formula (cf. Titchmarsh \cite[Lemma 3.19]{titch}, Lee--Sun \cite[\S3]{lee:sun0}).

\begin{thm}[Perron's Formula] \label{perron}
Suppose that $a_n$ is a sequence and $A(x)$ is a non-decreasing function such that $|a_n|=O(A(n))$.
Let $F(s)=\sum_{n \geq 1} \frac{a_n}{n^s}$ for $\Re s:=\sigma>\sigma_a$, the abscissa of absolute convergence of $F(s)$. Then for all $D>\sigma_a$ and $T>0$, one has
\begin{align*}
\sum_{n \leq x} a_n= \frac{1}{2\pi i} \int_{D-iT}^{D+iT} F(s) \frac{x^s}{s}ds &+ O\left(\frac{x^D |F|(D)}{T} \right) +O \left(\frac{A(2x)x\log x}{T} \right) \\ &+ O \left( A(x) \mathrm{min} \left\{ \frac{x}{T|x-M |},1 \right\} \right)
\end{align*}
as $T$ tends to infinity, where \[ |F|(\sigma):=\sum_{n \geq 1} \frac{|a_n|}{n^\sigma} \] for $\sigma > \sigma_a$ and $M$ is the nearest integer to $x$.

\end{thm}

Proposition \ref{dirichlet:pr} enables us to obtain a Quasi-power estimate of $\Eb_N[\exp(w C)| \Omega_N]$ by applying Theorem \ref{perron} to $L(2s,w)$. We first check the conditions of Perron's formula.

\begin{lem} \label{cond:complexity}
For $z \in \Omega_N$, we have $\ell(z)=O(\log N)$.
\end{lem}

\begin{proof}
Recall that there is $R<1$ such that for all $z \in I$ we have $|z|\le R$. Explicitly, we may take $R = \sqrt{15/16}$. 

Let $z \in \Omega_N$. Write $z$ in the form $z = u/v$ with $u,v \in \mathcal O$, which we assume to be relatively prime. Write $T(u/v) = u_1/v_1$ with relatively prime $u_1,v_1 \in \mathcal O$. We claim that $|v_1| \le R|v|$. Indeed, by the definition of $T$,
$T(u/v)= v/u - [v/u]$. Put $\alpha = [v/u]$.
Then, $T(u/v) = u_1/v_1$ with $v_1=u$ and $u_1=v-\alpha u$.
This proves the claim. 

Inductively, if we put $T(u_j/v_j) = u_{j+1}/v_{j+1}$, then we have $|v_{j+1}|\le R|v_{j}|$ for all $j \geq 1$. This yields the desired bound $\ell(z)=O(\log N)$.
%
\end{proof}

\begin{lem} \label{cond:A(n)}
Suppose $k>0$ satisfies $\ell(z) \leq k \log n$ for all $n$ and $z \in \Omega_n$, and $M>0$ satisfies $c(\alpha)\le M$ for all $\alpha \in \mathcal A$. For any $\varepsilon>0$, we have 
\[ 
|d_n(w)|
\ll n^{1+\varepsilon +kM\Re w} \]
for all sufficiently large $n$. The implied constant only depends on $\varepsilon$.
\end{lem}
\begin{proof}
To begin with, we claim that $|\Sigma_n| \ll n^{1+\varepsilon}$ for any $\varepsilon>0$, where the implied constant depends on $\varepsilon$. To prove the claim, if $z \in \Sigma_n$, we write it as $z = u/v$ for some $u,v \in \mathcal O$ satisfying $|v|^2=n$ and $|u|^2<n$ and we will enumerate $u$ and $v$ separately. 

We first count the number of $v$'s satisfying $|v|^2=n$, which we temporarily denote by $a_n$. Using the fact that $\alpha \mapsto |\alpha|^2$ is a quadratic form on $\mathcal O$, one can identify the formal power series $\sum_{n \ge 0} a_nq^n$ with the theta series associated with the quadratic form. By a general theory of theta series, treated in \cite[\textsection\,2.3.4]{cohen} and \cite[\textsection\,3.2]{1-2-3} for example, it is a modular form of weight one. Using a general asymptotic for such forms, given in \cite[Remarks\,9.2.2.\,(c)]{cohen} for example, we conclude that $a_n = O(\sigma_0(n))$ where $\sigma_0(n)$ denotes the number of positive divisors of $n$. A well-known bound \cite[\textsection\,13.10]{Apostol} is $\sigma_0(n)=o(n^\varepsilon)$ for any $\varepsilon>0$. 

Now we turn to $v$. Since the condition $|v|^2 <n $ cuts out the lattice points in a disc of area $2 \pi n$, the number of $v$'s is $O(n)$. Adding up, we obtain $|\Sigma_n| \ll n^{1 + \varepsilon}$. To proceed, notice that the assumptions imply $C(z) \le k M \log n$. Combining it with the earlier bound for $|\Sigma_n|$ to conclude $|d_n(w)|\ll n^{1+\varepsilon+k M \Re w}$.
\end{proof}



Together with a suitable choice of $T$, we obtain:

\begin{prop} \label{mgf:normal}
For a non-vanishing $D(w)$ and $\gamma>0$, we have 
\[ \sum_{n \leq N} d_n(w) =D(w) N^{ 2 s_0(w)}(1+O(N^{-\gamma})).\]
\end{prop}

\begin{proof}
Recall that Proposition \ref{dirichlet:pr} (2) allows us to apply Cauchy's residue theorem to obtain: 
\[ \frac{1}{2\pi i} \int_{\mathcal{U}_T(w)} L(2s,w) \frac{N^{2s}}{2s} d(2s)=\frac{E(w)}{ s_0(w)} N^{2s_0(w)}.\]
Here, $E(w)$ is the residue of $L(2s,w)$ at the simple pole $s=s_0(w)$ and $\mathcal{U}_T(w)$ is the contour with the positive orientation, which is a rectangle with the vertices $1+\al_0+iT$, $1-\al_0+iT$, $1-\al_0-iT$, and $1+\al_0-iT$. Together with Perron's formula in Theorem \ref{perron}, we have  
\begin{align*}
\sum_{n \leq N} d_n(w)&= \frac{E(w)}{s_0(w)} N^{2 s_0(w)}+ O\left(\frac{N^{2(1+\al_0)}}{T} \right) +O \left(\frac{A(2N)N\log N}{T} \right) + O(A(N)) \\ &+ O \left( \int_{1-\al_0-iT}^{1-\al_0+iT} |L(2s,w)| \frac{N^{2(1-\al_0)}}{|s|} ds    \right) \\
&+ O \left( \int_{1-\al_0 \pm iT}^{1+\al_0 \pm iT} |L(2s,w)| \frac{N^{2\Re s}}{T} ds    \right) .
\end{align*} 

Note that the last two error terms are from the contour integral, each of which corresponds to the left vertical line and horizontal lines of the rectangle $\mathcal{U}_T(w)$. Let us write the right hand side of the last expression as 
\[ \sum_{n \leq N} d_n(w)= \frac{E(w)}{s_0(w)} N^{2 s_0(w)} \left(1+ \mathrm{\rom{1}+ \rom{2}+\rom{3}+\rom{4}+\rom{5}}\right). \]

By Proposition~\ref{dirichlet:pr}, we have $0 < \alpha_0 \le \frac 1 2$. 
Choose $\widehat{\al}_0$ with \[  \frac{11}{28}\al_0 < \widehat{\al}_0 < \al_0 \] and set \[ T= N^{2\al_0+4 \widehat{\al}_0}.\] 
Notice that $\frac{E(w)}{s_0(w)}$ is bounded in the neighbourhood $W$ since $s_0(0)=1$. Note also from Proposition \ref{dirichlet:pr} that $\Re s_0(w)>1-(\al_0-\widehat \al_0)$. Below, we explain how to obtain upper bounds for the error terms in order.

($\mathrm{\rom{1}}$) The error term $\mathrm{\rom{1}}$ is equal to $O(N^{2(1-2\widehat{\al}_0-\Re s_0(w))})$. Observe that the exponent satisfies \[ 2(1-2\widehat{\al}_0-\Re s_0(w))<2(\al_0-3\widehat \al_0)<0.\]

($\mathrm{\rom{2}}$) 
By Lemma \ref{cond:A(n)}, for any $\varepsilon$ with $0<\varepsilon<\frac{\widehat{\al}_0}{4}$, we can take $W$ from Lemma \ref{comp:pressure} small enough to have $ k \Re w < \varepsilon$ so that $A(N)=O(N^{1+2\varepsilon})$ and $\log N \ll N^{\varepsilon}$. Then the exponent of $N$ in the error term $\mathrm{\rom{2}}$ is equal to 
\[ 2+3\varepsilon-2( \al_0+  2\widehat{\al}_0 - \Re s_0(w))\leq  - \frac{21}{4} \widehat{\al}_0 < 0. \]

($\mathrm{\rom{3}}$)  Similarly, the error term $\mathrm{\rom{3}}$ is equal to $O(N^{1+2\varepsilon-2\Re s_0(w)})$. The exponent satisfies \begin{align*}  1+ 2\varepsilon -2\Re s_0(w)< - 1+ 2\al_0 -\frac{3}{2} \widehat{\al}_0< -\frac{3}{2} \widehat{\al}_0<0.  \end{align*}
Here, recall that $0< \al_0 \leq \frac{1}{2}$.

($\mathrm{\rom{4}}$)  For $0<\xi<\frac{1}{10}$, we have $|L(2s,w)| \ll |t|^{\xi}$ by Proposition \ref{dirichlet:pr} where $t=\Im s$. The error term $\mathrm{\rom{4}}$ is $O(N^{2(1-\al_0-\Re s_0(w))} T^\xi)$ and the exponent of $N$ is equal to 
\begin{align*} & 2(1-\al_0-\Re s_0(w))+ (2\al_0+4 \widehat{\al}_0)\xi < \frac{1}{5} \al_0- \frac{8}{5} \widehat \al_0<0
\end{align*}

($\mathrm{\rom{5}}$)  The last term $\mathrm{\rom{5}}$ is $O(T^{\xi-1}  N^{2(1+\al_0-\Re s_0(w))} (\log N)^{-1})$. Hence, the exponent satisfies 
\begin{align*} & (2\al_0+4 \widehat{\al}_0)(\xi-1)+2(1+\al_0-\Re s_0(w))
\\ &< \frac{11}{5} \al_0 - \frac{28}{5} \widehat \al_0<0.
\end{align*}
By taking \[ \gamma= \mathrm{max} \left( 2(3\widehat \al_0 -\al_0),\frac{8}{5} \widehat \al_0- \frac{1}{5} \al_0,  \frac{28}{5} \widehat \al_0-\frac{11}{5} \al_0 \right) ,\]
we obtain the theorem. 
\end{proof}

Finally by applying Theorem \ref{thm:hwang}, we conclude the following limit Gaussian distribution for $K$-rational trajectories.

\begin{thm} \label{clt:discrete}
Take $c$ as in Theorem \ref{clt:cont} and further assume that it is bounded. For suitable positive constants $\mu(c)$ and $\delta(c)$, and for any $u \in \R$,
\begin{enumerate}
\item the distribution of $C$ on $\Omega_N$ is asymptotically Gaussian, 
\[\Pbb_N \left[\frac{C - \mu(c) \log N}{\delta(c) \sqrt{\log N}}\leq u  \Big\lvert\, \Omega_N \right]=\frac{1}{\sqrt{2\pi}}\int_{-\infty}^u e^{-\frac{t^2}{2}}dt+O\left(\frac{1}{\sqrt{\log N}}\right) . \]

\item the expectation and variance satisfy
\begin{align*}
\Eb_N[C|\Omega_N] &= \mu(c)\log N+\mu_1(c)+O(N^{-\gamma}) \\
\Vb_N[C|\Omega_N] &= \delta(c) \log N+\delta_1(c)+O(N^{-\gamma})
\end{align*}
for some $\gamma>0$, constants $\mu_1(c)$ and $\delta_1(c)$.
\end{enumerate}
\end{thm}

\begin{proof}
Proposition \ref{mgf:normal} yields that with a suitable $0< \gamma<\al_0$, the moment generating function admits the quasi-power expression, i.e. for $w \in W$
\[ \Eb_N[\exp(w C)|\Omega_N]= \frac{D(w)}{D(0)} N^{2(s_0(w)-s_0(0))}(1+O(N^{-\gamma}))     \]
holds where $D(w)=\frac{E(w)}{s_0(w)}$ from Proposition \ref{mgf:normal} is analytic on $W$. 

Take $U(w)=2(s_0(w)-s_0(0))$, $V(w)=\log \frac{D(w)}{D(0)}$, $\beta_N=\log N$, and $\kappa_N=N^{-\gamma}$. 
We put $\mu(c) = U'(0)$, $\delta(c) = U''(0)$, $\mu_1(c)=V'(0)$, and $\delta_1(c) = V''(0)$.
Observe that we have $s_0'(0)=-\frac{\partial \lambda}{\partial w} (1,0)/ \frac{\partial \lambda}{\partial s} (1,0)$ since $\lambda_{s_0(w),w}=1$ for $w \in W$. Further, the derivatives of the identity $\log \lambda_{s_0(w),w}=0$ yield 
\[ \frac{\partial \lambda}{\partial s}(1,0) s_0''(0)= \frac{d^2}{dw^2} \lambda_{1+s_0'(w)w,w} \bigg|_{w=0}.     \]
Thus by Lemma \ref{comp:pressure}, we have $U''(0)=2 s_0''(0) \neq 0$ if $c$ is not a coboundary.  Applying Theorem \ref{thm:hwang}, we obtain the statement.
\end{proof}

\section{Equidistribution modulo $q$} \label{sec:equid}

In this section, we show that for any integer $q>1$ and a bounded digit cost $c:\Ac \ra \Z_{\geq 0}$, the values of $C$ on $\Omega_N$ are equidistributed modulo $q$. This follows from the following estimate for $\Eb[\exp(i \tau C)|\Omega_N]$ when $|\tau|$ is away from 0. Applying Theorem \ref{perron} to $L(2s,i\tau)$, we have:

\begin{prop} \label{mgf:eq}
Let  $0<|\tau| < \pi$. Then, there exists $0<\delta<2$ such that we have 
$$\sum_{n \leq N}  d_n(i \tau)=	O (N^{\delta}).$$
\end{prop}

\begin{proof}
By Proposition \ref{dirichlet:pr}, $L(2s, i\tau)$ is analytic in the rectangle $\mathcal{U}_T$ with vertices $1+\alpha_1+iT$, $1-\alpha_1+iT$, $1-\alpha_1-iT$, and $1+\alpha_1-iT$. Cauchy's residue theorem yields 
\[ \frac{1}{2 \pi i} \int_{\mathcal{U}_T} L(2s,i \tau) \frac{N^{2s}}{2s} d(2s)=0 \]
and together with Perron's formula in Theorem \ref{perron}, we have 
\begin{align*}
\sum_{n \leq N} d_n(i\tau)&=  O\left(\frac{N^{2(1+\al_1)}}{T} \right) +O \left(\frac{A(2N)N\log N}{T} \right) + O(A(N)) \\ &+ O \left( \int_{1-\al_1-iT}^{1-\al_1+iT} |L(2s, i\tau)| \frac{N^{2(1-\al_1)}}{|s|} ds \right)\\
&+ O \left( \int_{1-\al_1 \pm iT}^{1+\al_1 \pm iT} |L(2s,i\tau)| \frac{N^{2\Re s}}{T} ds  \right) .
\end{align*} 

We briefly denote this by $\sum_{n \leq N} d_n(i\tau)= \mathrm{\rom{1}+ \rom{2}+\rom{3}+\rom{4}+\rom{5}} $. 
Taking \[ T= N^{5\al_1}, \] 
the error terms are estimated as follows. 


($\mathrm{\rom{1}}$) The error term \rom{1} is simply equal to $O(N^{2 -3\alpha_1})$. 

($\mathrm{\rom{2}}$) For any $0<\varepsilon<\frac{\alpha_1}{4}$, we can take $A(N)=O(N^{1+2\varepsilon})$ and $\log N \ll N^{\varepsilon}$. Then the exponent of $N$ in the error term \rom{2} is equal to 
\[ 2+ 3\varepsilon- 5 \al_1< 2- \frac{17}{4} \alpha_1 <2 .\]

($\mathrm{\rom{3}}$) The error term \rom{3} is equal to $O(N^{1+ \alpha_1 /2} )$. 

($\mathrm{\rom{4}}$) For $0<\xi<\frac{1}{10}$, we have  $|L(2s, i\tau)| \ll |t|^\xi $. Thus, the error term \rom{4} is $O (T^\xi  N^{2(1-\alpha_1)} )$ and the exponent of $N$ is equal to 
\[ 2(1-\alpha_1) + 5 \al_1 \xi < 2- \frac{3}{2}\alpha_1 <2 . \]

($\mathrm{\rom{5}}$) The last term \rom{5} is $O (T^{\xi-1} N^{2(1+\alpha_1)} (\log N)^{-1})$, whence the exponent of $N$ satisfies 
\[ 5 \al_1 (\xi-1)+2(1+ \alpha_1) < 2- \frac{5}{2} \alpha_1<2 .\]
By taking \[ \delta = \max \left( 2- 3\alpha_1, 2- \frac{17}{4} \alpha_1, 2- \frac{3}{2}\alpha_1, 2- \frac{5}{2} \al_1 \right) \]
which is strictly less than 2, we complete the proof. 
\end{proof}

Now we present an immediate consequence of Proposition \ref{mgf:eq}:

\begin{thm} \label{thm:equid}
Take $c$ as in Theorem \ref{intro:clt1}. Further assume that $c$ is bounded and takes values in $\Z_{\geq 0}$. For any $a \in \Z/q\Z$, we have
\[ \Pbb_N[C \equiv a \ (\mathrm{ mod } \ q) |\Omega_N]=q^{-1}+o(1) ,   \]
i.e., $C$ is equidistributed modulo $q$.
\end{thm}

\begin{proof}
Observe from Proposition \ref{mgf:normal}, we have $\sum_{n \leq N} d_n(0) \gg N^2$. Then Proposition \ref{mgf:eq} yields that with  $\delta_0:=2-\delta >0$ and $\tau$ under the same condition, we have
\begin{equation} \label{mgfeq}
\Eb_N[\exp(i \tau C)|\Omega_N]=\frac{\sum_{n \leq N} d_n(i \tau)}{\sum_{n \leq N} d_n(0)} \ll O(N^{-\delta_0}) .   
\end{equation}
Then for $a \in \Z/ q\Z$, we have 
\begin{align*}
\Pbb_N[C \equiv a \ (\mathrm{ mod } \ q) |\Omega_N] &=\sum_{m \in \Z \atop m \equiv a (q)} \Pbb_N[C \equiv m |\Omega_N] \\
&=\sum_{m \in \Z} \left( \frac{1}{q} \sum_{k \in \Z / q\Z} \exp \left( \frac{2 \pi i}{q} k(m-a)  \right)  \right)\Pbb_N[C \equiv m |\Omega_N] \\
&=\frac{1}{q} \sum_{k \in \Z / q\Z} e^{-\frac{2 \pi i}{q} k a} \cdot \Eb_N \left[\exp\left(\frac{2\pi i}{q} ka \right) \bigg| \Omega_N \right] .
\end{align*}

We split the summation into two parts: $k=0$ and $k \neq 0$. The term corresponding to $k=0$ is the main term which equals to $q^{-1}$. For the sum over $k \neq 0$, taking $0<\tau<q^{-1}$ in \eqref{mgfeq}, we obtain the result.
\end{proof}

\hide{--------

\section{Old calculations}

In this section, we study the operator $\mathcal L_{s,w}$ when $(s,w)=(1,0)$. 
We put $\mathbb H := \mathcal L_{1,0}$.
Let $\mathcal B_i  = \bigoplus_{P \in \mathcal P[i]} C^1(\overline P)$. 
Then we have a decomposition
\begin{align}
    \mathcal B = \mathcal B_0 \oplus \mathcal B_1 \oplus \mathcal B_2
\end{align}
and the operator $\mathbb H$ can be written as
\begin{align}
    \mathbb H = \sum_{0 \le j \le i \le 2} \mathbb H^i_j
\end{align}
with $\mathbb H^i_j \colon \mathcal B_i \to \mathcal B_j$. 

\subsection{Spectral gap for $\mathbb H^2_2$}
In this subsection, we show that the operator $\mathbb H^2_2$ acting on $\mathcal B_2$ has a spectral gap.
During this subsection, we regard $\mathcal B_2$ as a Banach space with respect to the supremum norm.

We first prove the compactness.
\begin{prop}\label{prop:cpt}
The operator $\mathbb H^2_2$ acts compactly on $\mathcal B_2$.
\end{prop}
\begin{proof}
Note that $\mathbb H^2_2 \colon \mathcal B_2 \to \mathcal B_2$ can be written as
\begin{align*}
    \mathbb H^2_2 = \sum_{P,Q \in \mathcal P[2]} \mathbb H^Q_P
\end{align*}
with $\mathbb H^Q_P \colon C^1(\overline Q) \to C^1(\overline P)$. 
To show $\mathbb H^2_2$ is compact, it suffices to show that each $\mathbb H^Q_P$ is compact.
Now we show that $\mathbb H^Q_P$ is compact.
By definition, we have
\begin{align*}
    \mathbb H^Q_P(f) = \sum_{ \langle\alpha\rangle \in \mathcal H(P,Q)} |J_\alpha| \cdot \left( f \circ \langle\alpha\rangle \right) (x)
\end{align*}
where each term is compact since $|J_\alpha|$ is bounded from above.
Moreover, we have
\begin{align*}
    \sum_{ \langle\alpha\rangle \in \mathcal H(P,Q)} |J_\alpha(z)| \le \sum_{\substack{\alpha \in \mathcal O \\ |\alpha| \ge \sqrt 2}} |z+\alpha|^{-4} < \infty
\end{align*}
for all $z \in \overline P$.
By the Arzela-Ascoli theorem, it follows that each $\mathbb H^Q_P$ is compact.
\end{proof}

We continue to establish the spectral gap.
\begin{thm}
The operator $\mathbb H^2_2$ acting on $\mathcal B_2$ has a spectral gap. Its dominant eigenvalue is one. 
\end{thm}
\begin{proof}
Let $\mathcal B_2^+ \subset \mathcal B_2$ be the subspace consisting of tuples $(f_P)_{P\in \mathcal P[2]}$ such that each $f_P$ takes a non-negative value everywhere. 
Write $f \ge g $ if $f-g \in \mathcal B_2^+$. Let $v = (v_P)_P \in \mathcal B_2^+$ be the tuple with $v_P \equiv 1$ for all~$P$. 
To apply the criterion \cite{Krasnoselskii}, it suffices to show that the quadruple $(\mathcal B_2, \mathcal B_2^+, \mathbb H_2^2, v)$ satisfies the properties
\begin{enumerate}
    \item $\mathcal B_2^+$ is closed under addition and scaling by positive numbers,
    \item $\mathcal B_2^+$ is a closed subset of $\mathfrak B$ whose interior is non-empty,
    \item every $f \in \mathcal B_2^+$ can be written as $f = p_1-p_2$ with $p_1,p_2 \in \mathcal B_2^+$,
    \item $\mathbb H^2_2(\mathcal B_2^+) \subset \mathcal B_2^+$, and
    \item if $f \in \mathcal B_2^+$ is not zero, then there exist some positive integer $n$ and positive real numbers $c_1$ and $c_2$, such that $c_1 v \le \left(\mathbb H_2^2\right)^n \left(f\right) \le c_2 v$.
\end{enumerate}
Except for the last one, they are easy to verify. The first property follows directly from the definition. The second follows from observing that the interior of $\mathcal B_2^+$ contains tuples of functions with positive infimums. To verify the third property, note that any continuous function $g$ on a compact set can be rewritten as $g = \left(g+2M\right) - 2M$ where $M$ is the maximum of $g$. The fourth property follows from the positivity of $|J_\alpha|$. 

To show the last property, we need a lemma.
\begin{lem}
Let $P \in \mathcal P[2]$ and $J$ be an open subset of $P$. Then, there is some $n \ge 1$ such that $T^n(J)$ is dense in $I$.
\end{lem}
\begin{proof}[proof of lemma]
For each $m \ge 1$, the union of all $\mathcal O_{\ba}$'s with $|\ba | =m$ is dense in $I$.
Moreover, there is a sequence $\rho_m$ of positive reals with $\rho_m \to 0$ as $m \to \infty$ such that the diameter of $\mathcal O_{\ba}$ with $|\ba |=m$ is at most $\rho_m$.
Therefore, such $J$ contains $\mathcal O_{\ba}$ for some $m$ and $\ba$ with $|\ba | = m$. 
On the other hand, there is some $k \ge 1$ such that $T^k\mathcal O_\alpha$ is dense in $I$ for any $\alpha \in \mathcal O$ with nonempty $\mathcal O_\alpha$. 
Then, $T^{m+k}(\mathcal O_\alpha)$ is dense in $I$. 
\end{proof}

We proceed to verify the last condition.
Let $f \in \mathcal B_2^+$ is a non-zero element.
Then, $A_\epsilon = \{z \in I \colon f(z) >\epsilon\}$ is a non-empty open subset for some sufficiently small $\epsilon$.
Then, there is some $P \in \mathcal P[2]$ such that $P \cap A_\epsilon$ is non-empty. 
Take $J = P \cap A_f$.
By the lemma, $T^n(J)$ is dense in $I$ for some $n \ge 1$. Then, $c_1 v \le \left(\mathbb H^2_2\right)^n(f)$ for some $c_1 \ge 0$. 
Since $\left(\mathbb H^2_2\right)^n(f) \in \mathcal B_2$, it is clear that there is some $c_2 >0$ with $\left(\mathbb H^2_2\right)^n(f) \le c_2 v$.
\end{proof}

\subsection{Spectral radius of $\mathbb H^1_1$}
Let $r_1$ be the spectral radius of $\mathbb H^1_1 \colon \mathcal B_1 \to \mathcal B_1$. 
The aim of this subsection is to prove
\begin{align}\label{eq:radius1}
    r_1 <1.
\end{align}

\begin{prop}
The operator $\mathbb H^1_1$ acts compactly on $\mathcal B_1$.
\end{prop}
\begin{proof}
It is similar to the proof of Prop.\,\ref{prop:cpt}.
\end{proof}
Since any non-zero eigenvalue of a compact operator is an isolated point in the spectrum, we may assume that there is an eigenvalue $\lambda$ of $\mathbb H^1_1$ such that $|\lambda|=r_1$.
In particular, it suffices to show
\begin{align}\label{eq:radius2}
    |\lambda| <1.
\end{align}

To prove \eqref{eq:radius2}, we will use the following norm on $\mathcal B_1$.
Recall that each $f \in \mathcal B_1$ is a tuple $f= (f_P)_P$ labelled by $P \in \mathcal P[1]$.
Put
\begin{align*}
    \left|\!\left|f\right|\!\right|_1 = 
    \sum_{P \in \mathcal P[1]}
    \int_{z \in \overline P} \left|f_P(z)\right| ds_P 
\end{align*}
where $ds_P$ is the length element of the curve $\overline P$. 
Let $s_1=\left|\!\left|\mathbb H^1_1 \right| \! \right|_1$ the corresponding operator norm. 
By definition,
\begin{align*}
    s_1 = \sup_{f} \left|\!\left| \mathbb H^1_1 (f)\right|\!\right|_1
\end{align*}
where the supremum is taken over the set of all $f\in \mathcal B_1$ with $\left|\!\left|f\right|\!\right|_1=1$.
Such $f$ can be taken in the eigenspace of $\lambda$, which implies that
\begin{align*}
    |\lambda| \le s_1.
\end{align*}
Since $\lambda$ satisfied $|\lambda|=r_1$, it follows that $r_1 \le s_1$. 
In particular, \eqref{eq:radius1} is reduced to 
\begin{align}\label{eq:radius3}
    s_1 <1.
\end{align}

Put
\begin{align}
    R := \sup_{z \in I} |z|.
\end{align}
Note that $R<1$ in all of eight systems $(I,T)$.
The inequality \eqref{eq:radius3} would follow from the next proposition.
\begin{prop}\label{prop:operatornorm:h1}
Let $f \in \mathcal B_1$ be an element with $\left|\!\left|f\right|\!\right|_1=1$. Then, $\left|\!\left| \mathbb H^1_1f\right|\!\right|_1\le R^2$.
\end{prop}
\begin{proof}
By definition of $\mathbb H^1_1$, for an element $f \in \mathcal B_1$ with $\left|\!\left|f\right|\!\right|_1=1$, we have
\begin{align*}
    \left|\!\left| 
    \mathbb H^1_1f
    \right|\!\right|_1
    = \sum_{P \in \mathcal P[1]} 
    \int_{z \in \overline P}
    \left|
    \sum_{Q \in \mathcal P[1]}\sum_{\langle \alpha \rangle^P_Q} \left|z+\alpha\right|^{-4} \cdot f_Q \circ \langle\alpha\rangle^P_Q (z)
    \right|
    ds_P
\end{align*}
and it follows that
\begin{align*}
    \left|\!\left| 
    \mathbb H^1_1f
    \right|\!\right|_1
    \le
    \sum_{Q \in \mathcal P[1]}
    \left(
    \sum_{P \in \mathcal P[1]} 
    \sum_{\langle \alpha \rangle^P_Q} 
    \int_{z \in \overline P}
    \left|z+\alpha\right|^{-4} \cdot
        \left|
    f_Q \circ \langle\alpha\rangle^P_Q (z)
    \right|
    ds_P
    \right).
\end{align*}
The double sum in the parentheses satisfies the inequality:
\begin{align}\label{eq:changeofvariables}
    \sum_{P \in \mathcal P[1]} 
    \sum_{\langle \alpha \rangle^P_Q} 
    \int_{z \in \overline P}
    \left|z+\alpha\right|^{-4} \cdot
        \left|
    f_Q \circ \langle\alpha\rangle^P_Q (z)
    \right|
    ds_P
    \le 
    \int_{z \in \overline Q} \left|z\right|^2 \left|f_Q(z)\right| ds_Q
\end{align}
because for each $P$ and $\langle \alpha \rangle ^P_Q \in \mathcal H(P,Q)$ we have the change of variables formula 
\begin{align*}
\int_{z \in \overline P}
    \left|z+\alpha\right|^{-4} \cdot
        \left|
    f_Q \circ \langle\alpha\rangle^P_Q (z)
    \right|
    ds_P
=\int_{z \in h_\alpha(\overline P)} |z|^2 \left| f_Q(z) \right| ds_Q
\end{align*}
and $h_{\alpha_1}(P_1)\cap h_{\alpha_2}(P_2)$ is nonempty if and only if $(P_1,\alpha_1)=(P_2,\alpha_2)$.

Taking the sum of \eqref{eq:changeofvariables} over $Q \in \mathcal P[1]$ and using $|z| \le R$ for $z \in Q$, we conclude, as desired, $\left|\!\left| \mathbb H^1_1f\right|\!\right|_1\le R^2$. 
\end{proof}

\subsection{Spectral radius of $\mathbb H^0_0$}
Let $r_0$ be the spectral radius of $\mathbb H^0_0 \colon \mathcal B_0 \to \mathcal B_0$. In this subsection, we will prove
\begin{align}\label{eq:r0}
    r_0 <1.
\end{align}
The argument which we give below is similar to our proof of \eqref{eq:radius1}, but the details are simpler because $\mathcal B_0$ is finite-dimensional.

If $P \in \mathcal P[0]$ is a $0$-cell, denote its unique element by $z_P$. 
Each $f \in \mathcal B_0$ is a tuple $f=(f_P)_{P \in \mathcal P[0]}$ and we put

\begin{align*}
\left|\!\left|f\right|\!\right|_1 := 
            \sum_{P \in \mathcal P[0]}  \left|f_P(z_P)\right|.
\end{align*}
For $f \in \mathcal B_0$ and $P \in \mathcal P[0]$, the definition of $\mathbb H^0_0$ implies
\begin{align*}
    \left|\!\left|
    \mathbb H^0_0 f
    \right|\!\right|_1
    =
    \sum_{P \in \mathcal P[0]} 
    \left|
    \sum_{Q \in \mathcal P[0]}\sum_{\langle \alpha \rangle ^P_Q \in \mathcal H(P,Q)} \left|z_P+\alpha\right|^{-4}f_Q(z_Q)
    \right|
    \\
    \le
    \sum_{Q \in \mathcal P[0]}
    \left(
        \sum_{P \in \mathcal P[0]} 
        \sum_{\langle \alpha \rangle ^P_Q \in \mathcal H(P,Q)} \left|z_P+\alpha\right|^{-4}\left|f_Q(z_Q)\right|
    \right)
\end{align*}
in which the double sum in the parentheses is empty unless $T(z_Q)=z_P$ for some $P \in \mathcal P[0]$. 
Observe that such $P$ is unique if exists.
Since $|z_P+\alpha|^{-4} \le R^4$ whenever $\langle \alpha \rangle ^P_Q \in \mathcal H(P,Q)$, we have
\begin{align}
    \left|\!\left|
    \mathbb H^0_0 f
    \right|\!\right|_1
    \le
    R^4
    \left|\!\left|
    f
    \right|\!\right|_1,
\end{align}
which implies that the operator norm 
$\left|\!\left|
    \mathbb H^0_0
    \right|\!\right|_1
$
is at most $R^4$.
It implies that $r_0 \le R^4$.
Because $R<1$, it proves \eqref{eq:r0}.
\subsection{Spectral gap for $\mathbb H$}

For $f_P \in C^1( \overline P)$, we say $f_P>0$ if $f(z)>0$ for all $z \in \overline P$. For $\psi_i = (f_P)_{P \in \mathcal P[i]} \in \mathcal B$, we say $\psi_i >0$ if $f_P>0$ for all $P \in \mathcal P[i]$.

Recall that we have $\mathcal P$ is a cell structure on $I$ compatible with $T$. Let $\mathcal B = C^1(\mathcal P)$ be the space of functions defined in \textsection\,\ref{section:functionspace}.

\begin{prop}
Suppose that $\lambda_{1,0}=1$ is a simple eigenvalue of $\mathcal L_{1,0}$ and that any other eigenvalue $\lambda$ of $\mathcal L_{1,0}$ satisfies $|\lambda| < 1$ and that there is an eigenfunction $\psi_{1,0} = \left(\psi_{1,0,0},\psi_{1,0,1},\psi_{1,0,2}\right)$ with $\psi_{1,0,i}>0$ for $i=0,1,2$. 
\end{prop}

\begin{proof}

\end{proof}

Put $\Lc = \mathcal L_{1,0}$ and for $i \le j,$ let
\begin{align}
    \mathbb H_{i,j} \colon \mathcal B_j \to \mathcal B_i
\end{align}

\par
Let $\psi = \psi_0 + \psi_1 + \psi_2$ be a non-zero eigenvector for $\mathbb H$ with eigenvalue $\lambda$. \tcb{Positivity of each $\psi_i$ for $i=0,1$.}
\par
\begin{prop}\label{prop:PF}
The spectrum of $\mathbb H^*$ contains $\lambda$ and its eigenspace is one-dimensional. 
\end{prop}
\begin{proof}
\end{proof}

Let $\mu_L$ be a measure on $2$-cells obtained by restricting Lebesgue measure on $\mathbb C$. 

\begin{prop}\label{prop:density}
Define $\mu = (0,0,\psi_2 \mu_L)$. It satisfies $\mathbb H^* \mu = \lambda \mu$. 
\end{prop}
\begin{proof}
The eigenfunction $\psi$ satisfies, by definition, $\mathbb H\psi = \lambda \psi$. Since $\mathbb H$ decomposes as $\mathbb H = \sum_{0 \le i \le j \le 2} \mathbb H_{i,j}$, $\psi_2$ satisfies $\mathbb H_{2,2} \psi_2 = \lambda\psi_2$. Then,  the decomposition $\mathbb H^* = \sum_{0 \le i \le j \le 2} \mathbb H_{i,j}^*$ combined with $\mathbb H_{2,2}^*\mu_L=\mu_L$ implies that $\mathbb H^*\mu = \mathbb H^*_{2,2}\mu$.
\end{proof}
We conclude from Proposition~\ref{prop:PF} and Proposition~\ref{prop:density} that the $\psi_2\mu_L$ generates the $\lambda$-eigenspace of $\mathbb H^*$. 

\subsection{Ruelle-Perron-Frobenius}
The aim of this subsection is to establish the following spectral gap and other basic spectral properties of $\mathcal L_{\sigma,u}$ for real parameters $\sigma, u.$ 

\begin{thm}[Ruelle-Perron-Frobenius theorem for $\mathcal L_{\sigma,u}$]
For all $(\sigma,u)$ in an open neighborhood of $(1,0) \in \mathbb R^2$, the operator $\mathcal L_{\sigma,u}$ has a unique eigenvalue $\lambda_{\sigma,u}$ with the following properties.
\begin{enumerate}
	\item The eigenvalue $\lambda_{\sigma,u}$ is simple.
	\item If $\lambda$ is eigenvalue of $\mathcal L_{\sigma,u}$ other than $\lambda_{\sigma,u}$, then $|\lambda_{\sigma,u}|>|\lambda|$.
	\item One can choose an eigenfunction $\psi_{\sigma,u}=(\psi_{\sigma,u,0},\psi_{\sigma,u,1},\psi_{\sigma,u,2})$ for the eigenvalue $\lambda_{\sigma,u}$ so that $\psi_{\sigma,u,i}>0$ for all $i=0,1,2$.
	\item When $(\sigma,u)=(1,0)$, $\lambda_{1,0}=1$.
\end{enumerate}
\end{thm}

Now our goal is to show that $$ \| \mathcal L_{s, w} \|_{(t)} < |t|^{-\gamma}.$$

Let us first show $\int | \widetilde{\mathcal{L}}_{s,w} f |^2 d\mu \ll \frac{ \| f \|_{(t)} }{|t|^\beta} $ and then show that it implies our goal.

Note that
$$\int | \widetilde{\mathcal{L}}_{s,w} f |^2 d\mu = \sum_P \int  | \widetilde{\mathcal{L}}_{s,w} f |^2 d\mu_P \leq \left(\sum_P \int  | \widetilde{\mathcal{L}}_{s,w} f | d\mu_P\right)^2.$$

For each $P \in \mathcal P [1],$ and $i=1,2$, let$$I_i := \sum_{Q \in \mathcal P [i]} \int \sum_{\alpha \in \mathcal H(P,Q)} (f \circ h_\al) |z+ \al| ^{-4 \sigma} d\nu_P.$$ 
Then
\begin{align*}
I_1 &= \sum_{Q \in \mathcal P [1]} \int_P \sum_{\alpha \in \mathcal H(P,Q)} (f \circ h_\al)  |z+ \al| ^{-2} |z+ \al| ^{2-4 \sigma} d\nu_P \\
&= \sum_{Q \in \mathcal P [1]}  \sum_{\alpha \in \mathcal H(P,Q)}\int_{h_\al(P)}|f| |z| ^{-2+4 \sigma} d\nu_Q 
< \sum_{Q \in \mathcal P [1]}  \int_{Q}|f| |z| ^{-2+4 \sigma} d\nu_Q. \\
&<R _d^{-2+4 \sigma} \sum_{Q \in \mathcal P [1]}  \int_{Q}|f| d\nu_Q < R _d^{-2+4 \sigma}   \sup |f| \sum_{Q \in \mathcal P [1]}  \nu(Q). 
\end{align*}
However, 
\begin{align*}
I_2 &= \sum_{Q \in \mathcal P [2]} \int_P \sum_{\alpha \in \mathcal H(P,Q)} (f \circ h_\al)  |z+ \al| ^{-2} |z+ \al| ^{2-4 \sigma} d\nu_P \\
&\neq \sum_{Q \in \mathcal P [2]}  \sum_{\alpha \in \mathcal H(P,Q)}\int_{h_\al(P)}|f| |z| ^{-2+4 \sigma} d\nu_Q \end{align*}
We still want to compare $d \nu_P$ and $d \nu_Q$. Green's theorem doesn't seem to work...

Each $P \in \mathcal P$ is called a cell, and we denote by $\bar P$ the closure of $P$ in $I_d$. If $f \colon I \to \mathbb C$ is a function, let $f|_P$ be the restriction of $f$ to $P$. Define $C^1(\mathcal P)$ 
to be the space of functions $f \colon I \to \mathbb C$ such that $f|_P$ belongs to $C^1(\bar P)$ for each $P \in \mathcal P$. We have a natural isomorphism
\begin{align}
C^1(\mathcal P) \xrightarrow{\sim} \bigoplus_{P \in \mathcal P} C^1(\bar P)
\end{align}
which maps $f$ to the tuple $(f_P)_P$, where $f_P$ denotes the unique $C^1$ function which extends $f_P$. Indeed, an inverse image of a tuple $(g_P)_P$ is the unique function $f$ whose value at some point $x \in I$ is given by $g_P(x)$ where $P$ is the unique cell containing $x$. Moreover, we extend $f \in C^1(\mathcal P)$ to a function $\tilde f$ on the Riemann sphere $\hat{\mathbb C}$ by letting $\tilde f(x)=0$ whenever $x \not \in I$. If we denote by $\tilde C^1(\mathcal P)$ the space of functions on $\hat{\mathbb C}$ whose restriction to $I$ belongs to $C^1(\mathcal P)$ and vanishes outside $I$, then the map $f \mapsto \tilde f$ provides an isomorphism $C^1(\mathcal P) \xrightarrow{\sim} \tilde C^1(\mathcal P)$. Throughout, we interchangeably use $f$, $(f_P)_P$, and $\tilde f$, and identify three corresponding spaces of functions.
--------}

\hide{---------

\section{Basic properties of the domains and the partition}
\subsection{The domain $I$ and the partition $\Pc$}

Next we prove the remark after Proposition \ref{prop:2.1}.
\begin{lem}
The domain $I$ is contained in the ball of radius $R<1$ centered at the origin.
\end{lem}
\begin{proof}
\begin{enumerate} 
\item $d=1,2$ : The domain $I$ is defined by $|x| \geq 1/2, |y| \leq \sqrt{d}/2$, thus $R \leq \sqrt{1+d}/2<1.$
\item $d+3,7,11$, hexagonal domain : the maximum occurs at $x=0$ and $y= \frac{d+1}{4 \sqrt{d}}.$ Thus $R \leq \frac{11+1}{4\sqrt{11}}<1.$
\item $d=3,7,11$, rectangular domain: the domain is defined by $|x| \geq 1/2, |y| \leq \sqrt{d}/4$, thus $R \leq \frac{\sqrt{4+d}}{4}<1.$
\end{enumerate}
\end{proof}

The following theorem is so-called the finite range theorem.
\begin{lem}[Thm 1. \cite{nakada:pre}] There is a finite number of connected subsets $P_i$ of the domain $I$ such that for any admissible sequence ${\al} = (\al_1, \cdots, \al_n)$,

$$(T^n(O_{ \alpha}))^o= P_i^o$$ for some $m$.
\end{lem}
Now define the partition $\mathcal{P}$ to be the join of $\{ TO_\alpha \}_{\alpha \in O_d}$ and their rotations and reflections (more explanations here) which is a finite set by the finite range structure.

In fact, in the proof of the finite range theorem, Ei, Ito, Nakada, and Natsui proved the following property.
\begin{lem}

\end{lem}
For any $\alpha$, there exists $beta$ such that $O_\al subset T O_\beta$

\begin{lem}
Each cell is mapped into a single cell by $h_\al$.
\end{lem}
\begin{proof}
Suppose that a 1-cell $P$ is mapped to more than one 2-cell. Denote $\cup S_i \subset h_\alpha (P)$, where the union is a disjoint union of at least two sets $S_i = P_i \cap h_\alpha(P)$ for distinct 2-cells $P_i$.
Then $$T(\cup S_i) = \cup T(S_i) \subset P= T (h_\alpha(P)).$$ Since $S_i$ is a subset of $O_\alpha$, the sets $T(S_i)= T|_{O_\alpha}(S_i)$ are disjoint sets contained in distinct 2-cells of the partition, which is a contradiction to the finite range structure property proved by Nakada et. al.

Thus we conclude that a 1-cell $P$ is mapped to a unique 2-cell union some 1-cells (and 0-cells).
Since $h_\al$ is a homography, $h_\al$ cannot intersect one 2-cell without intersecting another 2-cell.
Therefore, $P$ is mapped by $h_\al$ to a union of 1-cells.

Similarly, if a 1-cell $P$ is mapped to a union of 1-cells, then it is contained in the boundary of one 2-cell. Since $h_\alpha$ is a homography, smooth curves are mapped to smooth curves. We conclude that if a 1-cell is mapped to a union of 1-cells, it is mapped to at most one 1-cell.
The 0-cells are vertices, thus each 0-cell is mapped to at most one 0-cell.
\end{proof}

\subsection{Proof of Proposition \ref{prop:compatibility}} \label{proof:compatibility}

\begin{center}
\begin{table}\caption{Empty $O_\alpha$}\label{table:empty}
\begin{tabular}[htbc]{|c|c|}
\hline
$d$ &  $\alpha$ 
\\
\hline
$1$&$\pm 1, \pm \sqrt{-1}$
\\
\hline
$2$&$\pm 1
$
\\
\hline
$3$&$\pm 1, \pm w, \pm(w-1)$
\\
\hline
$7$&$\pm 1$
\\
\hline
$11$&$\pm 1$
\\
\hline
\end{tabular}
\end{table}
\end{center}

--------------}

\bibliography{bianchi}

\bibliographystyle{abbrv}

\end{document}